\documentclass{amsart}

\pdfoutput=1

\usepackage[a4paper, margin=2.5cm]{geometry}

\usepackage{booktabs}

\usepackage[T1]{fontenc}

\usepackage{amsthm, amssymb}
\usepackage{mathtools}
\usepackage{mleftright}
\usepackage{textcomp}

\usepackage{tikz}
\usetikzlibrary{3d, decorations.markings, arrows.meta}
\tikzset{%
  ->-/.style={decoration={markings, mark=at position 0.5 with {\arrow{>}}},
              postaction={decorate}},
  ->>-/.style={decoration={markings, mark=at position 0.5 with {\arrow{>>}}},
               postaction={decorate}},
}

\usepackage{enumitem}

\usepackage{subcaption}
\usepackage{hyperref}
\hypersetup{%
    unicode=false,          
    pdftoolbar=true,        
    pdfmenubar=true,        
    pdffitwindow=false,     
    pdfstartview={FitH},    
    pdftitle={Flatness Constants},    
    pdfauthor={Codenotti, G.; Hall, T.; Hofscheier, J.},     
    pdfkeywords={Geometry of Numbers} {Lattice Polytopes} {Lattice Width} {Flatness Constants}, 
    pdfnewwindow=true,      
    colorlinks=true,       
    linkcolor=blue,          
    citecolor=red,        
    filecolor=magenta,      
    urlcolor=cyan           
}
\usepackage{cleveref}

\title[Generalised Flatness Constants: A Framework Applied in Dimension $2$]{Generalised Flatness Constants:\\A Framework Applied in Dimension $2$}
\author[G.~Codenotti]{Giulia Codenotti}
\address[G.~Codenotti]{Goethe-Universit\"at, FB 12 -- Institut f\"ur Mathematik, Postfach 11 19 32, D--60054 Frankfurt am Main, Germany}
\email{codenotti@math.uni-frankfurt.de}

\author[T.~Hall]{Thomas Hall}
\author[J.~Hofscheier]{Johannes Hofscheier}
\address[T.~Hall, J.~Hofscheier]{School of Mathematical Sciences\\University of Nottingham\\Nottingham, NG7 2RD\\UK}
\email{\{pmyth1,johannes.hofscheier\}@nottingham.ac.uk}


\DeclarePairedDelimiter{\ceil}{\lceil}{\rceil}
\DeclarePairedDelimiter{\floor}{\lfloor}{\rfloor}

\newcommand{\NN}{\mathbb{N}}
\newcommand{\ZZ}{\mathbb{Z}}
\newcommand{\QQ}{\mathbb{Q}}
\newcommand{\RR}{{\mathbb{R}}}
\newcommand{\TT}{\mathcal{T}}

\newcommand{\novel}[1]{{\emph{#1}}}

\newcommand{\closure}[1]{\overline{#1}}

\newcommand{\aff}{{\mathrm{aff}}}
\newcommand{\cone}{{\mathrm{cone}}}
\newcommand{\conv}{\mathrm{conv}}
\newcommand{\GL}{\mathrm{GL}}

\newcommand{\strint}{\mathrm{int}}
\newcommand{\relint}{\mathrm{relint}}

\newcommand{\width}{\mathrm{width}}
\newcommand{\flatness}{\mathrm{Flt}}
\newcommand{\vol}{{\mathrm{vol}}}
\newcommand{\Hom}{{\mathrm{Hom}}}

\newcommand{\lspan}{{\mathrm{span}}}

\newcommand{\set}[1]{\mleft\{{#1}\mright\}}
\newcommand{\setcond}[2]{\set{{#1}\colon{#2}}}

\theoremstyle{plain}
\newtheorem{theorem}{Theorem}[section]
\newtheorem{corollary}[theorem]{Corollary}
\newtheorem{lemma}[theorem]{Lemma}
\newtheorem{claim}{Claim}
\newtheorem{proposition}[theorem]{Proposition}

\theoremstyle{definition}
\newtheorem{example}[theorem]{Example}
\newtheorem{definition}[theorem]{Definition}
\newtheorem{remark}[theorem]{Remark}
\newtheorem{notation}[theorem]{Notation}

\newtheorem{question}[theorem]{Question}

\begin{document}

\begin{abstract}
    Let $A \in \{ \ZZ, \RR \}$ and $X \subset \RR^d$ be a bounded set.
    Affine transformations given by an automorphism of $\ZZ^d$ and a translation in $A^d$ are called (affine) $A$-unimodular transformations.
    The image of $X$ under such a transformation is called an $A$-unimodular copy of $X$.
    It was shown in~\cite{AHN} that every convex body whose width is ``big enough'' contains an $A$-unimodular copy of $X$.
    The threshold when this happens is called the generalised flatness constant $\flatness_d^A(X)$.
    It resembles the classical flatness constant if $A=\ZZ$ and $X$ is a lattice point.
    In this work, we introduce a general framework for the explicit computation of these numerical constants.
    The approach relies on the study of $A$-$X$-free convex bodies generalising lattice-free (also known as hollow) convex bodies.
    We then focus on the case that $X=P$ is a full-dimensional polytope and show that inclusion-maximal $A$-$P$-free convex bodies are polytopes.
    The study of those inclusion-maximal polytopes provide us with the means to explicitly determine generalised flatness constants.
    We apply our approach to the case $X=\Delta_2$ the standard simplex in $\RR^2$ of normalised volume $1$ and compute $\flatness^{\RR}_2(\Delta_2)=2$ and $\flatness^{\ZZ}_2(\Delta_2)=\frac{10}3$.
\end{abstract}

\maketitle


\section{Introduction}\label{sec:intro}

Let $\RR^d$ be the Euclidean space equipped with the Euclidean norm $| \cdot |$.
The space of convex bodies, i.e., non-empty compact convex sets in $\RR^d$ (note that some authors define convex bodies to be open, see, for instance,~\cite[Section 2]{Siegel}), is denoted by $\mathcal{K}^d$.
Examples of convex bodies are polytopes, convex hulls of finitely many points in $\RR^d$.
A polytope $P \subset \RR^d$ is a lattice polytope if it is the convex hull of finitely many points in the integer lattice $\ZZ^d$.
(Lattice) polytopes are studied in a variety of mathematical areas such as algebraic geometry, commutative algebra, geometry of numbers, combinatorics and statistics.

Here we study the lattice width of convex bodies motivated by questions on lattice polytopes and symplectic manifolds.
For a convex body $K \subset \RR^d$, and a linear form $u \in \Hom(\ZZ^d,\ZZ) = (\ZZ^d)^*$, the \novel{width of $K$ along $u$} (or \novel{with respect to $u$}) is given by
\[
  \width_u(K) \coloneqq \sup_{x,y\in K} |u(x) - u(y)| \text{.}
\]
The \novel{lattice width of $K$} (or simply \novel{width of $K$}), $\width(K)$, is the minimum of its widths $\width_u(K)$ along all $u \in (\ZZ^d)^*\setminus\set{0}$.
Khinchin's celebrated flatness theorem~\cite{Khinchine} guarantees that for every dimension $d$ there exists a constant which bounds the width of convex bodies which are disjoint from the integer lattice $\ZZ^d$.
This gives rise to the \novel{classical flatness constant}
\[
	\flatness_d = \sup \setcond{\width(K)}{K \in \mathcal{K}^d, K \cap \ZZ^d = \emptyset} \text{.}
\]
It is conjectured that $\flatness_d$ is roughly proportional to $d$ (\cite[last paragraph in Section 8]{Barvinok}).
To the authors' knowledge, the best known upper bound at the time of writing is $\flatness_d\le O(d^{4/3}\log^ad)$, where $a$ is a constant~\cite{FlatnessUpperBound}.
Explicit values for $\flatness_d$ for low dimensions are scarce: clearly, $\flatness_1=1$, and Hurkens has shown that $\flatness_2=1+\frac2{\sqrt{3}}$~\cite{Hurkens}.
However, already $\flatness_3$ is not known: in~\cite{Flatness3DConjecture, Flatness3DUpperBound} the bounds $2 + \sqrt{2} \le \flatness_3 \le 3.972$ are shown and it is conjectured that $\flatness_3 = 2 + \sqrt{2}$.

In~\cite{AHN}, Averkov, Hofscheier, and Nill introduced \novel{generalised flatness constants} that provide a unifying approach to several questions on lattice polytopes and symplectic manifolds.
Generalised flatness constants $\flatness_d^A(X)$ depend on the choice of a ring $A \in \set{\ZZ, \RR}$, and the choice of a fixed bounded subset $X \subset \RR^d$ and its $A$-unimodular copies.
An \novel{$A$-unimodular transformation} $T \colon \RR^d \to \RR^d$ maps an $x \in \RR^d$ to $Mx+b$, for some $M \in \GL_d(\ZZ)$ and $b \in A^d$.
We say $Y \subset \RR^d$ is an \novel{$A$-unimodular copy of $X \subset \RR^d$} if $Y = T(X)$ for some $A$-unimodular transformation $T\colon \RR^d \to \RR^d$.
Then $\flatness_d^A(X)$ is the supremum over the widths of convex bodies in $\RR^d$ that don't contain an $A$-unimodular copy of $X$:
\[
    \flatness_d^A(X) \coloneqq \sup \setcond{\width(K)}{K \in \mathcal{K}^d, K \; \text{contains no $A$-unimodular copy of $X$}} \text{.}
\]
By~\cite[Theorem 2.1]{AHN}, $\flatness_d^A(X)$ is a well-defined real number. By taking $A=\ZZ$ and $X$ a lattice point, the usual flatness constant is recovered, i.e., $\flatness_d^\ZZ(\{\mathbf{0}\}) = \flatness_d$, justifying the definition of \emph{generalised} flatness constants.
A main result of this work is the computation of generalised flatness constants in dimension $2$ when $X$ is the $2$-dimensional standard simplex $\Delta_2 = \conv( \mathbf{0}, e_1, e_2)$, where $e_1,e_2$ denotes the standard basis of $\ZZ^2$.
\begin{theorem}\label{thm:main}
    We have $\flatness_2^\RR(\Delta_2) = 2$ and $\flatness_2^\ZZ(\Delta_2) = \frac{10}3$.
\end{theorem}

A direct implication of the main theorem is that $2$-dimensional convex bodies $K \subset \RR^2$ whose width is larger than $2$ contain an $\RR$-unimodular copy of $\Delta_2$.
This bound is sharp in the sense that a convex body with lattice width $2$ contains an $\RR$-unimodular copy of $\Delta_2$ (see \Cref{prop:equal-flt-case}). 
Similarly, any convex body whose width is larger than $\frac{10}3$ contains a $\ZZ$-unimodular copy of $\Delta_2$.
Again this bound is sharp (in the same sense as above).
In particular, a convex body whose width is at least $\frac{10}3$ is spanning, i.e., the lattice points contained in the convex body affinely generate the ambient lattice.
Spanning lattice polytopes turn out to have strong Ehrhart theoretical properties equivalent to ones of IDP polytopes, a much stronger combinatorial assumption on the polytope (see~\cite{EhrhartTheorySpanning, SpanningUPP}).
The search for an effective and sufficient spanning test for lattice polytopes was one of the main motivations for the introduction of generalised flatness constants.

The proof of \Cref{thm:main} relies on the study of $A$-$X$-free convex sets: if $X \subset \RR^d$ is a fixed bounded set, then a convex set $K \subset \RR^d$ is called \novel{$A$-$X$-free} if the relative interior of $K$ contains no $A$-unimodular copy of $X$.
Here, we follow the convention that the relative interior of a point is the point itself.
In \Cref{sec:prelim}, we show that the flatness constant $\flatness_d^A(X)$ is equal to the supremum over widths of $A$-$X$-free convex bodies.
A key result in the study of generalised flatness constants is the following statement.
\begin{theorem}\label{thm_incl_max}
    If $X \subset \RR^d$ is a full-dimensional polytope, then every inclusion-maximal $A$-$X$-free convex body $K \subset \RR^d$ is a polytope.
\end{theorem}
\Cref{lem:incl-max} shows that, for any bounded $X$, to determine generalised flatness constants we can restrict our study to \emph{inclusion-maximal} $A$-$X$-free convex sets. Precisely, we show that 

\[  
    \sup \left\{ \width(K) \;\middle|\;  \parbox{3.1cm}{\small$K$ inclusion-maximal\\
        $A$-$X$-free convex body} \right\} \leq \flatness_d^A(X) \leq \sup\left\{\width(K) \;\middle|\;
        \parbox{3cm}{\small $K$ inclusion-maximal \\
        $A$-$X$-free convex set}\right\} \text{.}
\]

It is straightforward to verify that in two dimensions, $A$-$\Delta_2$-free, inclusion-maximal $2$-dimensional convex sets which are unbounded are strips with rational slopes of sufficiently small width (see \Cref{prop:unbounded}).
\Cref{thm:main} then follows from studying the inclusion-maximal $A$-$\Delta_2$-free polygons.
If $A=\RR$, a theoretical argument shows that the width of inclusion-maximal $\RR$-$\Delta_2$-free convex polygons is bounded by $2$.
We further show that there are infinite families of inclusion-maximal $\RR$-$\Delta_2$-free convex polygons (see \Cref{sec_incl_max_dim_2}).
Examples include the cross-polygon $\conv(\pm e_1,\pm e_2)$ and the triangle $\conv(e_1, e_2, -e_1 - e_2)$ which are both lattice polygons of width exactly $2$.
If $A=\ZZ$, we show that there is a unique $\ZZ$-$\Delta_2$-free polygon which maximises the width, namely the triangle $\conv(-2e_1 + 2e_2,\frac43 e_1 + \frac13 e_2,-\frac13 e_1 - \frac43 e_2)$, from which the result follows.
Note that this unique maximiser isn't a lattice polygon.

For both $A=\ZZ$ and $A=\RR$, we show that at least one of the $A$-$\Delta_2$-free polygons with width equal to $\flatness_2^A(\Delta_2)$ is a triangle. The same is true for the usual flatness constant in the plane, which is uniquely achieved at a triangle~\cite{Hurkens}.
Further, the conjectured maximiser from~\cite{Flatness3DConjecture} in three dimensions is a tetrahedron.
It is thus natural to ask the following question.

\begin{question}
    Is there always at least one $A$-$X$-free simplex among width maximisers of an $A$-generalised flatness constant? If $A=\ZZ$, are all maximisers simplices?
\end{question}

A positive answer to these questions would simplify the calculation of explicit values of the flatness constant greatly, since it would then no longer be necessary to check the width of the many inclusion-maximal $A$-$X$-free convex bodies that are not simplices.

We conclude the introduction by relating our results to the computation of the Gromov width of symplectic manifolds.
Let $(M, \omega)$ be a $2d$-dimensional symplectic manifold.
The \novel{Gromov width} of $M$ (denoted by $c_G(M)$) is the supremum over capacities $\pi r^2$ of balls with radius $r$ that can be symplectically embedded in $M$ (see~\cite{Gromov}).
We use the identification $S^1=\RR/\ZZ$ following the convention in~\cite{LMS13,Lu06}.
We are particularly interested in the case when $M$ is a symplectic toric manifold with moment polytope $\Delta$, i.e., a compact connected $2d$-dimensional symplectic manifold $(M,\omega)$ equipped with an effective Hamiltonian action of a torus $T \cong (S^1)^d$ and with a choice of a corresponding moment map $\mu \colon M \to \mathfrak{t}^*$ where $\mathfrak{t}$ denotes the lie algebra of $T$ (so from now on we assume $M$ to be toric).
The explicit computation of  Gromov width is still wide open; for example, even for symplectic toric manifolds it is not known how to read off the Gromov width from the moment polytope.
Therefore there is a huge interest in finding effective upper and lower bounds for the Gromov width~\cite{ALZ,SimplicesInNOBodies,KT05,Kav19,Lu06,LMS13,MP15,MP18,Sch05}.
In particular, the result in~\cite[Corollary~11.4]{Kav19} (see also~\cite{LMS13, Lu06, SimplicesInNOBodies}) can be restated in terms of generalised flatness constants as follows: the Gromov width of a symplectic toric manifold with moment polytope $P \subset \RR^d$ is at least $\width(P)\cdot \flatness_d^\RR(\Delta_d)^{-1}$. 
Here $\Delta_d\subset\RR^d$ denotes the $d$-dimensional standard simplex.
Combining this with \Cref{thm:main} implies a lower bound on the Gromov width of $4$-dimensional symplectic toric manifolds.

\begin{theorem}\label{thm:symplectic}
    Let $(M,\omega)$ be a $4$-dimensional symplectic toric manifold with moment polygon $\Delta$.
    Then the Gromov width $c_G(M)$ of $M$ accepts the following upper and lower bound:
    \[
        \frac{\width(\Delta)}2 \le c_G(M) \le \width(\Delta)\text{.}
    \]
\end{theorem}
The lower bound is a straightforward implication of~\cite[Corollary~11.4]{Kav19} and \Cref{thm:main}.
The upper bound was conjectured in~\cite[Conjecture 3.12]{AHN} and subsequently verified for $4$-dimensional symplectic toric manifolds by Chaidez and Wormleighton~\cite[Corollary 4.19]{CW20}.
It is known that the upper bound is tight in the sense that there exist $4$-dimensional symplectic toric manifolds whose Gromov width coincides with the lattice width of their moment polytopes (see, for instance,~\cite[Lemma~3.16]{AHN}).
There are also examples known where the Gromov width is strictly less than the width of the corresponding moment polygon (see, for instance,~\cite[Example 5.6]{HLS21}).
However, to the authors' knowledge it is not known if the lower bound of \Cref{thm:symplectic} is also tight.
That is, is there a $4$-dimensional toric symplectic manifold whose Gromov width coincides with half of the width of its moment polygon?

The paper is organised as follows.
In \Cref{sec:prelim}, we show how to reduce the calculation of flatness constants to that of inclusion-maximal $A$-$X$-free bodies, and show some key properties of these bodies.
\Cref{sec:1-2-dim} concerns the study of $\flatness_d^A(\Delta_d)$ in one dimension and begins its study in two dimensions by analysing the unbounded case.
\Cref{sec:Z_flatness_dim2} characterises inclusion-maximal $\ZZ$-$\Delta_d$-free polytopes, leading to a proof of the case $A=\ZZ$ of \Cref{thm:main}.
\Cref{sec:r_flatness_d2} proves the case $A=\RR$ of \Cref{thm:main}.

Computations were carried out using Magma \cite{Magma}, polymake \cite{polymake}, Mathematica \cite{Mathematica}, and SymPy \cite{SymPy}.
The code can be found at \href{https://github.com/jhofscheier/gen-flat-const-dim2}{https://github.com/jhofscheier/gen-flat-const-dim2}.


\subsection*{Acknowledgements}
The third author is supported by a Nottingham Research Fellowship from the University of Nottingham.


\section{A general strategy to compute generalised flatness constants}\label{sec:prelim}

In this section, we prove foundational observations on generalised flatness constants.
We hope the approach introduced here provides an efficient framework for the study of generalised flatness constants in any dimensions and for various choices of $X$.
At the end of the section we will outline a general strategy for the computation of generalised flatness constants which we will follow to determine $\flatness_2^A(\Delta_2)$ for both $A =\ZZ$ and $A = \RR$.

In our study, we will need the notion of Minkowski addition:
recall for two (arbitrary) subsets $A, B \subset \RR^d$ the \emph{Minkowski sum} (or \emph{Minkowski addition}) is defined as
\[
  A + B = \{ a + b \colon a \in A, b \in B\}\text{.}
\]
It is well-known that if $A$ and $B$ are convex, compact, or polytopes, then the Minkowski sum will have the same properties.
Furthermore, Minkowski addition is cancellative on the set of convex bodies, i.e., if $A, B, C \subset \RR^d$ are convex bodies, then $A+C = B+C$ implies $A=B$.
We will write $B^d \subset \RR^d$ for the usual $d$-dimensional ball in Euclidean space with radius $1$, i.e., the set of all points $x$ in $\RR^d$ whose Euclidean norm is bounded by $1$, i.e., $|x|\le1$.
Finally, let $B_\infty^d \subset \RR^d$ be the (closed) unit ball with respect to the maximum norm $|\cdot|_\infty$.
Note $B_\infty^d = [-1,1]^d$ is a polytope.

\begin{remark}\label{rem:basic-props}
    The following are elementary but important observations.
    Their proofs are straightforward and are left to the reader.
    Let $X \subset \RR^d$ be a bounded set.
    Then:
    \begin{itemize}
    \item
        $\flatness_d^A(X) = \flatness_d^A(\conv(X))$;
    \item
        $\flatness_d^A(X) = \flatness_d^A(\closure{X})$, where $\closure{X}$ denotes the closure of $X$ with respect to the Euclidean topology.
    \end{itemize}
\end{remark}

\subsection{\texorpdfstring{$A$-$X$-free convex bodies}{A-X-free convex bodies}}
We begin by noticing that generalised flatness constants $\flatness_d^A(X)$ can equivalently be described via $A$-$X$-free convex bodies.
\begin{lemma}\label{lem:flt_coincides_with_free}
    For a bounded set $X \subset \RR^d$, we have:
    \[
        \flatness_d^A(X) = \sup\setcond{\width(K)}{K \in \mathcal{K}^d \, \text{is an $A$-$X$-free convex body}} \text{.}
    \]
\end{lemma}
\begin{proof}
	The case $X=\emptyset$ can be straightforwardly verified, so suppose $X$ is non-empty.
	
    Since every convex body $K \subset \RR^d$ that contains no $A$-unimodular copy of $X$ is $A$-$X$-free, the inequality ``$\le$'' straightforwardly follows.
    
    For the reverse inequality, let $K \subset \RR^d$ be an $A$-$X$-free convex body that contains an $A$-unimodular copy of $X$.
    If $K$ were just a point, it would follow that $X$ is $A$-unimodularly equivalent to $K$, so that, by our convention of relative interior of points, we would get $\mathrm{relint}(K)=K$ contained an $A$-unimodular copy of $X$, i.e., $K$ is not $A$-$X$-free, a contradiction.
    Hence, the dimension of $K$ is positive.
    By~\cite[Theorem~1.8.16]{Schneider}, for any $\varepsilon>0$, there exists a polytope $P \in \mathcal{K}^d$ such that $P \subset K \subset P+\varepsilon B^d$.
    Since $\dim(K)\ge1$, it follows for sufficiently small $\varepsilon>0$ that $\dim(P)\ge1$ too.
    By moving the facets of $P$ in by another $\varepsilon'>0$ (clearly this is done inside the affine span of $P$), we obtain a polytope $P' \in \mathcal{K}^d$ whose width can be chosen to be arbitrarily close to the width of $K$ and which doesn't contain an $A$-unimodular copy of $X$.
    So, $\flatness_d^A(X) \ge \width(K)$.
    The statement follows.
\end{proof}

The following lemma seems to be decisive for the explicit computation of generalised flatness constants in that it reduces the determination of an upper bound for $\flatness_d^A(X)$ to studying the width of inclusion-maximal $A$-$X$-free closed convex sets.
\begin{lemma}\label{lem:incl-max}
    Let $X \subset \RR^d$ be a non-empty bounded subset.
    Then every $A$-$X$-free convex body is contained in an inclusion-maximal $A$-$X$-free closed convex set. 
\end{lemma}
\begin{proof}
	By \Cref{rem:basic-props}, we can assume that $X$ is closed and convex.
	In particular, $X$ is compact since it is bounded.
	Let $K \in \mathcal{K}^d$ be an $A$-$X$-free convex body.
	
	Let $\mathcal{M}$ be the set of all $A$-$X$-free closed convex sets $C \subset \RR^d$ that contain $K$ ($C$ not necessarily bounded).
	Note $\mathcal{M}$ is partially ordered with respect to inclusion.
	Our goal is to apply Zorn's Lemma.
	Therefore, we need to show that every totally ordered subset $S \subset \mathcal{M}$ has an upper bound in $\mathcal{M}$.
	We set $C_0 \coloneqq \closure{\bigcup_{C \in S} C}$ which is a closed convex set in $\RR^d$.
	It remains to verify that $C_0$ is also $A$-$X$-free.
	Assume towards a contradiction that $C_0$ contains an $A$-unimodular copy of $X$ in its relative interior, say $Y \subset \mathrm{relint}(C_0)$ where $Y$ is an $A$-unimodular copy of $X$.
	Since $\mathrm{relint}(C_0) \subset \bigcup_{C \in S}C$, it follows $Y \subset \bigcup_{C \in S}C$.
	Consider the ``distance'' between the boundary $\partial C_0$ of $C_0$ (considered as subset in the affine span $\aff(C_0)$) and $Y$:
	\[
		d(\partial C_0, Y) \coloneqq \inf \setcond{ |x - y| }{x \in \partial C_0, y \in Y} \text{.}
	\]
	Since $\partial C_0 \cap Y = \emptyset$, $\partial C_0$ is closed, and $Y$ is compact, a classical result from point set topology implies $d(\partial C_0, Y)>0$.
	There exists an $\varepsilon>0$ such that the Minkowski sum $Y + (\varepsilon B_\infty^d \cap \aff(C_0))$ is contained in the interior $\mathrm{int}(C_0)$.
	Then clearly
	\[
    	Y \subset Y+\mleft(\frac\varepsilon2 B_\infty^d\cap\aff(C_0)\mright) \subset Y+\mleft(\varepsilon B_\infty^d\cap\aff(C_0)\mright) \subset \mathrm{relint}(C_0) \subset \bigcup_{C \in S} C \text{.}
	\]
	Since $Y+(\frac\varepsilon2 B^d\cap\aff(C_0))$ is compact, finitely many translations of $\varepsilon B^d_\infty\cap\aff(C_0)$ suffice to cover $Y+(\varepsilon B^d\cap\aff(C_0))$.
	Note the convex hull of these finitely many translates yield a polytope with vertices, say $v_1, \ldots, v_n$.
	There exist $C_i \in S$ with $v_i \in C_i$.
	Since $S$ is totally ordered, we conclude that there exists a $C \in S$ with $v_1, \ldots, v_n \in C$, and thus $Y+(\frac\varepsilon2B^d\cap \aff(C_0)) \subset C$.
	Hence, $Y \subset \mathrm{relint}(C)$.
	A contradiction.
	
	By construction $C_0 \in \mathcal{M}$ is an upper bound of $S$.
	The statement follows by Zorn's Lemma.
\end{proof}
\begin{remark}
    The inclusion-maximal set from \Cref{lem:incl-max} might be unbounded: consider for example the rectangle with vertices $(\pm a, 1), (\pm a, 0)$, for any large $a \in \RR$. This is a $\RR$-$\Delta_2$-free convex body and the only inclusion-maximal $\RR$-$\Delta_2$-free convex set containing it is the horizontal strip between height $0$ and $1$.
\end{remark}

By \Cref{lem:incl-max}, every $A$-$X$-free convex body is contained in an inclusion-maximal closed convex $A$-$X$-free set.
Since the width is monotone with respect to inclusion, we have the following.
\begin{equation}\label{ineq:upper_lower_bounds}
        \sup \left\{ \width(K) \;\middle|\;  \parbox{3.1cm}{\small$K$ inclusion-maximal\\
        $A$-$X$-free convex body} \right\} \leq \flatness_d^A(X) \leq \sup\left\{\width(K) \;\middle|\;
        \parbox{3cm}{\small $K$ inclusion-maximal \\
        $A$-$X$-free convex set}\right\} \text{.}
\end{equation}

That is, an upper bound on the width of inclusion-maximal  $A$-$X$-free closed convex sets $C \subset \RR^d$ (including the unbounded ones) gives an upper bound for $\flatness_d^A(X)$, while the width of any inclusion-maximal bounded $A$-$X$-free convex set yields a lower bound. 
A strategy to determine the exact value of generalised flatness constants is to compute these upper and lower bounds and show they agree by studying the explicit values of the width in the two subcases: 1) $C$ is unbounded; 2) $C$ is a convex body.

In this work and with regard to the applications of generalised flatness constants to symplectic geometry, the case when the convex hull $\conv(X)$ is full-dimensional plays a crucial role.
In this case, it suffices to consider full-dimensional $A$-$X$-free convex bodies, as the next lemma shows.

\begin{lemma}\label{lem:full-dim}
    Let $X \subset \RR^d$ be a bounded subset whose convex hull is full-dimensional.
    Then every $A$-$X$-free convex body is contained in a full-dimensional $A$-$X$-free convex body.
\end{lemma}
\begin{proof}
    By \Cref{rem:basic-props}, we may assume that $X$ is a full-dimensional convex body.
    Let $K \subset \RR^d$ be an $A$-$X$-free convex body of dimension $<d$.
    Then $K$ is contained in an affine hyperplane $H \subset \RR^d$.
    Let $U_H \subset \RR^d$ be the unique linear subspace parallel to $H$.
    We first show that it suffices to consider the case $\dim(K) = d-1$.
    
    Suppose $\dim(K) < d-1$.
    Let $B^{d-1} \subset U_H$ be the $d-1$-dimensional unit ball containing the origin.
    Then the Minkowski sum $K+B^{d-1}$ has dimension $d-1$ and is contained in $H$.
    Since $X$ is full-dimensional, $K+B^{d-1}$ is $A$-$X$-free (it cannot contain an $A$-unimodular copy of $X$).
    
    Hence, we may assume that $K$ is $d-1$-dimensional contained in an affine hyperplane $H \subset \RR^d$.
    Clearly, there exists a parallelepiped $\Pi^{d-1} \subset H$ which contains $K$.
    Let $v \in \RR^d$ such that $\RR v + U_H = \RR^d$.
    Take $\varepsilon >0$ such that the volume of the parallelepiped $\Pi^{d-1}+\varepsilon v$ is strictly smaller than the volume of $X$.
    Then $K \subset \Pi^{d-1}+\varepsilon v$ is full-dimensional and $A$-$X$-free (it cannot contain an $A$-unimodular copy of $X$).
\end{proof}

This guarantees that, whenever we work with a set $X$ whose convex hull is full dimensional, the supremums on the left and right hand side of~\ref{ineq:upper_lower_bounds} can be taken over just the \emph{full-dimensional} sets. 
A positive answer to the following question would confirm our suspicion that, once we restrict to full-dimensional sets, the width of the unbounded  sets will always be strictly less than the widths of the bounded ones and thus that the inequalities in~\ref{ineq:upper_lower_bounds} above are in fact equalities.

\begin{question}
Do full-dimensional unbounded inclusion-maximal $A$-$X$-free sets always have ``small'' width, that is, they are never maximisers of width among all $A$-$X$-free bodies?
\end{question}

\begin{remark}\label{rem:full-dim}
    Let us provide more details why the restriction to full-dimensional inclusion-maximal $A$-$X$-free closed convex sets $C \subset \RR^d$ is crucial.
    Consider the case that $X=\Delta_2$.
    Let $m \in \RR\setminus \QQ$ be an irrational real number.
    Then, for any closed interval $I \subset \RR$, we have an  $A$-$\Delta_2$-free convex body $\setcond{(x,m \cdot x)}{x \in I} \subset \RR^2$.
    These convex bodies are contained in the inclusion-maximal $A$-$\Delta_2$-free closed convex set $C = \setcond{(x, m \cdot x)}{x \in \RR} \subset \RR^2$ (note $C$ is inclusion-maximal $A$-$X$-free as otherwise it would contain a point outside the line $\{ y = m \cdot x \}$, and thus would contain a $2$-dimensional strip with irrational slope; now apply the argument from the proof of \Cref{prop:unbounded}).
    
    Since $m$ is irrational, any $u \in (\ZZ^2)^* \setminus \{ 0 \}$ induces a non-trivial linear form on the line $C = \{ y = m \cdot x \}$, and thus $\width_u(C) = \infty$.
    It follows that $\width(C) = \infty$, yielding the rather unhelpful upper bound $\flatness_2^A(\Delta_2) \le \infty$.
    However, we shall see that considering full-dimensional inclusion-maximal $A$-$\Delta_2$-free closed convex sets will provide exactly the bounds that we need to compute $\flatness_2^A(\Delta_2)$.
\end{remark}

In $2$ dimensions the unbounded case is straightforward and will be done in \Cref{sec:1-2-dim}.
The study of inclusion-maximal $A$-$X$-free convex bodies is more intricate and the later sections of the paper will be concerned with this investigation in $2$ dimensions.

The following lemma is another important ingredient in the explicit computation of generalised flatness constants.
Its significance lies in \Cref{cor:incl-max-contain-X}, which shows that an inclusion-maximal $A$-$X$-free convex body contains an $A$-unimodular copy of $X$.
That will provide us with tools necessary to study those bodies yielding the exact numerical value of the respective generalised flatness constant.

\begin{lemma}\label{lem:enlarge}
  Let $K, X \subseteq \RR^d$ be $d$-dimensional convex bodies.
  If $K$ does not contain an $A$-unimodular copy of $X$, there exists $\varepsilon>0$ such that $K+\varepsilon B^d$ does not contain an $A$-unimodular copy of $X$.
\end{lemma}
\begin{proof}
    We split the proof into two cases.
	Suppose $A=\RR$.
    We show the contrapositive, i.e., if for all (large enough) $n \in \NN$, there exists an $\RR$-unimodular transformation $T_n$ with $T_n(X) \subseteq K+\frac1n B^d$, then $K$ contains an $\RR$-unimodular copy of $X$.
    We can represent $T_n$ as a composition of a $\ZZ$-unimodular transformation $A_n$ and a translation $\delta_n \in [0,1]^d$, i.e. $T(x) = A_n (x) - \delta_n$.
    Since $T_n(X) = A_n(X) - \delta_n \subseteq K+\frac1n B^d$, we get
    \[
        A_n(X) \subseteq K + \frac1n B^d + \delta_n \subseteq K + B^d + [0,1]^d \text{.}
    \]
    Since $X$ is full-dimensional, there exists an affine $\ZZ$-basis $\mathcal{B} = \set{b_0, \ldots, b_d}$ of $\ZZ^d$ and a scaling factor $\eta > 0$ such that $\eta \mathcal{B} \subseteq X$.
    As a $\ZZ$-unimodular transformation, $A_n$ is uniquely determined by the images of the $b_i$.
    We require that $A_n$ maps $\eta\mathcal{B}$ onto an affine linearly independent subset of $\eta \ZZ^d \cap \left( K + B^d + [0,1]^d \right)$.
    Since $K + B^d + [0,1]^d$ is bounded and $\eta \ZZ^d$ is discrete, there are only finitely many choices for $A_n$.
    In particular, by restricting to an appropriate subsequence, we may assume that $A = A_n$ for all $n \in \NN$.
    
    Since $[0,1]^d$ is compact, there exists a convergent subsequence of $(\delta_n)_{n \in \NN}$ with limit $\delta \in [0,1]^d$.
    To keep notation simple, we use the same symbol $(\delta_n)_{n \in \NN}$ for this subsequence, i.e. $\delta_n$ converges to $\delta \in [0,1]^d$.
    Let $T$ be the $\RR$-unimodular transformation obtained by composing the $\ZZ$-unimodular transformation $A$ with the translation $\delta$.
    We show that $T(X) \subseteq K$.
    
    Suppose $x \in X$.
    Since $A(x) - \delta_n \in K + \frac1n B^d$, there exists $z_n \in \frac1n B^d$ such that
    \[
        y_n \coloneqq A(x) -\delta_n - z_n \in K \text{.}
    \]
    As $n$ goes to infinity, we get that $y_n \rightarrow A(x) - \delta = T(x)$.
    Since $K$ is compact, it follows that $T(x) \in K$.
    
    If $A=\ZZ$, we repeat the argument above, however now there are no translations $\delta_n \in [0,1]^d$.
\end{proof}

The proof of \Cref{lem:enlarge} shows the following result which will be needed below.

\begin{corollary}\label{cor_enlarge_translation}
  Let $K, X \subseteq \RR^d$ be $d$-dimensional convex bodies.
  If $K$ does not contain an $A$-translate of $X$, then there exists an $\varepsilon>0$ such that $K+\varepsilon B^d$ does not contain an $A$-translate of $X$.
\end{corollary}

Furthermore, the following result follows from \Cref{lem:enlarge}.
\begin{corollary}\label{cor:incl-max-contain-X}
    Let $K,X \subset \RR^d$ be $d$-dimensional convex bodies.
    If $K$ is inclusion-maximal $A$-$X$-free, then $K$ contains an $A$-unimodular copy of $X$.
\end{corollary}
\begin{proof}
    Assume towards a contradiction $K$ didn't contain an $A$-unimodular copy of $X$.
    By \Cref{lem:enlarge}, $K$ is contained in a strictly bigger convex body $K \subsetneq K' \subset \RR^d$ with the same property.
    Since $K'$ doesn't contain an $A$-unimodular copy of $X$, it is $A$-$X$-free.
    A contradiction to $K$ being inclusion-maximal.
\end{proof}

By the definition of generalised flatness constants, given a convex body $K \subset \RR^d$ with $\width(K)>\flatness_d^A(X)$, there exists an $A$-unimodular copy of $X$ that is contained in $K$.
However, a priori the case of equality $\width(K) = \flatness_d^A(X)$ is unclear, i.e., there could be such $K$ that do not contain an $A$-unimodular copy or there could be such $K$ that do.
\Cref{lem:enlarge} implies the following complete answer to this question.
\begin{proposition}\label{prop:equal-flt-case}
    Let $K, X \subset \RR^d$ be $d$-dimensional convex bodies with $\width(K) = \flatness_d^A(X)$.
    Then $K$ contains an $A$-unimodular copy of $X$.
\end{proposition}
\begin{proof}    
    Assume towards a contradiction $K$ doesn't contain an $A$-unimodular copy of $X$.
    By \Cref{lem:enlarge}, there exists $\varepsilon>0$ such that $K+\varepsilon B^d$ also doesn't contain an $A$-unimodular copy of $X$.
    A contradiction since $\width(K+\varepsilon B^d)>\width(K)$.
\end{proof}

We conclude the section by proving \Cref{thm_incl_max} for the case $A=\ZZ$.
Indeed, we show a more general version of \Cref{thm_incl_max} in that $X$ just needs to be a bounded set.
The case $A=\RR$ is more involved and will be treated in Sections~\ref{sec_incl_max_conv_body} and~\ref{sec_intersec_conv_bodies}.

\begin{proposition}[Case $A=\ZZ$ of \Cref{thm_incl_max}]\label{prop:A_Z_X_poly_K_poly}
    Let $X \subset \RR^d$ be a bounded set.
    Then every inclusion-maximal $\ZZ$-$X$-free convex body is a polytope.
\end{proposition}
\begin{proof}
    Let $K \subset \RR^d$ be an inclusion-maximal $\ZZ$-$X$-free convex body.
    There is a positive integer $N \in \ZZ_{>0}$ such that $K$ is contained in $[-N,N]^d$. Set
    \[
        A \coloneqq \setcond{a \in [-N,N]^d \cap \ZZ^d}{a \not\in \mathrm{relint}(K)} \text{,}
    \]
    and
    \[
        \Delta \coloneqq \bigcap_{a\in A} \set{H_a \, \text{half-spaces such that} \, a \in\partial H_a,\, K \subset H_a} \text{.}
    \]
    Such choices of half-spaces exist due to the Separation Theorem.
    Since $\Delta$ is a finite intersection of half-spaces (note $|A|<\infty$), it's a polytope.
    Notice all lattice points of $\Delta$ are either in $\mathrm{relint}(K)$ or in $\partial\Delta$.
    Thus $\Delta$ is $\ZZ$-$X$-free.
    Since $K \subset \Delta$, we have that $K = \Delta$ is a polytope.
\end{proof}


\subsection{\texorpdfstring{Inclusion-maximal $\RR$-$X$-free convex bodies}{Inclusion-maximal RR-X-free convex bodies}}\label{sec_incl_max_conv_body}

In this section we are going to prove \Cref{thm_incl_max} for the case $A=\RR$.

\begin{proposition}[Case $A=\RR$ of \Cref{thm_incl_max}]\label{prop_R_incl_max}
    Let $X \subset \RR^d$ be a polytope.
    Then every inclusion-maximal $\RR$-$X$-free convex body is a polytope.
\end{proposition}
Throughout $X \subset \RR^d$ will be a fixed full-dimensional polytope and $K \subset \RR^d$ a convex body.
Notice that, by~\Cref{lem:full-dim}, we may assume that $K$ is full-dimensional (a crucial assumption as explained in~\Cref{rem:full-dim}).
Then in the definition of $\RR$-$X$-free ``relative interior'' can be replaced by just ``interior''.

We start by investigating the set of $\RR$-unimodular copies $S$ of $X$ which are contained in $K$.
In general, this is an infinite set (see \Cref{fig_copies_not_finite}).
\begin{figure}[!ht]
    \begin{tikzpicture}
        \draw[-latex] (-.5,0) -- (2,0);
        \draw[-latex] (0,-.5) -- (0,1.5);
        \fill[fill=gray!50,opacity=.5,very thin,dashed] (.25,-.25) -- (1.25,.75) -- (.25,.75) -- cycle;
        \draw[very thin,dashed] (.25,-.25) -- (1.25,.75) -- (.25,.75) -- cycle;
        \draw[very thin,dashed] (0,0) -- (1,1) -- (0,1) -- cycle;
        \draw[very thin,dashed] (.5,-.5) -- (1.5,.5) -- (.5,.5) -- cycle;
        \draw[-latex,thin] (0,1) -- (.5,.5);
        \draw[thick] (0,1) -- (0,0) -- (.5,-.5) -- (1.5,.5) -- (1,1) -- cycle;
    \end{tikzpicture}
    \caption{A polygon containing $\RR$-unimodular copies of $\Delta_2$ translated along a line segment.\label{fig_copies_not_finite}
}
\end{figure}
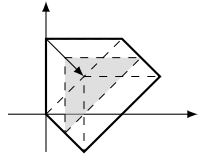
We will identify two polytopes in this set if one is the translation of the other.
The resulting set of equivalence classes will be denoted by $\TT_K(X)$.
\begin{lemma}\label{lem_TT_K_X_sim_finite}
    The set $\TT_K(X)$ of $\RR$-translation equivalence classes of $\RR$-unimodular copies of $X$ contained in $K$ is finite.
\end{lemma}
\begin{proof}
    Clearly for every $\RR$-unimodular copy $S$ of $X$ which is contained in $K$, there exists a $\ZZ$-unimodular copy $S'$ of $X$ such that $S' \subset K+[-1,0]^d$.
    Since $K+[-1,0]^d$ contains only finitely many lattice points, it follows that there can only be finitely many such $S'$ contained in $K+[-1,0]^d$.
    The statement follows.
\end{proof}

In what follows, we will use a complementary operation to Minkowski addition usually referred to as Minkowski difference.
We recall its definition and refer to~\cite{Schneider} for further details and references.
For subsets $A, B \subset \RR^d$ the \emph{Minkowski difference} $A\text{\textdiv}B$ is the set of translation vectors that move $B$ into $A$:
\[
    A\text{\textdiv} B = \{ x \in \RR^d \colon B+x\subset A \} = \bigcap_{b\in B} (A-b) \text{.}
\]
From the last equality, it is clear that the Minkowski difference $A\text{\textdiv}B$ of two convex bodies $A, B$ in $\RR^d$ is either empty or a  are convex body too.
Here is a justification why Minkowski difference can be regarded as a complementary operation to Minkowski addition.
\begin{lemma}[{\cite[Lemma 3.1.11]{Schneider}}]
    Let $A, B \subset \RR^d$ be two convex bodies.
    Then
    \[
        (A+B)\text{\textdiv}B = A \text{.}
    \]
    Furthermore, we have
    \[
        (A\text{\textdiv}B)+B=A
    \]
    if and only if there exists a convex body $C\subset \RR^d$ such that $A=B+C$.
\end{lemma}
From now on, let $S \subset \RR^d$ be a fixed $\ZZ$-unimodular copy of $X$.
We note that for any $\delta\in\RR^d$, we have
\[
    K\text{\textdiv}(S+\delta) = (K\text{\textdiv}S)-\delta\text{.}
\]
Furthermore, $K\text{\textdiv}S$ is not empty if and only if $K$ contains a translate of $S$.
Studying the Minkowski difference $K\text{\textdiv}S$ (also called \emph{inner parallel body relative to $S$}) will be key to proving \Cref{prop_R_incl_max}.
\begin{figure}[!ht]
    \begin{tikzpicture}
        \begin{scope}
            \clip (-1,0) circle (1cm);
            \clip (0,-1) circle (1cm);
            \clip (0,0) circle (1cm);
            \fill[fill=gray!50,opacity=.5] (0,0) circle (1cm);
            \draw[very thin] (-1,0) circle (1cm);
            \draw[very thin] (0,-1) circle (1cm);
        \end{scope}
        \draw (0,0) circle (1cm);
        \fill[fill=gray!35,opacity=.25] (0,0) -- (1,0) -- (0,1) -- cycle;
        \draw[very thin] (0,0) -- (1,0) -- (0,1) -- cycle;
        \node at (.25,.25) {$S$};
        \draw (-1.5,-.5) node[left] {$K\text{\textdiv}S$} edge[-latex,bend right] (-.5,-.5);
        \foreach \x in {-2,-1,0,1,2} \foreach \y in {-1,0,1} \fill (\x,\y) circle (.75pt);
        \fill (0,0) circle (1pt) node[below right] {$0$};
    \end{tikzpicture}
    \caption{$K\text{\textdiv}S$ for the unit disc $K$ with centre at the origin and $S = \conv(\mathbf{0}, e_1, e_2) \subset \RR^2$.\label{fig_Tran_disc_triang}
}
\end{figure}
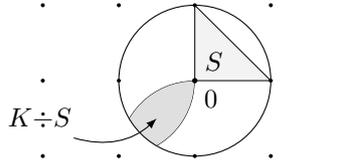
In general, $K\text{\textdiv}X$ will not be a polytope (see for instance \Cref{fig_Tran_disc_triang}).
However, since $X$ is a polytope there is an efficient way to determine the Minkowski difference $K\text{\textdiv}X$:

\begin{lemma}\label{lem_translations}
    Suppose $K\text{\textdiv}S\neq\emptyset$, i.e., there exists $\delta \in \RR^d$ such that $S + \delta \subset K$.
    Let $v_1, \ldots, v_n$ be the vertices of $S+\delta$.
    Then
    \[
		K\text{\textdiv}(S + \delta) = \bigcap_{i=1}^n (K-v_i) \text{.}
    \]
\end{lemma}
\begin{proof}
  Let $v \in K\text{\textdiv}(S+\delta)$.
  Then for all $i$, we have $v_i + v \in K$, and thus $v \in K-v_i$.
  
  Suppose $v \in K-v_i$, i.e. $v_i + v \in K$, for all $i$.
  Since $K$ is convex, it follows
  \[
    \conv(v_1 + v, \ldots, v_n + v) = \conv(v_1, \ldots, v_n) + v = S +\delta + v \subset K \text{.}
  \]
  Thus $v \in K\text{\textdiv}(S+\delta)$ and the statement follows.
\end{proof}

The collection of inner parallel bodies $K\text{\textdiv}S$ for $\ZZ$-unimodular copies $S$ of $X$ allows to characterise when exactly $K$ is $\RR$-$X$-free (see the next two statements).

\begin{lemma}\label{lem:int-full-dim}
	An $\RR$-translation of a $\ZZ$-unimodular copy $S$ of $X$ is contained in the interior of $K$ if and only if $\dim(K\text{\textdiv}S) = d$.
\end{lemma}
\begin{proof}
  Suppose there is a $\ZZ$-unimodular copy $S$ of $X$ and $\delta \in \RR^d$ such that $S + \delta$ is contained in the interior of $K$.
  Since $(K\text{\textdiv}(S+\delta)) + \delta = K\text{\textdiv}S$, we may replace $S$ by its translate $S+\delta$, so that $S \subset K$.
  We need to show that $\dim(K\text{\textdiv}S) = d$.
  Consider small open neighbourhoods $B_1, \ldots, B_n$ around the vertices $v_1, \ldots, v_n$ of $S$.
  Choose $B_i$ small enough such that it is contained in the interior of $K$.
  Move these neighbourhoods to the origin, i.e., consider $B_i-v_i$.
  The intersection $\bigcap_{i=1}^n(B_i - v_i)$ is an open neighbourhood of the origin which is contained in $K\text{\textdiv}S$, and thus $\dim(K\text{\textdiv}S)=d$.
  
  For the reverse direction, suppose $S$ is a $\ZZ$-unimodular copy of $X$ with $\dim(K\text{\textdiv}S) = d$.
  Like before, we may replace $S$ by $S+\delta$ for some $\delta \in \RR^d$ such that (after replacing) $S \subset K$.
  Denote the vertices of $S$ by $v_1, \ldots, v_n$.
  We claim there is an element $w$ in $K\text{\textdiv}S$ which is not contained in the translates $\partial K - v_i$ of the boundary of $K$ for all $i=1, \ldots, n$.
  In other words, we claim that there is
  \[
    w \in (K\text{\textdiv}S) \setminus \left( \bigcup_{i=1}^n \partial K - v_i \right) = \left( \bigcap_{i=1}^n (K-v_i) \right) \setminus \left( \bigcup_{i=1}^n \partial K - v_i \right)
  \]
  where the equality follows by \Cref{lem_translations}.
  Before proving the claim, let us see how it implies the statement:
  $w + v_i \in K \setminus \partial K$, and thus $\conv(v_1, \ldots, v_n) + w \subset K \setminus \partial K$, i.e. $S + w$ is contained in the interior of $K$.
  
  It remains to show the claim.
  Suppose by contradiction $(K\text{\textdiv}S) \setminus \left( \bigcup_{i=1}^n \partial K - v_i \right)$ is empty, i.e. $K\text{\textdiv}S$ is contained in the union of the shifted boundaries $\partial K - v_i$ for $i = 1, \ldots, n$.
  By our assumption $\dim(K\text{\textdiv}S)=d$, there are linearly independent vectors $w_1, \ldots, w_d \in \RR^d$ such that the $d$-dimensional simplex $\Delta \coloneqq \conv(\mathbf{0}, w_1, \ldots, w_d)$ is contained in $K\text{\textdiv}S$.
  We obtain a contradiction if we apply the $d$-dimensional Euclidean volume to the inclusion $\Delta \subset \bigcup_{i=1}^n (\partial K - v_i)$:
  \[
    0 < \vol(\conv(v_1, \ldots, v_n)) \le \vol\left( \bigcup_{i=1}^n \partial K - v_i\right) \le \sum_{i=1}^n \vol(\partial K - v_i) = 0 \text{.} \qedhere
  \]
\end{proof}

\begin{corollary}\label{lem_R_X_free_with_Tran}
    $K$ is $\RR$-$X$-free if and only if $\dim(K\text{\textdiv}S) < d$ for all $\ZZ$-unimodular copies $S$ of $X$.
\end{corollary}
\begin{proof}
	  The contrapositive, i.e., $K \subset \RR^d$ is not $\RR$-$X$-free if and only if there is some affine $\ZZ$-unimodular copy $S$ of $X$ with $\dim(K\text{\textdiv}S) = d$, straightforwardly follows from \Cref{lem:int-full-dim}.
\end{proof}

Fix an inclusion-maximal $\RR$-$X$-free convex body $K \subset \RR^d$.
Recall that we want to show that $K$ is a polytope.
As a first step, we approximated $K$ by a polytope $P \subset \RR^d$ which contains $K$ such that up to real translations $K$ and $P$ contain the same $\ZZ$-unimodular copies of $X$.
\begin{lemma}\label{lem_1st_approx}
  For a $d$-dimensional convex body $K \subset \RR^d$, there exists a $d$-dimensional convex polytope $P \subset \RR^d$ which contains $K$ such that $\TT_K(X) = \TT_P(X)$.
\end{lemma}
\Cref{lem_1st_approx} follows from the following well-known result (see also~\cite[Theorem~1.8.16]{Schneider}).
\begin{lemma}\label{lem_polytopal_approx}
  If $K \subset \RR^d$ is a $d$-dimensional convex body and $\varepsilon>0$, then there exists a polytope $P \subset \RR^d$ such that $K \subset P \subset K + \varepsilon B^d$.
\end{lemma}
\begin{proof}
Recall the (closed) unit ball $B_\infty^d$ with respect to the infinity norm $| \cdot |_\infty$ is a polytope, namely $B^d_\infty = [-1,1]^d$.
  Since any two norms on $\RR^d$ are equivalent, there exists $\alpha>0$ such that $\alpha B_\infty^d \subset \varepsilon B^d$.
  
  Since the collection $\setcond{ x + \alpha B_\infty^d}{x \in K}$ covers the compact set $K$, already a finite subset of balls suffices to cover $K$.
  The polytope $P$ obtained by taking the convex hull of these finitely many $\infty$-balls with radius $\alpha$ satisfies the statement, i.e., $K \subset P \subset K+\varepsilon B^d$.
\end{proof}
\begin{proof}[Proof of \Cref{lem_1st_approx}]
  Choose $N>0$ large enough such that $K$ is contained in the interior of $[-N,N]^d$.
  Then $\TT_K(X) \subset \TT_{[-N,N]^d}(X)$ and this inclusion is strict in general.
  Our strategy is to cut $[-N,N]^d$ with further half-spaces so that the resulting polytope $P$ satisfies $\TT_K(X) = \TT_P(X)$.
  
  Consider a $\ZZ$-unimodular copy $Q$ of $X$ such that $[-N,N]^d$ contains an $\RR$-translate of $Q$ but $K$ doesn't, i.e. $[Q] \in \TT_{[-N,N]^d}(X) \setminus \TT_K(X)$.
  By \Cref{cor_enlarge_translation}, there exists $\varepsilon_Q > 0$ such that $K + \varepsilon_Q B^d$ does not contain an $\RR$-translate of $Q$.
  Let $\varepsilon>0$ be the minimum of the $\varepsilon_Q$ over all $Q$ with $[Q] \in \TT_{[-N,N]^d}(X) \setminus \TT_K(X)$ (note $\TT_{[-N,N]^d}(X)$ is finite by \Cref{lem_TT_K_X_sim_finite}).
  If necessary reduce $\varepsilon>0$ such that $K + \varepsilon B^d \subset [-N,N]^d$.
  By \Cref{lem_polytopal_approx}, there exists a polytope $P \subset \RR^d$ such that $K \subset P \subset K + \varepsilon B^d$.
  By our choice of $\varepsilon$, we have that $\TT_{K + \varepsilon B^d}(X) = \TT_K(X)$, and by the monotonicity of $\TT_K(X)$ with respect to $K$, this polytope satisfies $\TT_K(X) = \TT_P(X)$.
\end{proof}

\begin{remark}
  In \Cref{lem_1st_approx}, the polytope $P$, which contains $K$,  isn't necessarily $\RR$-$X$-free in general (see \Cref{fig_enlarge_not_R_X}).
  \begin{figure}[!ht]
    \begin{tikzpicture}
      \fill[fill=gray!25,opacity=.5] (-.1,-.1) -- (-.1,1.15) -- (.82,.82) -- (1.15,-.1) -- cycle;
      \draw (-.1,-.1) -- (-.1,1.15) -- (.82,.82) -- (1.15,-.1) -- cycle;
      
      \draw[very thin,dashed] (-1.4,1) -- (1.4,1);
      \draw[very thin,dashed] (1,-1.4) -- (1,1.4);
      \draw[very thin,dashed] (-1.4,1.4) -- (1.4,-1.4);
    
      \draw[very thin] (0,0) ++(0:1) arc (0:90:1);
      \node[right] at (1.3,0) {$\scriptstyle K + \varepsilon B^2$};
      \draw (0,0) -- (0,1);
      \draw (0,0) -- (1,0);

      \draw[very thin] (0,0) ++(0:1.25) arc (0:90:1.25);
      \draw[very thin] (0,1) ++(90:.25) arc (90:180:.25);
      \draw[very thin] (1,0) ++(0:.25) arc (0:-90:.25);
      \draw[very thin] (0,0) ++ (180:.25) arc (180:270:.25);
      \draw[very thin] (-.25,0) -- (-.25,1);
      \draw[very thin] (0,-.25) -- (1,-.25);
      
      \fill[fill=gray!10,opacity=.5] (-1,-1) -- (0,-1) -- (-1,0) -- cycle;
      \draw[very thin] (-1,-1) -- (0,-1) -- (-1,0) -- cycle;
      
      \draw (-2.25,-.75) node[ above] {$\scriptstyle X+\begin{psmallmatrix}-1\\-1\end{psmallmatrix}$} edge[-latex,bend right] (-.75,-.75);
      \draw (-1.35,-1.35) node [left] {$\scriptstyle[-N,N]^2$} -- (1.35,-1.35) -- (1.35,1.35) -- (-1.35,1.35) -- cycle;
      \foreach \x in {-1,0,1} \foreach \y in {-1,0,1} \fill (\x,\y) circle (.75pt);
      \fill (0,0) circle (1pt) node[above right] {$\scriptstyle0$};
    \end{tikzpicture}
    \caption{Illustration of \Cref{lem_1st_approx}. Here, $K$ is a quarter of a disc and $X=\conv(\mathbf{0}, e_1, e_2) \subset \RR^2$ the standard simplex. The polytope $P$ might not be $\RR$-X-free.\label{fig_enlarge_not_R_X}
}
  \end{figure}
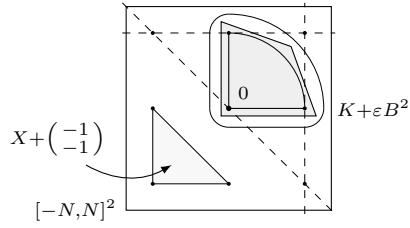
  By intersecting $P$ with further half-spaces (e.g., the half-spaces indicated by the dashed lines), we get a new polytope which is $\RR$-$X$-free.
  The proof of \Cref{prop_R_incl_max} will show this is always true.
\end{remark}

The proof of \Cref{prop_R_incl_max} will make use of the following theorem, whose proof we postpone until \Cref{sec_intersec_conv_bodies}.

\begin{theorem}\label{thm_finite_supp_hyper}
    Let $K_1, \ldots, K_n \subset \RR^d$ be convex bodies such that $\dim(K) < d$ where $K \coloneqq K_1 \cap \ldots \cap K_n$.
    Suppose $x \in K$ such that $x$ is contained in the boundary of every $K_i$ for $i=1, \ldots, n$.
    Then there are (nonempty) finite collections of closed half-spaces $\{ H^+(i,j) \}_{j \in J_i}$ for $i = 1, \ldots, n$ such that the boundary of $H^+(i,j)$ is a supporting hyperplane to $K_i$ at $x$, and
    \[
        \dim \left( \bigcap_{j \in J_1} H^+(1,j) \cap \ldots \cap \bigcap_{j \in J_n} H^+(n,j) \right) < d \text{.}
    \]
\end{theorem}

\begin{proof}[Proof of \Cref{prop_R_incl_max}]
  The idea is to construct an $\RR$-$X$-free polytope $P$ which contains $K$.
  Since $K$ is inclusion-maximal with respect to this property, it follows that $K = P$ is a polytope.
  \Cref{lem_1st_approx} yields a first approximation $P'$ of this polytope $P$ which satisfies $\TT_K(X) = \TT_{P'}(X)$.
  However, $P'$ might not be $\RR$-$X$-free in general (see \Cref{fig_enlarge_not_R_X}).
  We claim that $P'$ can be cut with further half-spaces such that the resulting polytope still contains $K$ and is $\RR$-$X$-free.
  We prove this claim by induction on the number $r$ of equivalence classes $[Q] \in \TT_{P'}(X)$ with $\dim(K\text{\textdiv}Q) = d$.
  By \Cref{lem_TT_K_X_sim_finite}, the number $r$ is finite.
  
  In the base case, i.e. $r=0$, we have $\dim(P\text{\textdiv}Q) < d$ for all $[Q] \in \TT_P(X)$, and thus by \Cref{lem_R_X_free_with_Tran}, the polytope $P'$ is already $\RR$-$X$-free.

  Consider the case $r \ge 1$.
  Let $Q$ be a $\ZZ$-unimodular copy of $X$ with $\dim(P'\text{\textdiv}Q) = d$.
  Note that by construction $[Q] \in \TT_K(X)$.
  We want to intersect $P'$ with finitely many half-spaces, so that the resulting polytope $P''$ contains $K$ and $\dim(P''\text{\textdiv}Q) < d$.
  By \Cref{lem_R_X_free_with_Tran}, $\dim(K\text{\textdiv}Q) < d$.
  Let $\delta \in K\text{\textdiv}Q$.
  If $v_1, \ldots, v_n$ are the vertices of $Q + \delta$, then, by \Cref{lem_translations}, $K\text{\textdiv}(Q+\delta) = (K - v_1) \cap \ldots \cap (K - v_n)$.
  If the origin $\mathbf{0}$ is in the interior of $K - v_i$, then this convex body does not contribute to the dimension drop and we omit it.
  By abuse of notation, let us assume that the origin $\mathbf{0}$ is contained in the boundary of $K - v_i$ for $i = 1, \ldots, n$.
  By \Cref{thm_finite_supp_hyper}, there exist (nonempty) finite collections of closed half-spaces $\{H^+(i,j)\}_{j \in J_i}$ for $i = 1, \ldots, n$ such that the boundary $H(i,j)$ of $H^+(i,j)$ is a supporting hyperplane to $K - v_i$ at $\mathbf{0}$ and the dimension of the cone $C \coloneqq \bigcap_{i=1}^n\bigcap_{j \in J_i} H^+(i,j)$ is at most $d-1$.
  We define a new polytope:
  \[
    P'' \coloneqq P' \cap \bigcap_{j \in J_1} (H^+(1,j) + v_1) \cap \ldots \cap \bigcap_{j \in J_n} (H^+(n,j) + v_n) \text{.}
  \]
  Since $H(i,j)$ is a supporting hyperplane to $K-v_i$ at $\mathbf{0}$, the affine hyperplane $H(i,j) + v_i$ is a supporting hyperplane to $K$ at $v_i$, and thus $P''$ contains $K$ as well.
  Note that $Q + \delta$ is contained in $P''$, so that
  \[
    P''\text{\textdiv}(Q + \delta) = (P'' - v_1) \cap \ldots \cap (P'' - v_n) \subset C \text{.}
  \]
  Since $P''\text{\textdiv}Q$ is contained in $C$ and $C$ is not full-dimensional, it follows that $\dim(P''\text{\textdiv}Q) < d$.
  As the number of $[Q'] \in \TT_{P''}(X)$ with $\dim(P''\text{\textdiv}Q') = d$ is at most $r-1$, it follows by the induction hypothesis that we can intersect $P''$ with finitely many further half-spaces such that the resulting polytope $P$ contains $K$ and is $\RR$-$X$-free.
  This completes the proof.  
\end{proof}


\subsection{Intersection of convex bodies}\label{sec_intersec_conv_bodies}

In this section, we prove Theorem~\ref{thm_finite_supp_hyper} which played a key role in the proof of \Cref{prop_R_incl_max}.
The proof will be given at the end of the section, after some preliminary work.
The key idea is to reduce the proof of Theorem~\ref{thm_finite_supp_hyper} to the study of convex cones.
More precisely, we consider the tangent cone to the individual convex bodies at the given common point.
Recall the \emph{support cone} (sometimes also called \emph{tangent cone} or \emph{projection cone}) of a convex body $K \subset \RR^d$ at a point $x$ in $K$:
\[
  S_K(x) = \overline{\RR_{>0}(K-x)} \text{.}
\]
The dual notion is the \emph{normal cone} to $K$ at $x$:
\[
  N_K(x) \coloneqq \left\{ \ell \in (\RR^d)^* \colon \ell \mleft( x' - x \mright) \ge 0 \; \text{for any $x' \in K$} \right\} \text{.}
\]
It is a well-known fact that the support cone is a closed convex cone whose dual cone is the normal cone, i.e. $S_K(x) = N_K(x)^\vee$.
Furthermore, note that the dimension of the convex body $K$ coincides with the dimension of its support cone $S_K(x)$ at any point $x$ in $K$.

In Theorem~\ref{thm_finite_supp_hyper}, we study an intersection of convex bodies whose dimension is strictly less than the ambient dimension, and by what we have just said, it suffices to consider the involved tangent cones.
Note that the linear forms defining the half spaces of Theorem~\ref{thm_finite_supp_hyper} are elements in the normal cones to the individual convex bodies.
A key step in the proof of Theorem~\ref{thm_finite_supp_hyper} will be to understand how the tangent cone to the intersection of the convex bodies relates to the individual tangent cones.
To this end, we recall the following well-known statement which relates the dual of an intersection of closed convex cones to the individual dual cones.

\begin{proposition}\label{prop_dual_intersec_cones}
  Let $C_1, \ldots, C_n \subset \RR^d$ be closed convex cones. Then we have
  \[
    \left( C_1 \cap \ldots \cap C_n \right)^\vee = \overline{C_1^\vee + \cdots + C_n^\vee} \text{.}
  \]
\end{proposition}

The following example shows that in general it is necessary to take the closure in Proposition~\ref{prop_dual_intersec_cones}:

\begin{example}\label{ex_sum_not_closed}
  Let $C_1 = \{ (x,y,z) \in \RR^3 \colon x^2 + y^2 \le z^2, z \ge 0 \}$ and $C_2 = \RR_{\ge0} (1,0,-1)$ (see Figure~\ref{fig_sum_not_closed}).
  \begin{figure}[!ht]
    \begin{tikzpicture}
      \begin{scope}[canvas is zx plane at y=1]
        \draw[dashed] (0,0) ++(90:1) arc (90:300:1);
        \draw (0,0) ++(-60:1) arc (-60:90:1);
      \end{scope}
      \begin{scope}[canvas is zx plane at y=.4]
        \draw[dashed] (0,0) ++(90:.4) arc (90:300:.4);
        \draw (0,0) ++(-60:.4) arc (-60:90:.4);
      \end{scope}
      \begin{scope}[canvas is zx plane at y=1.5]
        \draw[dashed] (0,0) ++(90:1.5) arc (90:300:1.5);
        \draw (0,0) ++(-60:1.5) arc (-60:90:1.5);
      \end{scope}
      \fill (0,0) circle (2pt);
      \draw (0,0,0) -- (2,2,0) node[midway,right] {$C_1$};
      \draw (0,0,0) -- (-1.6,2,1.2);
      \draw[very thick] (0,0,0) -- (1,-1,0) node[midway, left]{$C_2$};
    \end{tikzpicture}
    \caption{Illustration of the cones in Example~\ref{ex_sum_not_closed}.\label{fig_sum_not_closed}
}
  \end{figure}
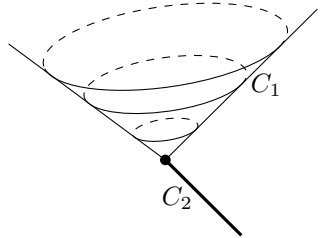
  Then $x_t \coloneqq (-t,1+\frac1t,\sqrt{t^2 + (1+\frac1t)^2}) \in C_1$ and $y_t \coloneqq (t,0,-t) \in C_2$.
  It is straightforward to verify that $\lim_{t\to\infty}(x_t + y_t) = (0,1,0)$ is contained in the closure of the sum of cones $C_1 + C_2$, but not in the sum.
\end{example}

Observe also that in this example, the line $\RR (1,0,-1)$ has one half-ray $\rho=\RR_{>0}(-1,0,1)$ contained in $C_1$ and its other half-ray $-\rho$ contained in the other cone $C_2$.
Waksman and Epelman~\cite{Waksman_Epelman} observed that such a line exists whenever the sum of two closed convex cones isn't closed:

\begin{theorem}[{\cite[Theorem on p.~95]{Waksman_Epelman}}]\label{thm_sum_not_closed}
  Suppose $d \ge 3$. Let $C_1, C_2 \subset \RR^d$ be two closed convex cones.
  If the sum $C_1 + C_2$ is not closed, then there exists a straight line $L = \rho + (-\rho)$ where $\rho = \RR_{>0} x$ for some $\mathbf{0} \neq x \in L$ such that $\rho \subset C_1$ and $-\rho \subset C_2$.
\end{theorem}
\begin{remark}
  Note that if the ambient dimension $d$ is less than or equal to $2$, then every closed convex cone is polyhedral.
  It is a well-known fact that the sum of two polyhedral cones is always closed.
  Indeed, if $x_1, \ldots, x_r, y_1, \ldots, y_s \in \RR^d$ (any dimension $d$), then $\cone(x_1, \ldots, x_r) + \cone(y_1, \ldots, y_s) = \cone(x_1, \ldots, x_r, y_1, \ldots, y_s)$, and thus the sum is closed.
  In particular, Theorem~\ref{thm_sum_not_closed} is an empty statement for $d\le2$, and Waksman and Epelman exclude these dimensions from their statement.
\end{remark}

\begin{corollary}\label{cor_sum_not_closed}
  Suppose $d \ge 3$. Let $C_1, \ldots, C_n \subset \RR^d$ be closed convex cones.
  If the sum $C_1 + \cdots + C_n$ is not closed, then there exists a straight line $L = \rho + (-\rho)$ where $\rho = \RR_{>0}x$ for some $\mathbf{0} \neq x \in L$ such that $\rho \subset C_i$ and $-\rho \subset C_1 + \cdots + \widehat{C_i} + \cdots + C_n$ for some index $i = 1, \ldots, n$.
\end{corollary}

Here, the notation $C_1 + \cdots + \widehat{C_i} + \cdots + C_n$ means that we omit the $i$-th summand.

\begin{proof}
  We do induction on $n$.
  The base case $n = 2$ is Theorem~\ref{thm_sum_not_closed}.
  Suppose $n>2$. If $C_2 + \cdots + C_n$ is not closed, by the induction hypothesis, there is a straight line $L = \rho + (-\rho)$ where $\rho = \RR_{>0}x$ for some $\mathbf{0} \neq x \in L$ such that $\rho \subset C_i$ and $-\rho \subset C_2 + \cdots + \widehat{C_i} + \cdots + C_n$ for some $i \in \{ 2, \ldots, n\}$.
  The statement then follows, since $C_2 + \cdots + \widehat{C_i} + \cdots + C_n \subset C_1 + C_2 + \cdots + \widehat{C_i} + \cdots + C_n$.
  If instead $C \coloneqq C_2 + \cdots + C_n$ is closed, it is a closed convex cone, and by Theorem~\ref{thm_sum_not_closed}, there exists a straight line $L = \rho + (-\rho)$ where $\rho = \RR_{>0}x$ for some $\mathbf{0} \neq x \in L$ such that $\rho \subset C_1$ and $-\rho \subset C_2 + \cdots + C_n$.
\end{proof}

The following statement will be used in the proof of Theorem~\ref{thm_finite_supp_hyper}.

\begin{lemma}\label{lem_half_spaces_same_dim}
  If $C \subset \RR^d$ is a closed convex cone, then there are finitely many linear forms $\ell_1, \ldots, \ell_n \in C^\vee$ such that
  \[
    \dim\left( \{ x \in \RR^d \colon \ell_1(x) \ge0\} \cap \ldots \cap \{ x \in \RR^d \colon \ell_n(x) \ge0\} \right) = \dim(C) \text{.}
  \]
\end{lemma}
\begin{proof}
  Let $e_1, \ldots, e_r \in C$ be a basis of the linear span $\lspan(C)$ of $C$.
  We extend these linearly independent vectors to a basis of $\RR^d$, say $e_1, \ldots, e_r, e_{r+1}, \ldots, e_d$.
  Let $f_1, \ldots, f_d$ be the dual basis of $(\RR^d)^*$.
  Then we have $\pm f_{r+1}, \ldots, \pm f_d \in C^\vee$, and thus
  \[
    \dim \left( \bigcap_{i=r+1}^d \{ x \in \RR^d \colon f_i(x) \ge0 \} \cap \bigcap_{i=r+1}^d \{ x \in \RR^d \colon -f_i(x) \ge0 \} \right) = \dim(C) \text{.}
  \]
\end{proof}

We are now ready to prove Theorem~\ref{thm_finite_supp_hyper}.

\begin{proof}[Proof of Theorem~\ref{thm_finite_supp_hyper}]
  Note $\dim(K) = \dim(S_K(x))$.
  Furthermore, the set of supporting hyperplanes of the tangent cone $S_{K_i}(x)$ at $x$ coincides with the set of supporting hyperplanes of $K_i$ at $x$.
  We continue by investigating the tangent cones $S_K(x), S_{K_1}(x), \ldots, S_{K_n}(x)$ and their relations.
  
  The following equality is straightforward to verify:
  \[
    \RR_{>0}(K - x) = \RR_{>0}\left( \left(K_1 \cap \ldots \cap K_n \right) - x \right) = \RR_{>0}(K_1 - x) \cap \ldots \cap \RR_{>0}(K_n - x).
  \]
  Since the dimension of the closure of a convex set coincides with the dimension of the original set, we get
  \begin{equation}
    \label{eq_1}
    \dim(S_K(x)) = \dim \left( \RR_{>0}(K-x) \right) = \dim \left( \RR_{>0}(K_1 - x) \cap \ldots \cap \RR_{>0} (K_n-x) \right) < d.
  \end{equation}
  We claim that $\dim(S_{K_1}(x) \cap \ldots \cap S_{K_n}(x)) <d$ as well.
  Assume towards a contradiction that the dimension of the intersection of the support cones were $d$, i.e., the intersection of the support cones is full-dimensional.
  Then the intersection of the support cones contains an affine basis of $\RR^d$, and thus a full-dimensional simplex $\Delta$.
  Let $y$ be the barycentre of $\Delta$.
  Let $\Delta' \subset \Delta$ be the simplex obtained by shrinking $\Delta$ with respect to its barycentre, e.g. $\Delta' = \frac12(\Delta - y) + y$.
  Then the smaller simplex $\Delta'$ is contained in the interior of every tangent cone $S_{K_i}(x)$ for $i = 1, \ldots, n$, and thus
  \[
    \Delta' \subset \RR_{>0}(K_1-x) \cap \ldots \cap \RR_{>0}(K_n-x) \text{,}
  \]
  which is a contradiction to inequality~\eqref{eq_1}.
  
  Hence $\dim(S_{K_1}(x) \cap \ldots \cap S_{K_n}(x)) < d$.
  By Proposition~\ref{prop_dual_intersec_cones}, we have
  \[
    (S_{K_1}(x) \cap \ldots \cap S_{K_n}(x))^\vee = \overline{N_{K_1}(x) + \cdots + N_{K_n}(x)} \text{.}
  \]
  We distinguish two cases, namely whether $N_{K_1}(x) + \cdots + N_{K_n}(x)$ is closed or not.
  
  Suppose $N_{K_1}(x) + \cdots + N_{K_n}(x)$ is closed.
  By Lemma~\ref{lem_half_spaces_same_dim}, there exist linear forms $\ell_1, \ldots, \ell_r \in (S_{K_1}(x) \cap \ldots \cap S_{K_n}(x))^\vee$ such that
  \begin{equation}
    \label{eq_2}
    \dim(\{ y \in \RR^d \colon \ell_1(y) \ge0 \} \cap \ldots \cap \{ y \in \RR^d \colon \ell_n(y) \ge0 \} ) = \dim (S_{K_1}(x) \cap \ldots \cap S_{K_n}(x)) < d \text{.}
  \end{equation}
  Since $(S_{K_1}(x) \cap \ldots \cap S_{K_n}(x))^\vee = N_{K_1}(x) + \cdots + N_{K_n}(x)$, every linear form $\ell_j$ can be expresses as a sum $\ell_j = \ell_{1,j} + \cdots + \ell_{n,j}$ for some linear forms $\ell_{i,j} \in N_{K_i}(x)$.
  We define $H^+(i,j) = \{ y \in \RR^d \colon \ell_{i,j}(y) \ge0\}$.
  Then the intersection of all these closed half-spaces is a polyhedral cone $C$ whose dual cone is given by $C^\vee = \cone(\{ \ell_{i,j} \colon i=1, \ldots, n; j = 1, \ldots r \})$.
  The dual cone $C^\vee$ contains the cone $D \coloneqq \cone(\ell_1, \ldots, \ell_r) \subset (\RR^d)^*$, and thus $C \subset D^\vee$.
  Since, by Equation~\eqref{eq_2}, $\dim(D^\vee) < d$, we are done.
  
  Suppose now $N_{K_1}(x) + \cdots + N_{K_n}(x)$ is not closed.
  Then by Corollary~\ref{cor_sum_not_closed}, there exists a straight line $L = \rho + (-\rho) \subset (\RR^d)^*$ where $\rho = \RR_{>0}\ell$ for some $\mathbf{0} \neq \ell \in L$ such that $\rho \subset N_{K_i}(x)$ and $-\rho \subset N_{K_1}(x) + \cdots + \widehat{N_{K_i}(x)} + \cdots + N_{K_n}(x)$ for some index $i=1, \ldots, n$.
  There exist $\ell_j \in N_{K_j}(x)$ for $j \neq i$ such that $-\ell = \ell_1 + \cdots + \widehat{\ell_i} + \cdots + \ell_n$.
  We define $H^+_i \coloneqq \{ y \in \RR^d \colon \ell(y) \ge 0\}$ and $H^+_j \coloneqq \{ y \in \RR^d \colon \ell_j(y) \ge0\}$ for $j \neq i$.
  Let $C$ be the intersection of these closed half-spaces.
  Then the dual cone is given by $C^\vee = \cone( \ell, \ell_1, \ldots, \widehat{\ell_i}, \ldots, \ell_n)$ which contains the straight line $L$.
  Thus $C \subset L^\vee$ where $\dim(L^\vee) = d-1$.
\end{proof}

\begin{remark}
    Note the proof of Theorem~\ref{thm_finite_supp_hyper} shows that the sets $\set{H^+(i,j)}_{j \in J_i}$ for $i=1, \ldots, n$ can be chosen such that $|J_i|\le 2$.
\end{remark}


\section{Preliminary observations in 1 and 2 dimensions}\label{sec:1-2-dim}

Let us begin our quest of determining $\flatness_d^A(\Delta_d)$ for $d = 1,2$ and $A \in \{ \ZZ, \RR \}$.
We first settle the one-dimensional case.
By Remark~\ref{rem:basic-props}, the flatness constants of bounded sets $X \subset \RR$ whose convex hull is full-dimensional are completely characterised by the flatness constants of closed intervals $I = [x,y] \subset \RR$.

Recall that the floor of a real number $x$, $\lfloor x \rfloor$, is the largest integer which is less than or equal to $x$.
Similarly, the ceiling of a real number $x$, $\lceil x \rceil$, is the smallest integer that is greater than or equal to $x$.

\begin{theorem}\label{thm:flt-dim1}
  Let $I = [x,y] \subset \RR$, with $x \le y$.
  Set $\delta \coloneqq \floor{x}+\ceil{y}$.
  Then
  \[
    \flatness_1^\ZZ(I) = \begin{cases}
      \max\{ \delta - 2x, 1 + 2y - \delta \} & \text{if} \; \ceil{x} - x \ge y - \floor{y},\\
      \max\{ 2y - \delta, 1 + \delta -2x \} & \text{otherwise.}
    \end{cases}
  \]
\end{theorem}
\begin{proof}
  Let $\Im$ be the set of all transformed intervals under $\ZZ$-unimodular transformations, i.e., $\Im\coloneqq \setcond{T(I)}{T \; \text{a $\ZZ$-unimodular transformation}}$.
  Note that $\Im$ comes equipped with a total ordering, namely $[x,y] < [x',y']$ if $x < x'$.
  In this proof, we will simply write ``$I$-free'' for ``$\ZZ$-$I$-free''.
  
    We call two intervals $J < J'$ in $\Im$ \novel{successive} if for any interval $J'' \in \Im$ such that $J \le J'' \le J'$ it follows that either $J=J''$ or $J'=J''$.
    It is straightforward to show that the inclusion-maximal $I$-free convex bodies are exactly the convex hulls of unions of two successive intervals in $\Im$.
  
  It remains to determine the structure of $\Im$ with respect to its total order.
  As $\Im$ is the union of $I + \ZZ$ and $-I + \ZZ$, there can be at most one translate of $-I$ between $I$ and $I+1$, say $I \le -I + \delta \le I+1$.
  To determine the translation factor $\delta$, we distinguish two cases:
  
  Suppose $\ceil{x} - x \ge y - \floor{y}$ (or equivalently $x - \floor{x} \le \ceil{y} - y$).
  Set $\delta \coloneqq \ceil{y} + \floor{x}$.
  With the above it is straightforward to verify that $I \le -I+\delta < I+1$ where $I = -I+\delta$ if and only if both $x,y$ are integers.
  Then up to translation by integers, every $I$-free convex body is contained in $I \cup (-I + \delta) = [x, -x+\delta]$ or $(-I+\delta) \cup (I+1) = [-y+\delta,y+1]$ which are also $I$-free.
  Hence, $\flatness_1^\ZZ(I) = \max\{ \delta - 2x, 1+2y-\delta \}$.
  
  If $\ceil{x}-x < y - \floor{y}$, replace $I$ by $-I = [a,b] = [-y,-x]$.
  Note that $\ceil{a}-a = \ceil{-y}+y = y - \floor{y} > \ceil{x} - x = b - \floor{b}$, and we are back in the previous case.
  Clearly $\flatness_1^\ZZ(I) = \flatness_1^\ZZ(-I)$.
\end{proof}

\begin{remark}
  Note that in the situation of Theorem~\ref{thm:flt-dim1}, we have $\flatness_1^\RR(I) = y-x$.
  Indeed, the maximal $I$-free convex bodies are exactly the translates of $I$.
  
  Furthermore, Theorem~\ref{thm:flt-dim1} shows that in general $\flatness_1^\ZZ(n X) \neq n \flatness_1^\ZZ(X)$ ($n \in \NN$).
  Indeed, using Theorem~\ref{thm:flt-dim1}, we immediately get that $\flatness_1^\ZZ([0,\frac43]) = 3$ while $2 \cdot \flatness_1^\ZZ([0,\frac23])=4$.
  This answers the question in~\cite{AHN} whether $\flatness_d^\ZZ(\cdot)$ is linear with respect to positive dilations.
\end{remark}

This gives a complete answer in $1$ dimension for full-dimensional $X \subset \RR$.
Let us now turn to $2$ dimensions.
Recall from Section~\ref{sec:prelim} our general strategy: considering inclusion-maximal $A$-$\Delta_2$-free closed convex sets $C$, we obtain an upper bound for $\flatness_d^A(\Delta_2)$.
Furthermore, recall that since $\Delta_2$ is full-dimensional, it suffices to consider full-dimensional inclusion-maximal $A$-$\Delta_2$-free closed convex sets $C$ (see Lemma~\ref{lem:full-dim}).
Here, we will classify the unbounded $C$.
Their widths will turn out to be strictly smaller than the maximal width of the bounded $C$'s.
Hence, the maximal width of the bounded $C$'s will be then equal to the respective flatness constant.

\begin{proposition}\label{prop:unbounded}
	Let $A \in \{\RR, \ZZ\}$.
	Then up to $A$-unimodular transformations there exists exactly one unbounded inclusion-maximal $A$-$\Delta_2$-free closed convex set $C \subset \RR^2$ of dimension $2$, namely
	\begin{itemize}
	\item
    	if $A = \ZZ$, then $C = [-1,1] \times \RR$; and
	\item
    	if $A = \RR$, then $C = [0,1] \times \RR$.
	\end{itemize}
	In particular, the width is $2$ in the first and $1$ in the second case.
\end{proposition}

For the proof of the previous statement, we need to recall the definition of the \emph{tail cone} (or \emph{recession cone}) of a closed convex set $A \subset \RR^d$:
\[
    \mathrm{tail}(A) \coloneqq \setcond{v \in \RR^d}{x+\lambda v \in A \, \text{for all} \, x \in A \, \text{and} \, \lambda\ge0} \text{.}
\]

\begin{proof}
	Since $C \subset \RR^2$ is an unbounded closed convex set, it follows by~\cite[Theorem~8.4]{Rockafellar} that $\mathrm{tail}(C) \neq \{\mathbf{0}\}$.
	Then $\mathrm{tail}(C)$ is a $1$-dimensional closed convex cone.
	Hence $\mathrm{tail}(C)$ lies on a line $y = m\cdot x$ for a real number $m \in \RR$.
	We claim $m \in \QQ$ is rational.
	
	Assume towards a contradiction that $m \in \RR \setminus \QQ$ were irrational.
	Since $C$ is full-dimensional, it contains a small affine ball $v + \varepsilon B^2$ for $v \in \RR^2$ and $\varepsilon >0$.
	Indeed, by approximating $v$ with a rational point and decreasing $\varepsilon$ (if necessary), we may assume that $v \in \QQ^2$.
	Then $v + \varepsilon B^2 + \mathrm{tail}(C)$ is contained in $C$ where $v + \mathrm{tail}(C)$ lies on an affine line parallel to $y = m \cdot x$, say $\set{y = m \cdot x + c} \subset \RR^2$ for $c \in \RR$.
	Note $(v + \mathrm{tail}(C)) \cap \QQ^2 = \set{v}$ as otherwise $m$ would be a rational number (a contradiction).
	By Kronecker's theorem~\cite{Kronecker84, Kronecker85} (see also~\cite[Theorem~438]{HardyWright} or~\cite[Theorem~1]{HlawkaEtAl}), for every $\delta, N > 0$ there exist integers $x_\delta, y_\delta \in \ZZ$ with $|x_\delta| \ge N$ (where both $x_\delta >0$ and $x_\delta <0$ can be chosen) such that $\mleft| m \cdot x_\delta + c - y_\delta \mright|<\delta$, i.e., $(x_\delta, y_\delta)$ comes arbitrarily close to $v + \mathrm{tail}(C)$.

	Let $x, y, z \in \ZZ^2$ be three lattice points that were chosen so that $y$ lies strictly closer to $v+\mathrm{tail}(C)$ than $x$ while $z$ lies closer to $v+\mathrm{tail}(C)$ than $y$.
	We may assume that $x,y,z$ lie in $v+\varepsilon B^2 + \mathrm{tail}(C)$, and thus in the interior of $C$.
	It follows that $\conv(x,y,z)$ is a lattice triangle contained in the interior of $C$.
	Triangulating this triangle into empty simplices yields a $\ZZ$-unimodular copy of $\Delta_2$ that is contained in the interior of $C$.
	A contradiction.
	Note that this solves both cases $A = \RR$ and $A = \ZZ$.

	Hence $m \in \QQ$ is rational, and thus up to a $\ZZ$-unimodular transformation, we have that $v + \mathrm{tail}(C)$ lies on the line given by $x = c$ for some $c \in \RR$.
	Consider the projection $\pi \colon \RR^2 \to \RR; (a,b) \mapsto a$.
	Since $C$ is convex, the closure of the image $\pi(C)$ is an interval, say $I = [r,s]$ for $r, s\in \RR$.
	Note the tail cone $\mathrm{tail}(C)$ ensures that for any $u$ in $I$ there exists an affine ray $w + \mathrm{tail}(C) \subset C$ that projects down to $u$.
	
	Suppose $A = \ZZ$.
	Then $s-r\le2$, as otherwise $I$ would contain two integers in its interior which would imply that $C$ contains a $\ZZ$-unimodular copy of $\Delta_2$ in its interior.
	It straightforwardly follows that (up to a translation by a lattice point) $C = [-1,1]\times \RR$.
	
	Suppose $A = \RR$.
	Then $s-r \le 1$, as otherwise $I$ would contain an $\RR$-translate of the interval $[0,1]$ in its interior which would imply that $C$ contains an $\RR$-unimodular copy of $\Delta_2$ in its interior.
	Then it is easy to see that (up to a real translation) $C = [0,1]\times \RR$.
\end{proof}
It remains to study the bounded cases.
The remaining sections of this manuscript will be concerned with this study.


\section{\texorpdfstring{The $\ZZ$-flatness constant of $\Delta_2$}{The Z-flatness constant of Delta-2}}\label{sec:Z_flatness_dim2}

Let $K \subset \RR^d$ be a convex body and let $X \subset \RR^d$ be an arbitrary bounded set.
By \Cref{prop:A_Z_X_poly_K_poly}, inclusion-maximal $\ZZ$-$X$-free bodies are polytopes (a stronger version of \Cref{thm_incl_max} for the case $A=\ZZ$).
The following definition is the key to characterising inclusion-maximal $\ZZ$-$\Delta_d$-free polytopes.
\begin{definition}
    A facet $F$ of a full-dimensional polytope $P \subset \RR^d$ is said to be $\ZZ$-$\Delta_d$-\emph{locked} if there exists a $\ZZ$-unimodular copy $T$ of $\Delta_d$ contained in $P$ such that $T\cap \relint(F) \neq \emptyset$ and $V(T)\setminus \relint(F) \subset \strint(P)$, where $V(T)$ denotes the set of vertices of $T$.
    Notice that $T$ gives rise to lattice points in the relative interior of $F$, namely $V(T)\cap\relint(F)$.
    Such lattice points are called \emph{locking points}.
\end{definition}

We will say simply \emph{locked} instead of $\ZZ$-$\Delta_d$-locked wherever it is clear from the context that we are discussing $\ZZ$-$\Delta_d$-flatness.
See \Cref{fig:locked_facets} for an illustration of the concepts of ``locked facet'' and ``locking point''.
\begin{figure}[!ht]
    \begin{tikzpicture}

        \fill[very thick,draw=blue,fill=blue!30] (-1,1.25) -- (2,.5) -- (2,-.75)-- (-1,-1.125) -- cycle;

        \draw (-1,.2) node [left] {$F_1$};
        \draw (1,1.1) node [left] {$F_2$};
        \draw (2,.1) node [right] {$F_3$};
        \draw (.5,-1) node [below] {$F_4$};

        \draw[red, dashed, thick] (-1,1)--(-1,0)--(0,0)--cycle;
        \draw[red, thick, dashed] (0,1)--(1,0)--(0,0)--cycle;
        \draw[red, thick, dashed] (0,-1)--(1,0)--(0,0)--cycle; 

        \foreach \y in {-1,...,1}
            \foreach \x in {-1,...,2}
                \fill (\x,\y) circle (2pt);
    \end{tikzpicture}
    \caption{Facets $F_1, F_2$ and $F_4$ of this $\ZZ$-$\Delta_2$-free polygon are locked, while $F_3$ is not. Dashed in red are $\ZZ$-unimodular copies of $\Delta_2$ that are locking the respective facets.
    Note the polygon is indeed $\ZZ$-$\Delta_2$-free since all interior lattice points are collinear.\label{fig:locked_facets}
}
\end{figure}
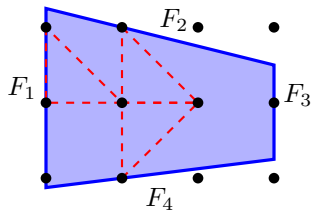
The definition of locked facet is crafted so that if a point is added beyond any locked facet, the resulting polytope is no longer a $\ZZ$-$\Delta_d$-free polytope.
Recall that a point $x$ is \emph{beyond} the facet $F$ of a full-dimensional polytope $P$ if it lies in the half-space which is defined by the supporting hyperplane of $F$ and which does not contain $\strint(P)$.
Furthermore, $x$ is \emph{beneath} $F$ if it lies in the same half-space as $P$.

The following proposition shows that being inclusion-maximal among $\ZZ$-$\Delta_d$-free polytopes is equivalent to all facets being locked. 
\begin{proposition}\label{prop:locked_incl_max_delta2_free}
    Let $P \subset \RR^d$ be a $\ZZ$-$\Delta_d$-free polytope.
    Then, $P$ is inclusion-maximal if and only if all its facets are locked.
\end{proposition}
\begin{proof}
    Suppose that all facets of $P$ are locked.
    Let $F$ be a facet of $P$ and $x \in \RR^d \setminus P$ be a point beyond $F$.
    Since $F$ is locked, there exists a $\ZZ$-unimodular copy $T$ of $\Delta_d$ contained in $P$ such that $T\cap \relint(F) \neq\emptyset$ and $V(T)\setminus \relint(F) \subset \strint(P)$.
    Let $Q=\conv(P,x)$.
    The relative interior of the facet $F$ lies in the interior of Q, and thus so does $T$. Hence $P$ is inclusion-maximal with respect to the property of being $\ZZ$-$\Delta_d$-free.
    
    We prove the reverse implication by verifying the contrapositive, i.e., if there exists a facet of $P$ that is not locked, then $P$ isn't inclusion-maximal with respect to the property of being $\ZZ$-$\Delta_d$-free.
    Let $F$ be a facet of $P$ that is not locked and let $Q$ be the polytope obtained from $P$ by moving the facet $F$ outwards by a small amount, such that $P \subsetneq Q$ but no new lattice points are captured, that is, $P \cap \ZZ^d = Q \cap \ZZ^d$.
    In particular, the set of $\ZZ$-unimodular copies of $\Delta_d$ which are contained in $P$ coincides with the set of such copies that are contained in $Q$.
    Note however that any lattice point in the relative interior of the facet $F$ of $P$ are in the interior of $Q$.
    From the assumption that $F$ is not locked we will now deduce that $Q$ is also  a $\ZZ$-$\Delta_d$-free polytope, and thus $P$ is not inclusion-maximal. 
    
    Let $H$ be the supporting hyperplane that defines the facet $F$ of $P$, and let $H_{\ge0}$ (resp.~$H_{>0}$) be the closed half-space (resp.~open half-space) with boundary equal to $H$ such that $P \subset H_{\ge0}$.
    Notice $Q \cap H_{\ge0} = P$.
    It remains to show that any $\ZZ$-unimodular copy $T$ of $\Delta_d$ that is contained in $Q$ isn't contained in the interior of $Q$.
    Recall from above that $T$ is also contained in $P$.
    We distinguish two cases:
    \begin{itemize}
        \item 
            If $T \cap \relint(F) = \emptyset$, then $T \subset Q \cap H_{>0} = P\setminus F$.
            Since $Q \cap H_{\ge0} = P$ and $T \not \subset \strint(P)$, it straightforwardly follows that $T \not \subset \strint(Q)$.
        \item
            If $T \cap \relint(F) \neq \emptyset$, then $V(T) \setminus \relint(F) \not \subset \strint(P)$, i.e., there is another facet $F'$ of $P$ that contains a vertex of $T$.
        Since $F'$ is contained in a facet of $Q$, it follows that $T \not\subset \strint(Q)$.\qedhere
    \end{itemize}
\end{proof}

We now focus on dimension $d=2$, where our goal is to show the following theorem.
\begin{theorem}[Case $A=\ZZ$ of \Cref{thm:main}]\label{theorem:flt_Z_D2}
$\flatness_2^\ZZ(\Delta_2) = \frac{10}{3}$.
\end{theorem}

By \Cref{lem:incl-max}, to prove the theorem it is enough to show that any inclusion-maximal $\ZZ$-$\Delta_2$-free convex set has width at most $\frac{10}{3}$ and to provide an example of a $\ZZ$-$\Delta_2$-free polygon of that width.
Unbounded full-dimensional inclusion-maximal $\ZZ$-$\Delta_2$-free convex sets were studied in \Cref{prop:unbounded} and have width $1$.
Thus we are left with determining the maximum width of inclusion-maximal $\ZZ$-$\Delta_2$-free convex bodies, which are polygons by \Cref{thm_incl_max}.
The rest of the section is devoted to proving that this width is $\frac{10}{3}$. 

Note \Cref{prop:locked_incl_max_delta2_free} guarantees that any inclusion-maximal $\ZZ$-$\Delta_2$-free polygon contains at least one interior lattice point, since otherwise it is impossible for its facets to be locked. 
The following proposition deals with (inclusion-)maximal $\ZZ$-$\Delta_2$-free polygons containing exactly one interior lattice point.
The case of polygons whose interior contains at least two lattice points will be treated afterwards. 

\begin{proposition}\label{prop:one_interior_point}
    If $P\subset \RR^2$ is a maximal $\ZZ$-$\Delta_2$-free polygon with $|\strint(P)\cap\ZZ^2|=1$, then $\width(P)\le 3$.
\end{proposition}
\begin{proof}
    Let $P$ be an inclusion-maximal $\ZZ$-$\Delta_2$-free polygon with exactly one interior lattice point.
    By \Cref{prop:locked_incl_max_delta2_free}, each facet $F$ of $P$ is locked.
    Since there's a unique lattice point in the interior of $P$, for each facet $F$ there is a $\ZZ$-unimodular copy $T$ of $\Delta_2$ contained in $P$ such that two vertices of $T$ are contained in the relative interior of $F$ and the third vertex of $T$ is in the interior of $P$.
    Up to an appropriate unimodular transformation, we may assume that there is a facet for which $T = \conv\mleft(\mathbf{0},e_1,e_2\mright)$, and that $\mathbf{0}$ is the interior lattice point of $P$.
    By using the fact that $\mathbf{0}$ is the only interior lattice point of $P$, we can conclude that $P$ is disjoint from the following regions (see \Cref{fig:one_int_lattice_pt}):
    \begin{itemize}
    \item[$C_1$]
    	The affine cone $(-e_1+e_2)+\cone(-e_1,-e_1+e_2)$ minus its apex $-e_1+e_2$ (otherwise $-e_1+e_2$ would be contained in the interior of $P$.
        Note that we include the open half-rays of the affine cone into $C_1$ because $\mathbf{0}$ is in the interior of $P$ and $e_2$ is in the relative interior of $F$);
    \item[$C_2$]
    	The affine cone $-e_1 + \cone(-e_1,-e_1-e_2)$ minus its apex $-e_1$ (otherwise $-e_1$ would be contained in the interior of $P$);
    \item[$C_3$]
    	The affine cone $(-e_1-e_2) + \cone(-2e_1-e_2,-e_1-2e_2)$ minus its apex $(-e_1-e_2)$ (otherwise $-e_1-e_2$ would be contained in the interior of $P$).
	\end{itemize}

    Let $C'_i$ be the region obtained from $C_i$ by reflecting along the line $\RR(e_1+e_2)$.
    Notice $C_3=C'_3$.
    \begin{figure}[!ht]
        \begin{tikzpicture}
            \draw[dashed] (-2.5,3.5) -- (3.5,-2.5);
            \foreach \x in {0,.5,1} \draw[dashed] (-2-\x,2-\x) -- (2-\x,-2-\x);
    
            \fill[very thick,draw=blue,fill=blue!50] (0,0) -- (1,0) -- (0,1) -- cycle;
            \draw[thick,densely dotted] (0,0) circle (4pt);
            
            \fill[draw=red!25,fill=red!25] (-1,1) -- (-3,3) -- (-3,1) -- cycle;
            \draw[red,thick] (-1,1) -- (-3,3);
            \draw[red,thick] (-1,1) -- (-3,1);
            \node at (-2,1.5) {$C_1$};
            
            \fill[draw=red!25,fill=red!25] (1,-1) -- (3,-3) -- (1,-3) -- cycle;
            \draw[red,thick] (1,-1) -- (3,-3);
            \draw[red,thick] (1,-1) -- (1,-3);
            \node at (1.5,-2) {$C'_1$};
    
            \fill[draw=red!25,fill=red!25] (-1,0) -- (-3,0) -- (-3,-2) -- cycle;
            \draw[red,thick] (-1,0) -- (-3,0);
            \draw[red,thick] (-1,0) -- (-3,-2);
            \node at (-2,-.4) {$C_2$};
    
            \fill[draw=red!25,fill=red!25] (0,-1) -- (0,-3) -- (-2,-3) -- cycle;
            \draw[red,thick] (0,-1) -- (0,-3);
            \draw[red,thick] (0,-1) -- (-2,-3);
            \node at (-.4,-2) {$C'_2$};
            
            \fill[draw=red!25,fill=red!25] (-3,-2) -- (-1,-1) -- (-2,-3) -- cycle;
            \draw[red,thick] (-1,-1) -- (-3,-2);
            \draw[red,thick] (-1,-1) -- (-2,-3);
            \node at (-2,-2) {$C_3$};
    
            \foreach \s [count=\n from 0] in {2,1,0,-1}
                \foreach \num in {0,...,\n}
                    \fill (-1+\num,-1+\n-\num) circle (2pt);
        \end{tikzpicture}
        \caption{The regions $C_i$ and $C'_i$ which are disjoint from $P$.\label{fig:one_int_lattice_pt}
}
    \end{figure}
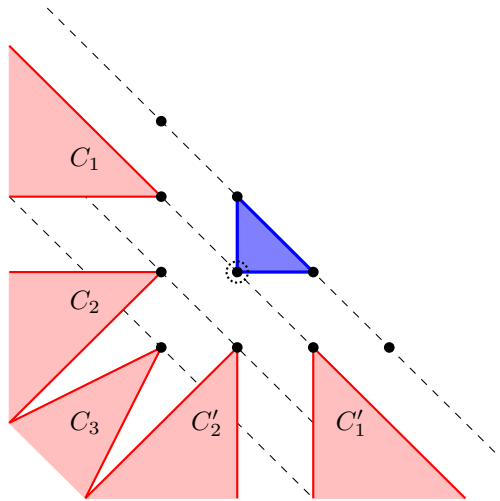
    Since $\mathbf{0}$ is the only interior lattice point, it follows from \Cref{prop:locked_incl_max_delta2_free} that every facet of $P$ contains two lattice points in its relative interior.
    We already know that $\set{x+y=1}$ cuts out a facet of $P$.
    For the remaining facets, the only candidates are lattice points disjoint from the regions $C_i$, $C'_i$ and contained in the open half-plane $\set{x+y<1}$, i.e., $-e_1 + e_2,  -e_1, -e_1 - e_2,  -e_2, e_1 - e_2$, those drawn in \Cref{fig:one_int_lattice_pt}.
    It is easy to see that the possibility for $P$ to have two other facets each containing two of these lattice points in the relative interior is if $P=\conv( -e_1 + 2e_2, -e_1 - e_2,  2e_1 - e_2)$, which has width equal to $3$.
\end{proof}

Now suppose the polygon $P$ contains at least two interior lattice points.
Clearly these interior lattice points are collinear, since any set of non-collinear points contains a triangle, and any lattice triangle can be triangulated into unimodular ones, and thus in particular contains a unimodular triangle.

The following theorem shows that polygons with at least two interior lattice points can have larger width than $3$ and the maximum width is achieved by a triangle with exactly two interior lattice points.

\begin{theorem}\label{thm:2-or-more-int-pts}
    If $P\subset \RR^2$ is a $\ZZ$-$\Delta_2$-free polygon with $|\strint(P)\cap\ZZ^2|\geq 2$, then $\width(P)\le \frac{10}3$.
    Equality is only achieved (up to $\ZZ$-unimodular transformations) by $\conv\mleft(\frac13 e_1+\frac53 e_2, -\frac43 e_1-\frac53 e_2,2e_1\mright)$, which contains exactly $2$ interior lattice points.
\end{theorem}

In order to prove \Cref{thm:2-or-more-int-pts}, it suffices to study inclusion-maximal $\ZZ$-$\Delta_2$-free polygons which have at least two interior lattice points.
Since in \Cref{prop:one_interior_point} we have already found a $\ZZ$-$\Delta_2$-free triangle with width equal to $3$, to find the polygons of largest width we can restrict our study to those whose width is greater than $3$.
The strategy for the proof is to use \Cref{prop:locked_incl_max_delta2_free} and to distinguish polygons according to their locking lattice points.

Thus from now on we let $P$ be a maximal $\ZZ$-$\Delta_2$-free polygon with at least two interior lattice points and width strictly greater than $3$.
Up to an affine unimodular transformation, we may assume that 
\begin{itemize}
    \item
    	the interior of $P$ contains $\mathbf{0}$ and $e_1$,
    \item 
    	any lattice point which does not lie on the horizontal axis is not in the interior of $P$, and
    \item
    	$P$ contains a point $p_0=xe_1+ye_2$ with $0<x<1$ and $y>\frac32$.
\end{itemize}
Indeed, there is a $\ZZ$-unimodular transformation that maps the lattice segment comprised by the interior lattice points of $P$ (recall they are collinear) onto a lattice segment lying on the $x$-axis.
Since this lattice segment has lattice length at least one, we can assume that both $\mathbf{0}$ and $e_1$ are contained in it.
To simplify notation, let us use the same symbol $P$ for the transformed polygon.
Since the width of $P$ is larger than $3$, it cannot be contained in the strip $\set{-\frac32\le y \le \frac32}$.
After possibly flipping along the $x$-axis, this shows that $P$ contains a point $p_0$ whose $y$-coordinate is larger than $\frac32$.
Note the triangle with vertices $\mathbf{0}$, $e_1$ and $p_0$ intersects the line $\set{y=1}$ in a segment of length less than $1$ which does not contain lattice points (as otherwise such a lattice point would be in the interior of $P$).
After applying an appropriate horizontal shearing this segment lies in the strip $\set{0\leq x <1}$, and thus so does $p_0$.
Recall that a linear unimodular transformation $\varphi \colon \ZZ^2 \to \ZZ^2$ of the form
\[
	\varphi\mleft(\lambda_1 b_1 + \lambda_2 b_2\mright) = \mleft( \lambda_1+k\lambda_2\mright)b_1 + \lambda_2 b_2 \text{,}
\]
for a lattice basis $b_1,b_2$ of $\ZZ^2$ and an integer $k \in \ZZ$ is called a \emph{shearing along the line $\RR b_1$}.

Suppose $P$ has been transformed to satisfy the conditions above.
Let us determine the possible lattice points that can lock a facet of $P$.
No lattice point on the $x$-axis can lock a facet, since they are collinear with the interior lattice points of $P$.
Since the triangle $T \coloneqq \conv(\mathbf{0}, e_1, p_0)$ is contained in $P$, it follows that the affine locally open rays $e_2 + \RR_{<0}e_1$ and $e_1+e_2+\RR_{>0}e_1$ are disjoint from $P$.
Hence the only lattice points in the upper half-plane $\set{y>0}$ that can be contained in $P$ are $e_2$ and $e_1+e_2$ (any other lattice point in $\set{y>0}$ forces $e_2$ or $e_1+e_2$ to be in the interior of $P$).
Next we determine which lattice points in the lower half-space $\set{y<0}$ can be contained in $P$.
Since $p_0 \in P$, every lattice point in $\{y<0\}$ which is in the upper closed half-space given by the line going through the two points $e_2$, $e_1+\frac32 e_2$ is disjoint from $P$ (otherwise $e_2$ is in the interior of $P$).
Using the symmetry induced by reflecting about the vertical line $\set{x=\frac12}$, it follows that $P\cap\{y<0\}$ is also disjoint from the upper closed half-space given by the line through the two points $\frac32 e_2$, $e_1+e_2$.
Note the remaining lattice points at height $y=-1$ could be contained in $P$: $-3e_1-e_2, -2e_1-e_2,-e_1-e_2, -e_2, e_1-e_2, 2e_1-e_2, 3e_1-e_2$ or  $4e_1-e_2$.
Let $q_k \coloneqq ke_1 - e_2$ for $k \in \set{-3, -2, -1}$.
Since $q_k$ cannot be an interior lattice point of $P$, it follows that $P$ is disjoint from the region $C_k$ which is defined to be $\sigma_k \setminus \set{ q_k}$ where $\sigma_k$ is the affine cone having apex at $q_k$ and supporting lines going through $q_k$ and $e_1$ or $e_1+\frac32 e_2$ respectively.
For $q_0 = -e_2$, the respective region is $C_0 = \sigma_0 \setminus \set{q_0}$ where $\sigma_0$ is the affine cone having apex at $q_0$ and supporting lines the $y$-axis and the line going through $-e_2$ and $e_1$ (see \Cref{fig:possible-locking-pts}).
\begin{figure}[!ht]
    \begin{tikzpicture}[scale=0.5]	
    	\fill[draw=red!25,fill=red!25] (-2,0) -- (-5,-1.5) -- (-3,-1) -- (-4.14,-1.71) -- (-2,-1) -- (-3.5, -2.25) -- (-1,-1) -- (-5,-6) -- (0,-1) -- (0,-6.5) -- (-5.5,-6.5) -- (-5.5,0) -- cycle;
    	\draw[red,thick] (-2,0) -- (-5,-1.5);
    	\draw[red,thick] (-5,-1.5) -- (-3,-1);
    	\draw[red,thick] (-3,-1) -- (-4.14,-1.71);
    	\draw[red,thick] (-4.14,-1.71) -- (-2,-1);
    	\draw[red,thick] (-2,-1) -- (-3.5, -2.25);
    	\draw[red,thick] (-3.5, -2.25) -- (-1,-1);
    	\draw[red,thick] (-1,-1) -- (-5,-6);
    	\draw[red,thick] (-5,-6) -- (0,-1);
    	\draw[red,thick] (0,-1) -- (0,-6.5);
		
		\draw[red,thick,dashed] (-2,-1) -- (-5.5,-2.161);
		\draw[red,thick,dashed] (-2,-1) -- (-5.5,-3.92);
		
		\path (-5.5,-3.9) -- (-5.5,-2.2) node[midway,above right] {$\scriptstyle C_{-2}$};

    	\begin{scope}[xshift=1cm, xscale=-1]
    		\fill[draw=red!25,fill=red!25] (-2,0) -- (-5,-1.5) -- (-3,-1) -- (-4.14,-1.71) -- (-2,-1) -- (-3.5, -2.25) -- (-1,-1) -- (-5,-6) -- (0,-1) -- (0,-6.5) -- (-5.5,-6.5) -- (-5.5,0) -- cycle;
        	\draw[red,thick] (-2,0) -- (-5,-1.5);
        	\draw[red,thick] (-5,-1.5) -- (-3,-1);
        	\draw[red,thick] (-3,-1) -- (-4.14,-1.71);
        	\draw[red,thick] (-4.14,-1.71) -- (-2,-1);
        	\draw[red,thick] (-2,-1) -- (-3.5, -2.25);
        	\draw[red,thick] (-3.5, -2.25) -- (-1,-1);
        	\draw[red,thick] (-1,-1) -- (-5,-6);
        	\draw[red,thick] (-5,-6) -- (0,-1);
        	\draw[red,thick] (0,-1) -- (0,-6.5);
    	\end{scope}
    	
    	\draw[very thin,dashed] (-2,-1) -- (1,0);
    	\draw[very thin,dashed] (-2,-1) -- (1,1.5);

    	\draw[dashed] (-5.5,0) -- (6.5,0);

    	\foreach \x in {-2, ...,3}
    		\fill[fill=black] (\x,0) circle (2pt);

		\foreach \x in {-3,...,4}
			\fill[fill=blue] (\x,-1) circle (2pt);

    	\foreach \y in {-2, ...,-3}
    		\foreach \x in {-3,...,4}
    			\fill[fill=black] (\x,\y) circle (2pt);

    	\fill[fill=black] (0,1) circle (2pt);
    	\fill[fill=black] (1,1) circle (2pt);

    	\fill[fill=black] (1,1.5) circle (2pt);
    	\fill[fill=black] (0,1.5) circle (2pt);
    \end{tikzpicture}
    \caption{Possible locking points in the lower half-plane in blue.
    	$P$ is disjoint from the red regions as otherwise it would pick up interior lattice points away from the $x$-axis.\label{fig:possible-locking-pts}
}
\end{figure}
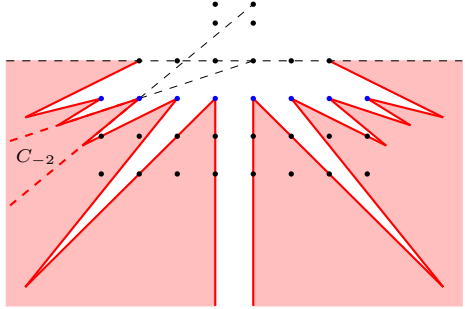
By using the symmetry induced by reflecting about the vertical axis $\set{x=\frac12}$, it straightforwardly follows that the possible lattice points that can be contained in $P$ and that lie in the lower half-plan $\set{y<0}$ can only be those from above at height $y=-1$.

These ten lattice points (two at height $y=1$ and eight at height $y=-1$; shown in red in \Cref{fig:possibilities_P'}) are therefore the only lattice points away from the $x$-axis which can be contained in $P$, and thus are the possible locking points.
Next we show that the assumption $\width(P)>3$ implies that at most four of them can be in $P$, more precisely, at most two of the points at height $y=-1$ can be in $P$.
Recall our (additional) assumption that the polygon $P$ is inclusion-maximal $\ZZ$-$\Delta_2$-free.

\begin{lemma}\label{lem:more_than_two_points}
	If $P\cap \{y=-1\}$ contains at least $3$ lattice points, then $\width(P)\le3$.
\end{lemma}
In the following, let us denote the basis of $(\ZZ^2)^*$ dual to $e_1, e_2$ by $e_1^*, e_2^*$.
\begin{proof}
    Suppose $P$ contains at least $3$ lattice points whose $y$-coordinate equals $-1$.
    Then $P$ is contained in the half-space $\{y\ge-1\}$, since any point outside this half-space would force the middle lattice point to be in the interior of $P$, which is not allowed.
    In particular, there is a facet $F$ supported by $\mleft\{y=-1\mright\}$.
    All other facets of $P$ are locked (see \Cref{prop:locked_incl_max_delta2_free}), and since they cannot be locked by points on  $\mleft\{y=-1\mright\}$, they can only be locked by $e_2$ and $e_1+e_2$.
    Thus $P$ is a triangle with one facet $F$, one facet through $e_2$ and another through $e_1+e_2$.
    The latter two facets intersect in the vertex $p_0$ from above.
    
    If $P$ is contained within the strip $\set{-1\le y \le 2}$, then $\width(P)\le3$.
    If not, we have $p_0 \in \set{y>2}$.
    Let us consider the facet $F$ of length $b$ as the base of the triangle $P$.
    With respect to this base, $P$ has height $h>3$.
    The triangle $\conv(e_2, e_1+e_2, p_0)$ is similar to $P$ and has base of length $1$ and height $h-2$.
    From the assumption $h>3$ and the equality (compare with the intercept theorem)
    \[
	    \frac{h-2}1=\frac hb\text{,}
    \]
    we obtain $b<3$.
    Note $b$ is the length of the facet $F$.
    Since $p_0$ lies in $\set{0<x<1}$ and the facets of $P$ that intersect at $p_0$ pass through $e_2$ and $e_1+e_2$ respectively, these facets must have positive and negative slope respectively.
    This means that the width of $P$ with respect to the linear form $e_1^*$ is $b$, the length of the facet $F$, and thus is less than $3$.
    Hence $P$ has width at most $3$.
\end{proof}

We have narrowed down the possible lattice points of $P$ that lie off the $x$-axis: up to two consecutive points from the set $\{-3e_1-e_2, -2e_1-e_2,-e_1-e_2, -e_2, e_1-e_2, 2e_1-e_2, 3e_1-e_2, 4e_1-e_2\}$, and points from $\{e_2, e_1+e_2\}$ (see the red points in \Cref{fig:possibilities_P'}). 
Since $P$ is a full dimensional polygon, it has at least three facets, and since each of these facets are locked, at least three points of the above are contained in $P$.
Further, we have observed that we can pick at most four of the above points (see \Cref{lem:more_than_two_points}), and thus the polygon $P$ is either a triangle (containing one lattice point on $y=-1$ and two on $y=1$, or vice versa) or a quadrilateral (containing two lattice points at $y=1$ and two at $y=-1$). 

In the following sections we treat the two cases separately.
In fact, we further subdivide the case where $P$ is a quadrilateral into two subcases.
The four facets of a quadrilateral $P$ are locked by both lattice points $e_2$ and $e_1+e_2$ on $y=1$ and by two consecutive lattice points on $y=-1$.
Since four is the maximum number of locking points the polygons under consideration can have, the relative interior of each facet of a quadrilateral $P$ contains exactly one locking point.
The convex hull $P'$ of the four locking points intersect the $x$-axis in a line segment $S = P'\cap\{y=0\}$ of length $1$, which implies that $S$ contains a lattice point.
Thus $P'$ can either be unimodularly equivalent to a rectangle, if the endpoints of $S$ are lattice points, or to a cross-polygon, if the endpoints of $S$ are not lattice points, and therefore $S$ contains exactly one lattice point in its interior.
Recall that a cross-polygon is $\ZZ$-unimodularly equivalent to $\conv(\pm e_1, \pm e_2)$.
Furthermore, the rectangle will always be $\ZZ$-unimodularly equivalent to $\conv(\pm e_2,e_1\pm e_2)$.
\Cref{fig:possibilities_P'} illustrates the possibilities for the convex hull $P'$ of the locking points.
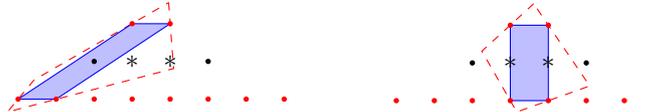
\begin{figure}[!ht]
    \centering
    \begin{tikzpicture}[scale=.5]
        \draw[red, dashed] (0.96, 1.56)  -- (1.09, -0.19) --(-3.28,-1.34) --(-2.58,-.49) -- cycle;  
            
        \fill[draw=blue,fill=blue!25] (1, 1) -- (0,1) -- (-3, -1) -- (-2, -1) -- (1, 1);    
        
        \fill[red] (0,1) circle (2pt);
        \fill[red] (1,1) circle (2pt);
        \fill (-1,0) circle (2pt);
        \node at (0,0) {$*$};
        \node at (1,0) {$*$};
        \fill (2,0) circle (2pt);
        \foreach \x in {-3,...,4} \fill[red] (\x,-1) circle (2pt);
    \end{tikzpicture}\qquad\qquad
    \begin{tikzpicture}[scale=.5]
        \draw[red, dashed] (0.615031,1.57022)  -- (-.75, .31) --(.17,-1.3) --(2.09,-.61) -- cycle;  
        
        \fill[draw=blue,fill=blue!25] (1, 1) -- (0, 1) -- (0, -1) -- (1, -1) -- (1, 1);    
        
        \fill[red] (0,1) circle (2pt);
        \fill[red] (1,1) circle (2pt);
        \fill (-1,0) circle (2pt);
        \node at (0,0) {$*$};
        \node at (1,0) {$*$};
        \fill (2,0) circle (2pt);
        \foreach \x in {-3,...,4} \fill[red] (\x,-1) circle (2pt);
    \end{tikzpicture}
    \caption{The two lattice points required to be in the interior of $P$ are denoted by ``$*$''.
        The red lattice points are the possible locking points of $P$.
        Possible convex hulls $P'$ of locking points (when $P$ is a quadrilateral) are drawn in blue.
        The left $P'$ is unimodularly equivalent to a cross-polygon while the right $P'$ is unimodularly equivalent to a rectangle.
        The red dashed polygons are examples of quadrilaterals circumscribed around $P'$.\label{fig:possibilities_P'}
}
\end{figure}
In what follows, we will often say that ``the polygon $P$ is circumscribed around $P'$''.
The precise definition is the following.
\begin{definition}
    Let $P, P' \subset \RR^2$ be polygons.
    We say that $P$ is \novel{circumscribed} around $P'$ if each vertex of $P'$ is contained in a facet of $P$, and each facet of $P$ contains a vertex of $P'$.
\end{definition}


\subsection{Triangles}\label{sec:triangles}
We first consider the case where $P$ is a triangle.
Since all facets of $P$ are locked, there are at least three locking points.
Thus, given three lattice points $A, B$ and $C$, two on the line $y=1$ and one on the line $y=-1$, or vice versa, we want to show that $\frac{10}{3}$ is an upper bound for the width of any $\ZZ$-$\Delta_2$-free triangle $P$ with facets locked by $A$, $B$ and $C$.
We do not require that the points $A,B$ and $C$ are the only locking points: a facet of $P$ might contain another lattice point in its interior. 

There are 22 possible triples $(A,B,C)$: $8$ have the two lattice points on the line $y=1$ and one on $y=-1$, while $14$ have one lattice point on $y=1$ and two consecutive ones on $y=-1$.
Reflecting along the line $x=\frac12$ yields the same result, so only $11$ triples need to be checked.
In \Cref{table:triangles} (below), we record these $11$ triples.
The cases 8--11 are not admissible: $P$ is assumed to contain a point $p_0$ in $\set{y>\frac32}$, and thus two points of the triple $(A,B,C)$ need to lock the two facets through this point. However, in cases 8--11 this is impossible while also guaranteeing that $\mathbf{0}$ and $e_1$ are in the interior of $P$.
For the remaining cases, we relax our assumptions and only require that the facets of $P$ are locked by the triple $(A,B,C)$, $P$ contains $\mathbf{0}, e_1$ and it does not contain any lattice points away from the $x$-axis in its interior. That is, we forget the requirement that $P$ contains a point $p_0=xe_1+ye_2$ with $0<x<1$ and $y>\frac32$, and we allow $\mathbf{0}, e_1$ to lie in the boundary of $P$. 
Note this relaxation is possible as long as $\frac{10}3$ is an upper bound for any $P$ satisfying the relaxed conditions. Our computations show that this is the case. This relaxation comes with two advantages, namely 1) the constraints on the polygon $P$ are simplified; 2) it allows further symmetry (unimodular transformations fixing the $x$-axis, such as reflection about or shearing along the $x$-axis) resulting into a reduction of cases.
In particular, the second point allows us to treat the following pairs of cases as being equivalent: $1 \sim 7$, $2 \sim 6$, and $3 \sim 5$.
Finally, because of this relaxation, \Cref{table:triangles} includes upper bounds that are smaller than $3$: in these cases, the largest width in the relaxed conditions was only achieved by a triangle $P$ with $e_1$ on the boundary.

In order to bound the width of polygons whose facets are locked by a fixed triple $(A,B,C)$, we employ a computer assisted strategy together with an approach that Hurkens has used to compute the classical flatness constant in $2$ dimensions~\cite{Hurkens}.
Let $X, Y$ and $Z$ be the vertices of the triangle $P$.
We consider $A, B, C, X,Y,Z$ as row vectors and write:
\[
    \renewcommand{\arraystretch}{1.3}
    \begin{bmatrix}A\\ B\\ C\end{bmatrix} = \begin{bmatrix}
         0 & \lambda & \bar{\lambda}\\
         \bar{\mu} & 0 & \mu \\
         \nu & \bar{\nu} & 0
    \end{bmatrix} \begin{bmatrix}X\\ Y\\ Z\end{bmatrix}
\]
for some $\lambda, \mu, \nu \in [0,1]$, with $\bar{\lambda}+\lambda=\bar{\mu}+\mu=\bar{\nu}+\nu = 1$.
Inverting the matrix, we obtain
\[
    \renewcommand{\arraystretch}{1.3}
    \begin{bmatrix}X\\ Y\\ Z\end{bmatrix} = \frac1{\lambda \mu \nu + \bar{\lambda} \bar{\mu} \bar{\nu}} \begin{bmatrix}
         -\mu \bar{\nu} & \bar{\lambda}\bar{\nu} & \lambda\mu \\
         \mu\nu &  -\bar{\lambda}\nu& \bar{\lambda}\bar{\mu} \\
         \bar{\mu}\bar{\nu} & \lambda\nu & -\lambda\bar{\mu}
    \end{bmatrix} \begin{bmatrix}A\\ B\\ C\end{bmatrix} \text{.}
\]
Since $A, B$ and $C$ are fixed, these are formulas for $X, Y$ and $Z$ in terms of the parameters $\lambda$, $\mu$, and $\nu$.
In fact, the coordinates of the pairwise differences of the vertices $X$, $Y$ and $Z$ are rational functions with a \emph{linear} numerator and denominator equal to $\lambda \mu \nu + \bar{\lambda} \bar{\mu} \bar{\nu}$:
\begin{equation}\label{eq:differences}
    \renewcommand{\arraystretch}{1.3}
    \begin{bmatrix}X-Y \\ Y-Z \\ Z-X\end{bmatrix} = \frac{1}{\lambda \mu \nu + \bar{\lambda} \bar{\mu} \bar{\nu}} \begin{bmatrix}
         -\mu & \bar{\lambda} & -1+\lambda+\mu \\
         -1+\mu+\nu &  -\nu& \bar{\mu} \\
         \bar{\nu} & -1+\lambda+\nu & -\lambda
    \end{bmatrix} \begin{bmatrix}A \\ B\\ C\end{bmatrix} \text{.}
\end{equation}
Thus the slopes of the facets of $P$ are rational functions with linear numerator and linear denominator in terms of the parameters $\lambda$, $\mu$, and $\nu$.
The conditions $\mathbf{0}, e_1 \in P$ and that the interior $\strint(P)$ of $P$ is disjoint from the lattice points off the $x$-axis constrain the possible slopes.
In terms of the parameters $\lambda$, $\mu$, and $\nu$, these constraints are linear.
We thus obtain a polytope $Q \subset [0,1]^3$ of admissible $\lambda$, $\mu$, and $\nu$.

Next, we express the width of $P$ in a chosen direction in terms of the parameters $\lambda$, $\mu$, and $\nu$.
Clearly, for a fixed direction, the width can be achieved on any pair of vertices depending on $\lambda$, $\mu$, and $\nu$.
Wherever it is achieved by the same two vertices, the width is a linear function (in $\lambda$, $\mu$, and $\nu$) divided by $\delta:=\lambda \mu \nu + \bar{\lambda} \bar{\mu} \bar{\nu}$. 
Our strategy includes choosing directions ``ad hoc'' such that:
\begin{enumerate}[label=\arabic*)]
\item
    The width in one such direction is achieved on the same pair of vertices for all parameters $\lambda, \mu, \nu \in Q$, so that the width of $P$ accepts an upper bound of the form
    \[
    	\frac{\min \set{\ell_1(\lambda, \mu, \nu), \dots, \ell_r(\lambda, \mu, \nu)}}{\delta} \qquad \text{where the $\ell_i$'s are linear forms in terms of $\lambda$, $\mu$, and $\nu$.}
    \]
\item
    The maximum over $Q$ of this function is at most $\frac{10}{3}$.
    The actual computations are carried out with polymake~\cite{polymake} and Mathematica~\cite{Mathematica}.
\end{enumerate}

For the cases where the upper bound of $\frac{10}3$ is achieved, we determine the respective extremal points $(\lambda, \mu, \nu) \in Q$ and show that these parameters correspond to triangles $P$ which are unimodularly equivalent to $\conv\mleft(\frac13 e_1 +\frac53 e_2, -\frac43 e_1 - \frac53 e_2, 2e_1\mright)$.
\begin{table}[!ht]
    \begin{tabular}{ccccc}
    \toprule
    & $\conv(\mathbf{0},e_1,A,B,C)$ & width & \parbox{2.5cm}{\centering width directions} & vertices of maximiser\\
    \midrule
    1 & \parbox{2.6cm}{\begin{tikzpicture}[scale=.5]
        \fill[draw=blue,fill=blue!25] (0, -1) -- (1, 0) -- (1, 1) -- (0, 1) -- (0, -1);
        \fill (0,1) circle (2pt);
        \fill (1,1) circle (2pt);
        \fill (-1,0) circle (2pt);
        \node at (0,0) {$*$};
        \node at (1,0) {$*$};
        \foreach \x in {-3,...,2} \fill (\x,-1) circle (2pt);
    \end{tikzpicture}} & $\leq \frac{10}3$& $e_1^*, e_2^*, e_1^*-e_2^*$ &
    $\frac13\begin{bmatrix}
    	-4 & 1 & 6 \\
		-5 & 5 & 0
    \end{bmatrix}$\\
    \midrule
    2 & \parbox{2.6cm}{\begin{tikzpicture}[scale=.5]
        \fill[draw=blue,fill=blue!25] (1, 0) -- (1, 1)-- (0, 1) -- (-1, -1) -- (1, 0);
        \fill (0,1) circle (2pt);
        \fill (1,1) circle (2pt);
        \fill (-1,0) circle (2pt);
        \node at (0,0) {$*$};
        \node at (1,0) {$*$};
        \foreach \x in {-3,...,2} \fill (\x,-1) circle (2pt);
    \end{tikzpicture}}& $\leq \frac2{\sqrt7-2}$ & $e_1^*, e_2^*, e_1^*-e_2^*$ & 
    $\scriptstyle\frac1{3\sqrt{7}}\begin{bmatrix}
        \scriptstyle-6-3\sqrt{7} & \scriptstyle8+\sqrt{7} & \scriptstyle1+2\sqrt{7}\\
        \scriptstyle-9 & \scriptstyle-2-\sqrt{7} & \scriptstyle5+4\sqrt{7}
    \end{bmatrix}$\\
    \midrule
    3 & \parbox{2.6cm}{\begin{tikzpicture}[scale=.5]
        \fill[draw=blue,fill=blue!25] (0, 1) -- (-2, -1)  -- (1, 0) -- (1, 1) -- (0, 1);
        \fill (0,1) circle (2pt);
        \fill (1,1) circle (2pt);
        \fill (-1,0) circle (2pt);
        \node at (0,0) {$*$};
        \node at (1,0) {$*$};
        \foreach \x in {-3,...,2} \fill (\x,-1) circle (2pt);
    \end{tikzpicture}} & $< \frac{10}{3}$ & $e_2^*,e_1^*-e_2^*$ &
    $\frac13\begin{bmatrix}
    	-12 & 3 & 3\\
		-5 & 5 & 0
    \end{bmatrix}$\\
    \midrule
    4 & \parbox{2.6cm}{\begin{tikzpicture}[scale=.5]
        \fill[draw=blue,fill=blue!25]  ( 1,  0) -- (-3, -1) -- ( 0,  1) -- ( 1,  1)-- cycle;
        \fill (0,1) circle (2pt);
        \fill (1,1) circle (2pt);
        \fill (-1,0) circle (2pt);
        \node at (0,0) {$*$};
        \node at (1,0) {$*$};
        \foreach \x in {-3,...,2} \fill (\x,-1) circle (2pt);
    \end{tikzpicture}}& $<3$ & $e_2^*$ &
    $\frac12\begin{bmatrix}
    	-10 & 2 & 2\\
		-3 & 3 & 0
    \end{bmatrix}$\\
    \midrule
    \multicolumn{5}{c}{Equivalent cases}\\
    \multicolumn{5}{c}{
	5 \parbox{3cm}{\centering \begin{tikzpicture}[scale=.5]
        \fill[draw=blue,fill=blue!25]   ( 1,  0) -- (-2, -1) -- (-3, -1) --  ( 1,  1)-- cycle;
        \fill (0,1) circle (2pt);
        \fill (1,1) circle (2pt);
        \fill (-1,0) circle (2pt);
        \node at (0,0) {$*$};
        \node at (1,0) {$*$};
        \foreach \x in {-3,...,2} \fill (\x,-1) circle (2pt);
    \end{tikzpicture}\\
    equivalent to case 3}\qquad
    6 \parbox{3cm}{\centering \begin{tikzpicture}[scale=.5]
        \fill[draw=blue,fill=blue!25] (1, 1) -- (-2, -1) -- (-1, -1) -- (1, 0) -- (1, 1);
        \fill (0,1) circle (2pt);
        \fill (1,1) circle (2pt);
        \fill (-1,0) circle (2pt);
        \node at (0,0) {$*$};
        \node at (1,0) {$*$};
        \foreach \x in {-3,...,2} \fill (\x,-1) circle (2pt);
    \end{tikzpicture}\\
    equivalent to case 2}\qquad
    7 \parbox{3cm}{\centering \begin{tikzpicture}[scale=.5]
        \fill[draw=blue,fill=blue!25] (1, 0) -- (1, 1) -- (-1, -1) -- (0, -1) -- (1, 0);
        \fill (0,1) circle (2pt);
        \fill (1,1) circle (2pt);
        \fill (-1,0) circle (2pt);
        \node at (0,0) {$*$};
        \node at (1,0) {$*$};
        \foreach \x in {-3,...,2} \fill (\x,-1) circle (2pt);
    \end{tikzpicture}\\
    equivalent to case 1}\vspace*{1ex}
    }\\
    \midrule
    \multicolumn{5}{c}{Not admissible triples $(A,B,C)$}\\[1ex]
    \multicolumn{5}{l}{8 \parbox{2.6cm}{\begin{tikzpicture}[scale=.5]
        \fill[draw=blue,fill=blue!25] (0, 0) -- (1, 1)  -- (1, -1) -- (0,-1) -- (0, 0);
        \fill (0,1) circle (2pt);
        \fill (1,1) circle (2pt);
        \fill (-1,0) circle (2pt);
        \node at (0,0) {$*$};
        \node at (1,0) {$*$};
        \foreach \x in {-3,...,2} \fill (\x,-1) circle (2pt);
    \end{tikzpicture}}\qquad
    9 \parbox{2.6cm}{\begin{tikzpicture}[scale=.5]
        \fill[draw=blue,fill=blue!25] (1, 0) -- (0, 1) -- (-1, -1) -- (0, -1) -- (1, 0);
        \fill (0,1) circle (2pt);
        \fill (1,1) circle (2pt);
        \fill (-1,0) circle (2pt);
        \node at (0,0) {$*$};
        \node at (1,0) {$*$};
        \foreach \x in {-3,...,2} \fill (\x,-1) circle (2pt);
    \end{tikzpicture}}\qquad
    10 \parbox{2.6cm}{\begin{tikzpicture}[scale=.5]
        \fill[draw=blue,fill=blue!25] (0, 1) -- (-2, -1) -- (-1, -1) -- (1, 0) -- (0, 1);
        \fill (0,1) circle (2pt);
        \fill (1,1) circle (2pt);
        \fill (-1,0) circle (2pt);
        \node at (0,0) {$*$};
        \node at (1,0) {$*$};
        \foreach \x in {-3,...,2} \fill (\x,-1) circle (2pt);
    \end{tikzpicture}}\qquad
    11 \parbox{2.6cm}{\begin{tikzpicture}[scale=.5]
        \fill[draw=blue,fill=blue!25]   ( 1,  0) -- (-2, -1) -- (-3, -1) --  ( 0,  1)-- cycle;
        \fill (0,1) circle (2pt);
        \fill (1,1) circle (2pt);
        \fill (-1,0) circle (2pt);
        \node at (0,0) {$*$};
        \node at (1,0) {$*$};
        \foreach \x in {-3,...,2} \fill (\x,-1) circle (2pt);
    \end{tikzpicture}}}\vspace*{1ex}\\
    \bottomrule
    \end{tabular}
    \caption{Up to symmetry, all possible triples of locking points $(A, B, C)$; we list the largest possible width a triangle circumscribed around them can have, the directions in which this width is achieved, and the vertices of the triangle of largest width.\label{table:triangles}
}
\end{table}

To illustrate our strategy, let us work out the details for one choice of locking points, namely $(A,B,C) = (e_1+e_2, -e_2, e_2)$ (case $1$ in \Cref{table:triangles}). 
The other cases 2--4 work similarly.
By~\eqref{eq:differences}, we have
\[
    X-Y = \mleft(-\frac{\mu}{\delta}, \frac{-2+2\lambda}{\delta}\mright), \qquad
    Y-Z = \mleft( \frac{-1+\mu+\nu}{\delta}, \frac{2\nu}{\delta}\mright), \qquad
    Z-X = \mleft(\frac{1-\nu}{\delta},\frac{2-2\lambda-2\nu}{\delta}\mright).
\]
\begin{figure}[!ht]
    \centering
    \begin{tikzpicture}
        \draw[red] (2,0) -- (1/3, 5/3) -- (-4/3,-5/3) -- cycle;
        \node at (2.3,.1) {$Z$};
        \node at (.45,1.81) {$Y$};
        \node at (-1.5,-1.7) {$X$};
        \fill[draw=blue,fill=blue!25] (0, -1) -- (1,0) -- (1, 1) -- (0, 1) -- (0, -1);        

		\node[above right] at (1,1) {$A$};
		\node[below right] at (0,-1) {$B$};
		\node[above left] at (0,1) {$C$};

        \fill (0,1) circle (2pt);
        \fill (1,1) circle (2pt);
        \fill (-1,0) circle (2pt);
        \fill (2,0) circle (2pt);
        \node at (0,0) {$*$};
        \node at (1,0) {$*$};
        \foreach \x in {-3,...,4} \fill (\x,-1) circle (2pt);
    \end{tikzpicture}
    \caption{A triangle $P$ (in red) with locking points $(A,B,C) = \mleft(e_1+e_2,-e_2,e_2\mright)$. This is the unique $\ZZ$-$\Delta_2$-free triangle of width $\frac{10}{3}$.\label{fig:triangle_case1}
}
\end{figure}
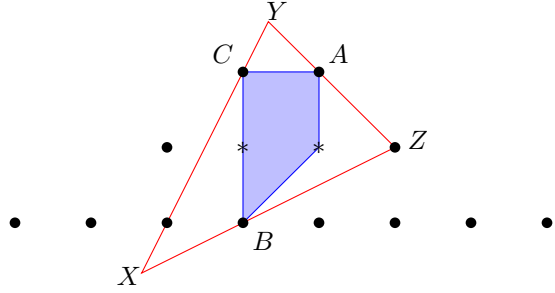

First we determine the polytope $Q$ of admissible parameters.
The slopes $m_{XY}$, $m_{YZ}$ and $m_{ZX}$ of the facets of $P$ through $\set{X,Y}$, $\set{Y,Z}$ and $\set{Z,X}$ respectively can be expressed in terms of $\lambda$, $\mu$, and $\nu$: 
\[
    m_{XY} = \frac{2-2\lambda}{\mu}, \qquad m_{YZ} = \frac{2\nu}{-1+\mu+\nu}, \qquad m_{ZX} = \frac{2-2\lambda-2\nu}{1-\nu}\text{.}
\]

The position of the vertices $X$, $Y$, and $Z$ of $P$ is constrained by the two assumptions that $P$ contains $\mathbf{0}, e_1$ and that no other lattice point away from the $x$-axis is in the interior of $P$.
Since $\mathbf{0} \in P$, we have $m_{XY}\ge0$, while $e_1 \in P$ yields $m_{YZ}\le0$ and $m_{ZX}\le1$.
Since $e_1-e_2$ is not in the interior of $P$, we have $m_{ZX}\ge0$.
Similarly, since $-e_1-e_2$ is not in the interior of $P$, we have $m_{XY}\ge2$. In a similar way, any further lattice point off the $x$-axis would give us other constraints. Many will be redundant, but some might further restrict the polytope $Q$ of admissible parameters. However, we do not need all constraints and we can stop once we have enough to obtain an upper bound not exceeding $\frac{10}{3}$. In this case, the following constraints on the slopes are enough:
\[
    m_{XY}\ge 2, \qquad m_{YZ}\le 0, \qquad 0 \le m_{ZX} \le 1 \text{.}
\]

Arithmetic manipulation of these inequalities yields constraints on the parameters $\lambda$, $\mu$, and $\nu$ which define the polytope of admissible parameters 
\[
    Q = \setcond{(\lambda, \mu, \nu) \in [0,1]^3}{1-\lambda-\mu \geq 0, 1-\mu-\nu \geq 0, 1-\lambda-\nu \geq 0, -1+2\lambda+\nu \geq 0} \text{.}
\]

We now determine the widths of $P$ in the directions $e^*_1$, $e^*_2$, and $e^*_1-e^*_2$.
On $Q$, these are achieved at $Z-X$, $Y-X$, and $Z-Y$ respectively:
\begin{align*}
    \width_{e^*_1}(P) &= e^*_1( Z-X ) = \frac{1-\nu}\delta\\
    \width_{e^*_2}(P) &= e^*_2( Y-X ) = \frac{2-2\lambda}\delta\\
    \width_{e^*_1-e^*_2}(P) &= (e^*_1-e^*_2)( Z-Y ) = \frac{1-\mu+\nu}\delta \text{.}
\end{align*}

We thus obtain
\begin{align*}
    \width(P) &\leq \min\set{\width_{e^*_1}(P), \width_{e^*_2}(P),  \width_{e^*_1-e^*_2}(P)}\\
    &= \frac{\min\set{1-\nu, 2-2\lambda, 1-\mu+\nu}}\delta\eqqcolon \frac{f(\lambda,\mu,\nu)}\delta \text{.} 
\end{align*}

We denote the numerator of the last expression $f(\lambda, \mu, \nu)$.
To show that any admissible triangle $P$ has width at most $\frac{10}{3}$, it suffices to verify
\[
    \max_{(\lambda, \mu, \nu) \in Q} f(\lambda, \mu, \nu) \leq \frac{10}3\text{.}
\]
To do so, we note that $f(\lambda, \mu, \nu)$ is a tropical polynomial, and using polymake, we calculate its \emph{regions of linearity}, which when intersected with $Q$ gives polytopes $Q_i$ over which $f$ coincides with a linear function $f_i$ (for further details on tropical geometry, we refer to~\cite{ETC}).
Using Mathematica~\cite{Mathematica}, for each $i$, we compute the maximum of the rational function $\frac{f_i}\delta$ over the region $Q_i$.
In this way, we verify that in this case there is exactly one point in $Q$ at which the maximum $\frac{10}3$ is achieved, namely at $(\lambda, \mu, \nu)=\mleft(\frac25, \frac15,\frac45\mright)$.
For these values, the corresponding triangle $P$ is exactly the triangle depicted in \Cref{fig:triangle_case1}, which will turn out to be the only $\ZZ$-$\Delta_2$-free polygon of width $\frac{10}3$, as stated in \Cref{prop:locked_incl_max_delta2_free}.


\subsection{Quadrilateral circumscribed around a rectangle}
Next, we consider the case where $P$ is a quadrilateral and the convex hull of its locking points $P'$ is $\ZZ$-unimodularly equivalent to a lattice rectangle with area equal to two.
Then, we can assume that $P' = \conv(\pm e_2, e_1 \pm e_2)$.

First, observe that the vertices of $P$ are in the vertical strip $0 < x < 1$ or in the horizontal strip $-1 < y < 1$.
Indeed, any point strictly outside of both strips forces one of the locking points to be in the interior of $P$, a contradiction.
Furthermore, a vertex on the boundary of a strip forces two facets to coincide, and thus $P$ to be a triangle, a case which was dealt with already in \Cref{sec:triangles}.
Thus, one vertex of $P$ lies in each of the four connected components of the union of the two strips minus the rectangle $P'$.

Note two vertices of $P$ in, say, the horizontal strip, one on each side of $P'$, completely determine $P$.
Let us denote those two vertices by $-\kappa e_1+\lambda e_2$ and $(\mu+1)e_1+\nu e_2$, with $\kappa, \mu >0$ and $-1<\lambda, \nu<1$.
Then the lines supporting the edges of $P$ are 
\[
    y=\frac{1-\lambda}\kappa x+1, \qquad y=-\frac{1+\lambda}\kappa x-1,\qquad y=-\frac{1-\nu}\mu x +\frac{1-\nu}\mu +1, \qquad y=\frac{1+\nu}\mu x - \frac{1+\nu}\mu-1
\]
and thus the two remaining vertices of $P$ are given by
\begin{equation}\label{eq:rect-top-bottom-vertex}
	\begin{aligned}
    &\mleft(\frac{(1-\nu)\kappa}{(1-\nu)\kappa + (1-\lambda)\mu}, \frac{(1-\nu)(1-\lambda)}{(1-\nu)\kappa + (1-\lambda)\mu}+1\mright) \qquad \text{and}\\
    &\mleft(\frac{(1+\nu)\kappa}{(1+\nu)\kappa+(1+\lambda)\mu},\frac{-(1+\nu)(1+\lambda)}{(1+\nu)\kappa+(1+\lambda)\mu}-1\mright).
    \end{aligned}
\end{equation}
Clearly the width in the horizontal direction is $\width_h = \kappa+\mu+1$, while from the previous formulae we obtain the width in the vertical direction: $\width_v = 2 +\frac{(1+\lambda)(1+\nu)}{(1+\lambda)\mu+(1+\nu)\kappa} + \frac{(1-\lambda)(1-\nu)}{(1-\lambda)\mu+(1-\nu)\kappa}$.

We first show if $\kappa, \mu$ and $\lambda$ are fixed, the maximum vertical width is attained for $\nu=\lambda$.
To that end compute the partial derivative of $\width_v$ with respect to $\nu$:
\[
    \frac{\partial}{\partial \nu}\width_v=-\frac{4\kappa\mu\cdot (\lambda-\nu)\cdot(\kappa(\lambda\nu-1)+\mu(\lambda^2-1))}{((1-\nu)\kappa+(1-\lambda)\mu)^2\cdot((1+\nu)\kappa + (1+\lambda)\mu)^2} \text{.}
\]
It is straightforward to verify that on our domain $\kappa,\mu>0$ and $-1<\lambda,\nu<1$ this partial derivative only vanishes at $\nu=\lambda$.
It is easy to check that this is indeed a maximum. 

We can thus focus on the case where $P$ has two horizontally aligned vertices $(\kappa,\lambda)$ and $(\mu,\lambda)$.
From Formulae~\eqref{eq:rect-top-bottom-vertex}, it readily follows that the top and bottom vertices are vertically aligned (see \Cref{fig:kite}).
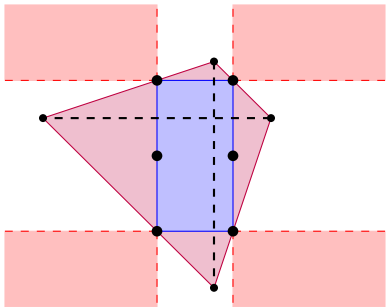
\begin{figure}[!ht]
    \centering
    \begin{tikzpicture}
            \fill[draw=purple,fill=purple!25] (.75,-1.75)-- (-1.5,.5)  -- (.75,1.25) -- (1.5,.5)  -- cycle;

            \fill[draw=blue,fill=blue!25] (0,1) -- ( 1,  1) -- (1, -1) -- ( 0, -1) -- cycle;
            \draw[thick,dashed] (1.5,.5)  -- (-1.5,.5) ;
            \draw[thick,dashed] (.75,1.25)  -- (.75,-1.75);
        
            \fill[draw=red!25,fill=red!25] (-2,1) --  (0,1) -- (0, 2) -- (-2,2) -- cycle;
            \draw[dashed,red] (0,1) -- (0,2);
            \draw[dashed,red] (0,1) -- (-2,1);

            \fill[draw=red!25,fill=red!25] (-2,-1) --  (0,-1) -- (0, -2) -- (-2,-2) -- cycle;
            \draw[dashed,red] (0,-1) -- (0,-2);
            \draw[dashed,red] (0,-1) -- (-2,-1);

            \fill[draw=red!25,fill=red!25] (3,-1) --  (1,-1) -- (1, -2) -- (3,-2) -- cycle;
            \draw[dashed,red] (1,-1) -- (1,-2);
            \draw[dashed,red] (1,-1) -- (3,-1);
            
            \fill[draw=red!25,fill=red!25] (3,1) --  (1,1) -- (1, 2) -- (3,2) -- cycle;
            \draw[dashed,red] (1,1) -- (1,2);
            \draw[dashed,red] (1,1) -- (3,1);            
                 
            \fill (0,1) circle (2pt);
            \fill (1,1) circle (2pt);
            \fill (0,0) circle (2pt);
            \fill (1,0) circle (2pt);
            \fill (0,-1) circle (2pt);
            \fill (1,-1) circle (2pt);
            
            \fill (-1.5,.5) circle (1.5pt);
            \fill (1.5,.5) circle (1.5pt);
            \fill (.75,1.25) circle (1.5pt);
            \fill (.75,-1.75) circle (1.5pt);
            
        \end{tikzpicture}
    \caption{The maximum width in the vertical direction is achieved when the vertices of the circumscribed quadrilateral are horizontally and vertically aligned.\label{fig:kite}
}
\end{figure}
Let $\zeta+1$ be the $y$-coordinate of the top vertex of $P$ and $-\xi-1$ be that of the bottom vertex.
We calculate the area $A$ of $P$ in two different ways.
Since the diagonals of $P$ are orthogonal, we have $A=\frac{(\kappa+\mu+1)(\zeta+\xi+2)}{2}$ where $\kappa+\mu+1$ respectively $\zeta+\xi+2$ are the lengths of the horizontal and vertical diagonal.
On the other hand, $P$ can be decomposed into the union of $P'$ and four triangles, each sharing an edge with $P'$ and a vertex with $P$.
The sum of the areas of these pieces gives $A= 2 + \kappa+\mu+ \frac{\zeta+\xi}{2}$.
These two expressions for the area $A$ of $P$ yield the equation $(\kappa+\mu)(\zeta+\xi)=2$.
Thus if $\kappa+\mu>2$ then $\zeta+\xi <1$.
We conclude by observing that $\kappa+\mu>2$ is equivalent to the horizontal width being greater than $3$ and $\zeta+\xi>1$ is equivalent to the vertical width being greater than $3$.
Since these conditions cannot happen at the same time, the width of $P$ is at most $3$.


\subsection{Quadrilateral circumscribed around a cross-polygon}
Let $\Diamond_2 \subset \RR^2$ be the $2$-dimensional cross-polygon, i.e., $\Diamond_2=\conv(\pm e_1, \pm e_2) \subset \RR^2$.
Here we are going to bound the width of inclusion-maximal $\ZZ$-$\Delta_2$-free quadrilaterals $P$ circumscribed around $\Diamond_2$.
As above, it suffices to consider such quadrilaterals whose width is greater than $3$.
We begin with some preliminary observations.

Switching to the 2-dimensional lattice generated by $f_1 \coloneqq (1,1)$ and $f_2 \coloneqq (1,-1)$ results in more manageable equations for the widths of $P$.
Let us denote the basis dual to the basis $f_1, f_2$ of the new lattice by $f_1^*,f_2^*$.
Explicitly, $f_1^* = \frac12(e_1^* + e_2^*)$ and $f_2^* = \frac12(e_1^* - e_2^*)$.
In this lattice, the cross-polygon which $P$ is circumscribed around has vertices $\pm f_1$ and $\pm f_2$.

Notice that $P$ has a vertex in each of the four following regions: $\pm\setcond{(x,y) \in \RR^2}{-1 < x < 1,\, y>1}$ and $\pm\setcond{(x,y) \in \RR^2}{x>1,\, -1 < y < 1}$ which we will refer to as the top, bottom, right, and left regions, respectively. 
We label the vertices of $P$ in the top and bottom regions as $Z = (\kappa,\lambda)$, $W = (\mu,\nu)$, respectively, where $-1 < \kappa,\mu < 1$ and $\lambda,-\nu > 1$.
The vertices in the right and left region are labeled $Y$, $X$ respectively.
Note that $P$ is completely determined by the parameters $\kappa, \lambda, \mu, \nu$ defining $Z$ and $W$.

By \Cref{prop:one_interior_point}, $P$ contains at least two interior lattice points since its width is assumed to be greater than $3$.
Clearly, $\mathbf{0}$ is one interior lattice point of $P$.
The second interior point could be either $\pm(2,0)$ or $\pm(0,2)$.
By symmetry, we may assume without loss of generality that $(0,2)$ is the other interior point.
This implies $\lambda > 2$ and that the top vertex $Z$ of $P$ is extremal in the sense that the width with respect to the directions $f_1^*$ and $f_2^*$, i.e., $\width_{f_i^*}(P)$, is attained at $Z$.
\begin{figure}[!ht]
    \centering
    \begin{tikzpicture}[scale=0.75]
     
 		\clip (-2.5,-3.5) rectangle (2.5,4.2);

        \fill[draw=gray!25,fill=gray!50] (0,2) -- (-1,3) -- (-1,4.5) -- (1,4.5) -- (1,3) -- cycle;
        \draw[dashed,gray] (0,2) -- (-1,3);
        \draw[dashed,gray] (0,2) -- (1,3);
        
        \fill[fill=gray!15] (-1,-1) -- (1,-1) -- (1,-3.5) -- (-1,-3.5) -- cycle;
        \draw[dashed,gray] (-1,-1) -- (1,-3);
        \draw[dashed,gray] (1,-1) -- (-1,-3);
        
        \node at (0,-1.4) {$A$};
        \node at (0.6,-2) {$C$};
        \node at (0,-2.6) {$B$};
        \node at (-0.6,-2) {$D$};
        
        \fill[draw=red!25,fill=red!25] (-2.5,1) -- (-1,1) -- (-1, 4.5) -- (-2.5,4.5) -- cycle;
        \draw[dashed,red] (-1,1) -- (-1,4.5);
        \draw[dashed,red] (-1,1) -- (-2.5,1);

        \fill[draw=red!25,fill=red!25] (2.5,1) -- (1,1) -- (1,4.5) -- (2.5,4.5) -- cycle;
        \draw[dashed,red] (1,1) -- (1,4.5);
        \draw[dashed,red] (1,1) -- (2.5,1);

        \fill[draw=red!25,fill=red!25] (-2.5,-1) --  (-1,-1) -- (-1, -3.5) -- (-2.5,-3.5) -- cycle;
        \draw[dashed,red] (-1,-1) -- (-1,-3.5);
        \draw[dashed,red] (-1,-1) -- (-2.5,-1);

        \fill[draw=red!25,fill=red!25] (2.5,-1) -- (1,-1) -- (1,-3.5) -- (2.5,-3.5) -- cycle;
        \draw[dashed,red] (1,-1) -- (1,-3.5);
        \draw[dashed,red] (1,-1) -- (2.5,-1);
        
        \fill[draw=blue,fill=blue!25] (1,1) -- (1,-1) -- (-1,-1) -- (-1,1) -- cycle;
        
        \fill (-2,4) circle (2pt);
        \fill (0,4) circle (2pt);
        \fill (2,4) circle (2pt);
        
        \fill (-1,3) circle (2pt);
        \fill (1,3) circle (2pt);
        
        \fill (-2,2) circle (2pt);
        \node at (0,2) {$*$};
        \fill (2,2) circle (2pt);
        
        \fill (-1,1) circle (2pt);
        \fill (1,1) circle (2pt);
        
        \fill (2,0) circle (2pt);
        \node at (0,0) {$*$};
        \fill (-2,0) circle (2pt);
        
        \fill (-1,-1) circle (2pt);
        \fill (1,-1) circle (2pt);
        
        \fill (-2,-2) circle (2pt);
        \fill (0,-2) circle (2pt);
        \fill (2,-2) circle (2pt);
        
        \fill (-1,-3) circle (2pt);
        \fill (1,-3) circle (2pt);
        
    \end{tikzpicture} \qquad \qquad
    \begin{tikzpicture}[scale=0.75]
     
 		\clip (-2.5,-3.5) rectangle (2.5,4.2);

        \fill[draw=gray!25,fill=gray!50] (0,2) -- (-1,3) -- (-1,4.5) -- (1,4.5) -- (1,3) -- cycle;
        \draw[dashed,gray] (0,2) -- (-1,3);
        \draw[dashed,gray] (0,2) -- (1,3);
        
        \fill[fill=gray!15] (1,-1) -- (0,-2) -- (1,-3) -- cycle;
        \draw[dashed,gray] (1,-1) -- (0,-2) -- (1,-3);
        
        \fill[draw=red!25,fill=red!25] (-3,-1) -- (-1,1) -- (-1, 4.5) -- (-3,4.5) -- cycle;
        \draw[dashed,red] (-1,1) -- (-1,4.5);
        \draw[dashed,red] (-1,1) -- (-3,-1);

        \fill[draw=red!25,fill=red!25] (3,-1) -- (1,1) -- (1,4.5) -- (3,4.5) -- cycle;
        \draw[dashed,red] (1,1) -- (1,4.5);
        \draw[dashed,red] (1,1) -- (3,-1);

        \fill[draw=red!25,fill=red!25] (-2.5,-1) --  (-1,-1) -- (-1, -3.5) -- (-2.5,-3.5) -- cycle;
        \draw[dashed,red] (-1,-1) -- (-1,-3.5);
        \draw[dashed,red] (-1,-1) -- (-2.5,-1);

        \fill[draw=red!25,fill=red!25] (2.5,-1) -- (1,-1) -- (1,-3.5) -- (2.5,-3.5) -- cycle;
        \draw[dashed,red] (1,-1) -- (1,-3.5);
        \draw[dashed,red] (1,-1) -- (2.5,-1);
        
        \fill[draw=blue,fill=blue!25] (1,1) -- (1,-1) -- (-1,-1) -- (-1,1) -- cycle;
                
        \draw (1/4,4) -- (-11/6,-1);
        \draw (1/4,4) -- (3/2,-1);
        
        \draw (1-1/4,-2.5) -- (-1.61403,-0.473672);
        \draw (1-1/4,-2.5) -- (6/5,1/5);
        
        \draw[very thin,draw=gray!120] (1/4-1, 4 - 20/3-1) -- (1/4 + 2+1, 6 - 20/3+1);
        \node at (-2.2,0.25) {$\scriptstyle H_1$};
        \draw[very thin,draw=gray!120] (1/4+1, 4 - 20/3-1) -- (1/4 - 3, 7 - 20/3);
        \node at (2,-1.5) {$\scriptstyle H_2$};
        
        \draw[very thin,draw=gray!120,dashed] (1/4,4) -- (1/4,4-20/3);
        \draw[very thin,draw=gray!120,dashed] (1,4-1/4-17/3) -- (1/4,4-1/4-17/3);

        \fill (1/4,4) circle (3pt) node[right] {$\scriptstyle Z$};
        
        \fill (1/4,4-20/3) circle (3pt) node[below] {$\scriptstyle Z'$};
        
        \fill (1,4-1/4-17/3) circle (3pt) node[right] {$\scriptstyle J$};
        
        \fill (-11/6,-1) circle (3pt) node[below] {$\scriptstyle V_1$};
        
        \fill (3/2,-1) circle (3pt) node[above right=-.75mm] {$\scriptstyle V_2$};
        
        \fill (1-1/4,-2.5) circle (3pt) node[below right=-.5mm] {$\scriptstyle W$};
        
        \fill (1/4,4-1/4-17/3) circle (3pt) node[above] {$\scriptstyle \widetilde W$};
        
        \fill (-2,4) circle (2pt);
        \fill (0,4) circle (2pt);
        \fill (2,4) circle (2pt);
        
        \fill (-1,3) circle (2pt);
        \fill (1,3) circle (2pt);
        
        \fill (-2,2) circle (2pt);
        \node at (0,2) {$*$};
        \fill (2,2) circle (2pt);
        
        \fill (-1,1) circle (2pt);
        \fill (1,1) circle (2pt);
        
        \fill (2,0) circle (2pt);
        \node at (0,0) {$*$};
        \fill (-2,0) circle (2pt);
        
        \fill (-1,-1) circle (2pt);
        \fill (1,-1) circle (2pt);
        
        \fill (-2,-2) circle (2pt);
        \fill (0,-2) circle (2pt);
        \fill (2,-2) circle (2pt);
        
        \fill (-1,-3) circle (2pt);
        \fill (1,-3) circle (2pt);
        
    \end{tikzpicture}
    \caption{(left) The cross-polygon in the transformed lattice.
    The top vertex $Z$ belongs to the dark grey region.
    The bottom vertex $W$ belongs to one of the four light grey subregions, labelled $A,B,C,D$.
    (right) The case where $W$ is in subregion $C$.\label{fig:cross_poly_triangle_lattice}
}
\end{figure}
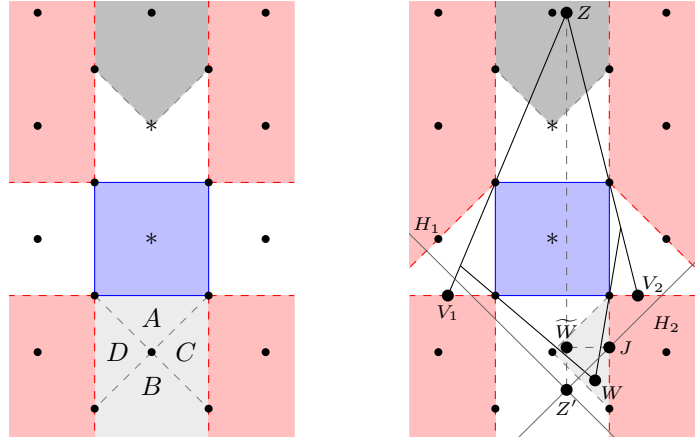
We can express the left and right vertices of $P$ in terms of the parameters $\kappa, \lambda, \mu, \nu$ as follows:
\[
    X = -f_2 - \frac{2}{\frac{\lambda-1}{\kappa+1} - \frac{\nu+1}{\mu+1}} \left(1,\frac{\lambda-1}{\kappa+1}\right),\qquad
    Y = f_2 + \frac{2}{\frac{\nu+1}{\mu-1} - \frac{\lambda-1}{\kappa-1}} \left(1, \frac{\nu+1}{\mu-1} \right).
\]
The width with respect to the horizontal functional $f^*_1+f^*_2$ is achieved at $X$ and $Y$ and can thus be expressed as follows. 
\[
    w_0 \coloneqq \width_{f^*_1+f^*_2}(P) = \mleft(f^*_1+f^*_2\mright)( Y - X) = 2 + \frac{2}{\frac{\nu+1}{\mu-1} - \frac{\lambda-1}{\kappa-1}} + \frac{2}{\frac{\lambda-1}{\kappa+1} - \frac{\nu+1}{\mu+1}}.
\]

We are now ready to prove the following.

\begin{proposition}
    Let $P$ be a maximal $\ZZ$-$\Delta_2$-free quadrilateral circumscribed around a cross-polygon.
    Then $\width(P) < \frac{10}3$.
\end{proposition}
\begin{proof}
    We have four cases depending on which subregion the bottom vertex $W$ lies in; these subregions are denoted by $A$, $B$, $C$ and $D$ in the left part of \Cref{fig:cross_poly_triangle_lattice}.
    Since the setup is symmetric about the $y$-axis, it suffices to consider the cases $A$, $B$, and $C$ as case $D$ is equivalent to case $C$.
    
    Let's start with case $A$.
    It turns out that it is enough to find the largest possible width of $P$ with respect to the directions $f_1^*$, $f_2^*$, and $f_1^*+f_2^*$.
    Notice that the bottom vertex $W$ of $P$ is not extremal with respect to any of those three width directions.
    Moving $W$ upwards increases all three of those widths.
    Thus, we may move $W$ to the line $y=-1$, in which case $P$ degenerates to a triangle.
    We can consider this triangle to be circumscribed around the \emph{three} points $f_1$, $f_2$, and $-f_2$, and thus it can be regarded as circumscribed around case 1 from \Cref{table:triangles} (one needs to apply a shearing to arrive at case 1), which has width less than or equal to $\frac{10}3$ with respect to those same three width directions.
    Hence, $\width(P) < \frac{10}3$.
    
    Let's now deal with case $B$.
    Here, the widths with respect to $f_1^*$ and $f_2^*$ are achieved at the vertices $Z$ and $W$:
    \[
    	w_1 \coloneqq \width_{f_1^*}(P) = f_1^*(Z - W) = \frac{\kappa-\mu}2 + \frac{\lambda-\nu}2\text {,}
    \]
    \[
        w_2 \coloneqq \width_{f_2^*}(P) = f_2^*\mleft(W - Z\mright) = \frac{\mu-\kappa}2 + \frac{\lambda-\nu}2 \text{.}
    \]
    The partial derivative of $w_1$ with respect to $\mu$ is
    \begin{equation}\label{eq:cross-poly_partial_mu}
    	\frac{\partial w_0}{\partial \mu}=\frac{8(\kappa-\mu)(\nu+1)(\lambda-1)\mleft((\lambda-1)(\kappa\mu-1)-(\nu+1)\mleft(\kappa^2+1\mright)\mright)}{\mleft(\mleft(\frac{\lambda-1}{\kappa+1}(\mu+1)-(\nu+1)\mright)\mleft(\nu+1-\frac{\lambda-1}{\kappa-1}(\mu-1)\mright)\mleft(\kappa^2-1\mright)\mright)^2} \text{.}
    \end{equation}
    Within our constraints $-1 < \kappa,\mu < 1$, $\lambda > 2$, and $\nu < -1$, the above expression vanishes if and only if $\kappa=\mu$; it is positive for $\mu<\kappa$ and negative for $\mu>\kappa$.
    Thus, $w_0$ is maximal along $\kappa=\mu$, where it is equal to $2+ \frac{4}{\lambda-\nu-2}$.
    By looking at the first summands of $w_1$ and $w_2$, it can be seen that $\min\{w_1, w_2\}$ is also maximal along $\kappa=\mu$.
    We obtain $\width(P) \leq \min\{2+ \frac{4}{\lambda-\nu-2}, \frac{\lambda-\nu}2\}$.
    Since $2 + \frac{4}{\lambda-\nu-2}$ decreases when $\frac{\lambda-\nu}2$ increases, and vice versa, the maximum is achieved when those expressions coincide, which occurs at $\nu=\lambda-6$. Hence, $\width(P) \leq 3$.
    
    Finally, we deal with case $C$; see the right part of \Cref{fig:cross_poly_triangle_lattice}.
	The width with respect to $f_1^*$ is achieved at $Z$ and $X$, while the width with respect to $f_2^*$ is achieved at $W$ and $Z$.
    Assume towards a contradiction that the widths with respect to the directions $f_1^*+f_2^*$, $f_1^*$, and $f_2^*$ are greater than or equal to $\frac{10}3$, i.e., $w_0, w_1, w_2 \geq \frac{10}{3}$. 
    Let $Z' \coloneqq Z-\mleft(0,\frac{20}{3}\mright)$ so that
    \[
    	f_1^*\mleft(Z-Z'\mright) = f_1^*\mleft(0,\frac{20}3\mright) = \frac{10}3 \qquad \text{and} \qquad f_2^*\mleft(Z'-Z\mright) = f_2^*\mleft(-\mleft(0,\frac{20}3\mright)\mright) = \frac{10}3 \text{.}
    \]
    
    Since $w_1$ (resp.~$w_2$) is assumed to be at least $\frac{10}3$, there is a point of $P$ below the line $H_1$ (resp.~$H_2$) passing through $Z'$ with slope $-1$ (resp.~$1$).
    In particular, vertex $X$ is below the line $H_1$ and vertex $W$ is below the line $H_2$, since they are the vertices maximising the width along the respective directions.
    
    We will use the following inequalities:
    \begin{equation}\label{eq:cross-poly_inequalities}
		\lambda \leq 4 \qquad \text{and} \qquad \kappa \ge 0\text{.}
    \end{equation}
    
    We first prove $\lambda \leq 4$.
	An upper bound for $w_0$ is given by the width with respect to $f_1^*+f_2^*$ of the triangle $T$ circumscribed around the cross-polygon with vertex $Z$ and opposite edge supported by the line $y=-1$. The bottom vertices of $T$ are $V_1 \coloneqq (-1-\frac{2(\kappa+1)}{\lambda-1}, -1)$ and $V_2 \coloneqq (1-\frac{2(\kappa-1)}{\lambda-1}, -1)$.
    Thus, $\width_{f_1^*+f_2^*}(T)= 2+\frac4{\lambda-1}$.
    Since by assumption $w_0 \geq\frac{10}{3}$, so is this larger width.
    We obtain $\lambda \leq 4$.
    
    To prove that $\kappa \ge 0$, consider the bottom-left vertex $V_1$ of the triangle $T$ defined in the previous paragraph.
    Since the vertex $X$ of $P$ lies below $H_1$, so does $V_1$, and thus $-2 - \frac{2(\kappa+1)}{\lambda-1} \leq \kappa+\lambda-\frac{20}3$, which is equivalent to $\frac{20}3 + \kappa - \frac{17}{3}\lambda + \kappa\lambda + \lambda^2 \geq 0$.
    Consider the left side of the previous inequality as a family of functions $f_\kappa \colon [2,4] \to \RR$ on the closed interval $[2,4]$ for parameters $\kappa \in [-1,1]$.
    Observe that $\kappa \in [-1,1]$ is admissible if and only if there exists $\lambda \in [2,4]$ such that $f_\kappa(\lambda)\ge0$ (we include the case $\lambda=2$).
    Then $\kappa \in [-1,1]$ is admissible if the maximum of $f_\kappa$ on $[2,4]$ is non-negative.
    It is straightforward to show
    \[
    	\max_{\lambda \in [2,4]} f_\kappa(\lambda) = \begin{cases}
			5 \kappa & \kappa > -\frac13\\
			3 \kappa -\frac23 & \text{otherwise}
		\end{cases} \qquad \text{at} \; \lambda=\begin{cases}
			4 & \kappa> -\frac13\\
			2 & \text{otherwise}
		\end{cases}
    \]
    Hence, $\kappa\ge0$ are the only admissible parameters.
        
    Next, we aim to bound $w_0$ from above by considering a different quadrilateral $\widetilde{P}$ circumscribed around the cross-polygon, with top vertex $Z$ and bottom vertex $\widetilde{W} \coloneqq (\kappa, \lambda-\kappa-\frac{17}{3})$ (note that the inequalities from \eqref{eq:cross-poly_inequalities} guarantee that $\lambda-\kappa-\frac{17}3 < -1$).
    It is straightforward to compute that $\widetilde{w}_0 \coloneqq \width_{f_1^* + f_2^*}(\widetilde{P}) = 2 + \frac{4}{\kappa + \frac{11}{3}}$.
    Since $\kappa \geq 0$, we get that $\widetilde{w}_0 \leq 2+\frac{12}{11} < \frac{10}{3}$.
    
    To reach a contradiction, we show that $w_0 \leq \widetilde{w}_0$.
    Consider the partial derivatives of $w_0$ with respect to $\mu$ and $\nu$, the coordinates of the vertex $W$.
    The first was already computed in \Cref{eq:cross-poly_partial_mu}, and we observed that the maximum is achieved along $\mu = \kappa$.
    The second is computed here.
    
    \begin{equation}\label{eq:cross-poly_partial_nu}
        \frac{\partial w_0}{\partial \nu} = \frac{2}{(1-\mu)\mleft(\frac{\nu+1}{\mu-1} - \frac{\lambda-1}{\kappa-1}\mright)^2} + \frac{2}{(1+\mu)\mleft(\frac{\lambda-1}{\kappa+1} - \frac{\nu+1}{\mu+1}\mright)^2}
    \end{equation}
    
    Since $-1 < \mu < 1$, \Cref{eq:cross-poly_partial_nu} shows that $\frac{\partial w_0}{\partial \nu} > 0$.
    Now, consider the intersection point $J$ of the lines $H_2$ and $\set{x=1}$.
    It's straightforward to compute $J = (1,\lambda-\kappa-\frac{17}{3})$.
    The point $J$ has largest $y$-coordinate of all points in $C$ on or below $H_2$ (see \Cref{fig:cross_poly_triangle_lattice}).
    Thus, $\nu \leq \lambda - \kappa - \frac{17}{3}$.
    So, moving $W$ horizontally to the line $x=\kappa$ and then vertically to the line $y=\lambda-\kappa-\frac{17}{3}$ increases the width $w_0$, and thus $w_0 \leq \widetilde{w}_0$.
    Combined with $\widetilde{w}_0 < \frac{10}{3}$, we get that $\width(P) \leq w_0 < \frac{10}{3}$, a contradiction.
\end{proof}

This concludes the proof of \Cref{theorem:flt_Z_D2}.


\section{\texorpdfstring{$\RR$-Generalised Flatness Constant of $\Delta_2$}{R-Generalised Flatness Constant of Delta2}}\label{sec:r_flatness_d2}
This section focuses on the $\RR$-flatness constant. 
The first goal is to prove the following theorem (case $A=\RR$ of \Cref{thm:main}).

\begin{theorem}[Case $A=\RR$ of \Cref{thm:main}]\label{theorem:flt_R_D2_is_2}
    $\flatness_2^\RR(\Delta_2) = 2$.
\end{theorem}

It is straightforward to verify that the cross-polygon $\Diamond_2\coloneqq \conv(\pm e_1, \pm e_2)$ is $\RR$-$\Delta_2$-free and has width $2$, and hence $\flatness_2^\RR(\Delta_2) \geq 2$.
To prove \Cref{theorem:flt_R_D2_is_2}, we thus need to bound the $\RR$-flatness constant of $\Delta_2$ from above by $2$.
By Lemmas~\ref{lem:full-dim} and~\ref{lem:incl-max}, it suffices to bound the lattice width of inclusion-maximal $\RR$-$\Delta_2$-free closed convex sets $C$ by $2$.
By \Cref{prop:unbounded}, if $C$ is unbounded, its lattice width is bounded by $1$.
Hence, it remains to study the width of the bounded $C$'s which, by \Cref{thm_incl_max}, are polygons.
Our strategy is to show that any polygon $P \subset \RR^2$ with width greater than 2 is not $\RR$-$\Delta_2$-free.
A key ingredient in the proof is the notion of rational diameter:

\begin{definition}
    Let $K \subset \RR^d$ be a convex body.
    The \novel{rational diameter of $K$} is the largest dilation of an $\RR$-unimodular copy of the unit segment $[0,1]$ which is contained in $K$, i.e.
    \[
        l(K) \coloneqq \max\setcond{l \in \RR_{\geq 0}}{lS \subseteq K \; \text{for some $\RR$-unimodular copy $S$ of $[0,1]$}}.
    \]
\end{definition}

\begin{notation}
    In what follows, we will always assume that the rational diameter of $P$ is achieved with a horizontal line segment.
    This we can do without loss of generality, because, were this not the case, we could apply an $\RR$-unimodular transformation mapping the rational diameter into a horizontal segment.
    We thus use the following shorthand notation for horizontal line segments $S \subset \RR^d$: $[x,y] \coloneqq \conv(x e_1, y e_1)$.
\end{notation}

The following lemma shows that for polygons the width is bounded from above by twice the rational diameter, i.e., $\width(P) \le 2l$ where $l$ is the rational diameter of $P$.

\begin{lemma}\label{lemma:width_leq_2l}
    Let $P \subset \RR^2$ be a polygon with rational diameter $l=1+2a > 0$ (that is, $a > -\frac12$), achieved with $S = [-a,1+a] \subseteq P$.
    Then $P \subset \setcond{(x,y) \in \RR^2}{-l \leq y \leq l}$.
\end{lemma}
\begin{proof}
    The affine line $\{ y = l \}$ is covered by segments $S_b \coloneqq S + (lb,l)$ for $b \in \ZZ$.
    First, notice that we must have $(x,l) \not\in \strint(P)$ for all $(x,l) \in S_0$, as otherwise the line segment with end points $(x,0)$ and $(x,l)$ could be extended upwards and still be contained in $P$, contradicting that $l$ is maximal.
    Now, the same argument holds for the points in the segments $S_b$ for any $b \in \ZZ$ by applying an appropriate unimodular transformation (a shearing).
    Since the segments $S_b$ cover the horizontal line at height $l$, all points $(x,l)$ with $x \in \RR$ must be disjoint from $\strint(P)$.
    Applying a reflection about the $x$-axis gives us that all $(x,-l)$ with $x \in \RR$ must be disjoint from $\strint(P)$.
    Thus, $P \subset \setcond{(x,y) \in \RR^2}{-l \leq y \leq l}$ as desired.
\end{proof}

Recall that for the proof of \Cref{theorem:flt_R_D2_is_2} we are only interested in polygons of width strictly larger than $2$. \Cref{lemma:width_leq_2l} shows that we then only need to consider polygons with rational diameter $l> 1$.
The following lemma will allow us to also bound the rational diameter from above.

\begin{lemma}\label{lemma:l_geq_2_implies_width_leq_2}
    Let $P \subset \RR^2$ be an $\RR$-$\Delta_2$-free polygon with rational diameter $l= 1+2a>1$ (that is, $a>0$), achieved with $S = [-a,1+a] \subseteq P$.
    Then $\strint(P)$ is disjoint from the segments $[-a, a] + (b, \pm1)$, for all $b \in \ZZ$.
    In particular, if $a \geq \frac12$ then $P \subset \setcond{(x,y) \in \RR^2}{-1 \leq y \leq 1}$ and $\width(P) \leq 2$.
\end{lemma}
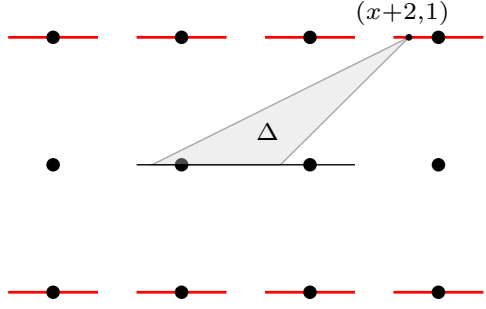
\begin{figure}[ht!]
     \centering
     \scalebox{1.3}{
     \begin{tikzpicture}[scale=1.3]
       \pgfmathsetmacro{\A}{.35};
       \pgfmathsetmacro{\Z}{-.23};
        \draw[thin] (-\A,0) -- (1+\A,0);
        \foreach \x in {-1, ..., 2}{
        \foreach \y in {-1, 1}{
        \draw[red, thick]  (\x-\A,\y) -- (\x+\A, \y);}}
        \foreach \x in {-1, ..., 2}{
        \foreach \y in {-1, ..., 1}{
        \fill (\x,\y) circle (1.5pt);}}
        \fill (\Z+2,1) circle (.75pt);
        \draw (\Z+1.95, 1) node[ above] {$\scriptstyle (x+2,1)$};
        \draw[fill=gray!35,opacity=.35] (\Z+2, 1) -- (\Z,0) -- (\Z+1,0) -- cycle;
        \draw (.67,.1) node[ above] {$\scriptstyle \Delta$};
    \end{tikzpicture}}
    \caption{Forbidden segments from \Cref{lemma:l_geq_2_implies_width_leq_2}.\label{fig:forbidden_segments}
}
\end{figure}
\begin{proof}
    Consider the point $(x+b,1)$, for some $-a < x < a$ and $b \in \ZZ$.
    The convex hull of this point and the points $(x,0), (x+1,0) \in P$ is an $\RR$-unimodular copy $\Delta$ of $\Delta_2$.
    If the point $(x+b,1)$ were contained in the interior of $P$, the Minkowski difference $P\text{\textdiv}\Delta$ would be $2$-dimensional.
    By \Cref{lem_R_X_free_with_Tran}, this contradicts $P$ being $\RR$-$\Delta_2$-free.
    Therefore, $(x+b,1) \not\in \strint(P)$, for all $x+b \in \strint([-a,a]) + \ZZ$.
    The same argument holds for all points $(x+b,-1)$.
    Thus, $\strint(P)$ is disjoint from the interiors of all segments $[-a,a] + (b,\pm1)$.
    
    In fact, $\strint(P)$ is disjoint from not just the interior, but from the whole segment $[-a,a] + (b,\pm1)$ for all $b \in \ZZ$: suppose otherwise that one of the endpoints is contained in the interior of $P$.
    Then, there exists an open ball around that endpoint which is contained in the interior of $P$.
    This yields a contradiction, as the interiors of the segments must be disjoint from $\strint(P)$.
    
    Suppose $l \geq 2$, that is, $a \geq \frac12$.
    Then the interior of $P$ is disjoint from all points $(x+b, \pm1)$ with $-\frac12 \leq x \leq \frac12$ and $b \in \ZZ$, which cover the entire horizontal lines at height $1$ and $-1$.
    Thus $P$ must be contained in the strip $\setcond{(x,y) \in \RR^2}{-1 \leq y \leq 1}$, which has width $2$.
\end{proof}

We are now ready to prove \Cref{theorem:flt_R_D2_is_2}.
We use the bounds on the lattice diameter from the previous lemmas and the lower bound on the width to show that certain areas of the plane are disjoint from the polygon $P$.
Eventually this allows us to bound the polygon so tightly that we reach a contradiction. 

\begin{proof}[Proof of \Cref{theorem:flt_R_D2_is_2}]
    Let $P \subset \RR^2$ be an $\RR$-$\Delta_2$-free polygon, with rational diameter $l$ and assume towards a contradiction that $\width(P) > 2$.
    We may assume without loss of generality that the rational diameter is achieved with $S = [-a,1+a]$, where $l = 1+2a$.
    By \Cref{lemma:width_leq_2l}, we have $P \subset \setcond{(x,y) \in \RR^2}{-l \leq y \leq l}$, and thus $2l \geq  \width(P) > 2$, that is, $l>1$, i.e. $a > 0$.
    By \Cref{lemma:l_geq_2_implies_width_leq_2}, $\strint(P)$ is disjoint from all segments $[-a,a] + (b,\pm1)$ with $b \in \ZZ$, and, since $\width(P) >2$, it follows that $a < \frac12$.
    
    Since $\width(P) >2$, there exists a point $(r,s) \in P$ with either $1 < s \leq l$ or $-l \leq s < 1$.
    Due to the symmetry about the $x$-axis, we may assume that $1 < s \leq l$.
    We may also assume that $(r,s)$ is the point of $P$ with largest $y$-coordinate.
    Furthermore, we may assume that $a \leq r \leq 1-a$; otherwise, we apply a shearing so that the $x$-coordinate of our point would satisfy this condition.
  \begin{figure}[ht!]
       \centering
       \scalebox{1.3}{
       \begin{tikzpicture}[scale=1.3]
       \pgfmathsetmacro{\R}{.55}
       \pgfmathsetmacro{\S}{1.1}
       \pgfmathsetmacro{\A}{.25}

        \draw[thin] (-\A,0) -- (1+\A,0);
        \fill (\R, \S) circle (1pt);
        \draw (\R-.5,\S+.1) node[ above] {$\scriptstyle (r,s)$} edge[-latex,bend left] (\R,\S);
        \fill (\R,\S-1-2*\A) circle (1pt);
        \node[above] at (.5, -.475) {$\scriptstyle (r,s-l)$}; 
        \draw[thin] (\A, 1) -- (.5, .5+1/\A/4);
        \draw[thin] (1-\A, 1) -- (.5, .5+1/\A/4);
        \draw[thin, dashed, ->>-] (-\A, 0) -- (\R, \S)  ;
        \draw[thin, dashed, ->-] (\R, \S) -- (1+\A, 0);
        \draw[red, thick, ->>-]  (-\R-2*\A,-\S) -- (-\A, 0) ;
        \draw[red, thick, ->-]   (-1-2*\A+\R,\S) -- (-\A,0);
        \fill[fill=red!35,opacity=.25] (-\A, 0) -- (-\R-2*\A,-\S) -- (-\R-2.8*\A,-\S) -- (-\R-2.8*\A,\S) -- (-1-2*\A+\R,\S) -- cycle;
        \node[above] at (-.95, .35) {$\scriptstyle C_1$};
        \draw[red, thick, ->-] (1+\A, 0) -- (2+2*\A -\R,-\S);
        \draw[red, thick, ->>-] (1+\A,0) -- (1+2*\A+\R,\S);
        \fill[fill=red!35,opacity=.25] (1+\A, 0) -- (2+2*\A -\R,-\S) -- (2+3.1*\A -\R,-\S) -- (2+3.1*\A -\R,\S) -- (1+2*\A+\R,\S) -- cycle;
        \node[above] at (1.95, .35) {$\scriptstyle C_2$};
        \draw[red, thick] (\R,\S-1-2*\A) -- (2.7*\R-1.7*\A-1.7, 2.7*\S-2.7-5.4*\A);
        \draw[red, thick] (\R,\S-1-2*\A) -- (2.7*\R+1.7*\A, 2.7*\S-2.7-5.4*\A);
        \fill[fill=red!35,opacity=.25] (\R,\S-1-2*\A) -- (2.7*\R-1.7*\A-1.7, 2.7*\S-2.7-5.4*\A) --  (2.7*\R+1.7*\A, 2.7*\S-2.7-5.4*\A) -- cycle;
        \node[above] at (.45, -1) {$\scriptstyle C_3$};
        \foreach \x in {-1, ..., 2}{
        \foreach \y in {-1, 1}{
        \draw[red, thick]  (\x-\A,\y) -- (\x+\A, \y);}}
        \foreach \x in {-1, ..., 2}{
            \foreach \y in {-1, ..., 1}{
                \fill (\x,\y) circle (1.5pt);
			}
		}
        \fill (\R,\S-1-2*\A) circle (1pt);
  \end{tikzpicture}}
        \caption{Regions $C_1$, $C_2$, $C_3$ from Claims~\ref{claim:cone12_forbid} and~\ref{claim:cone3_forbid}.\label{fig:forbidden_regions_tikx}
}
    \end{figure}
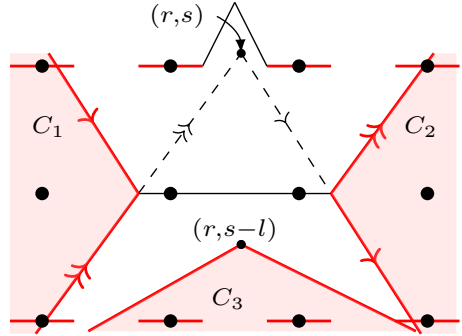
    Let $L_1$ be the line through the points $(r,s)$ and $(1+a,0)$, and $L_2$ the line through $(r,s)$ and $(-a, 0)$.
    We define affine pointed cones $C_1$ and $C_2$ with apex in $(-a, 0)$ respectively in $(1+a, 0)$, and rays bounding them above and below, parallel to $L_1$ and $L_2$ for $C_1$, respectively $L_2$ and $L_1$ for $C_2$.
    Precisely, $C_1 \coloneqq (-a,0) + \cone((r-1-a,s), (-r-a,-s))$ and $C_2 \coloneqq (1+a,0) + \cone((r+a,s), (-r+1+a,-s))$.
    \begin{claim}\label{claim:cone12_forbid}
        The regions $C_1$ and $C_2$ are disjoint from $\strint(P)$.
    \end{claim}
    \begin{proof}[Proof of \Cref{claim:cone12_forbid}]
        The arguments for the affine cones $C_1$ and $C_2$ are the same, so we conduct the proof for $C_1$.
        
        Since the rational diameter is achieved at $S\coloneqq [-a, 1+a]$, no other point on the horizontal axis can be in $P$.
        If a point $(x,y)$ in the interior of $C_1$ with $y<0$ were in $P$, the segment connecting it to $(r,s)$ would also be in $P$, which is a contradiction as this segment intersects the horizontal axis outside of $S$.
 
        Suppose now that a point $(x,y)$ in the interior of $C_1$ with $0<y<1$ were in $P$.
        The horizontal segment $S'$ with an endpoint in $(x,y)$ and one in $L_1$ would then also be in $P$.
        Since $S'$ intersects the boundary ray of $C_1$ parallel to $L_1$ in an interior point, $S'$ strictly contains a segment congruent to $S$.
        Since $S'$ has lattice length strictly larger than $l$, we obtain a contradiction to $l$ being the rational diameter of $P$.
        
        Finally, no point  $(x,y) \in C_1$ with $y\geq 1$ can be in $P$ because it would force $P$ to contain the lattice point $(0,1)$ in its interior, and thus a small translate of $\conv((0,0), (0,1), (1,0,))$ would be contained in the interior of $P$, a contradiction to the hypothesis that $P$ is $\RR$-$\Delta_2$-free.

        These three cases together show that no point in the interior of $C_1$ can be in $P$, which is equivalent to our claim that $C_1$ is disjoint from the interior of $P$.
    \end{proof}
    
    Let $C_3 \coloneqq (r,s-l) + \cone((r-1-a,s-l), (r+a,s-l))$ be the affine cone with apex in $(r, s-l)$, and two boundary rays lying on the two lines passing through $(r, s-l)$ and the two endpoints of $S$ respectively.
    \begin{claim}\label{claim:cone3_forbid}
        The region $C_3$ is disjoint from $\strint(P)$.
    \end{claim}
    \begin{proof}[Proof of \Cref{claim:cone3_forbid}]
        Observe that the point $q \coloneqq (r, s-1-2a)$ cannot be in $\strint(P)$, because if any point $q'$ vertically below $q$ were in $P$, the segment with endpoints $q'$ and $(r,s)$ would have lattice length strictly larger than $S$, a contradiction. 
        
        If any point $p \in \strint(C_3)$ were also in $P$, the triangle $\conv(p, (-a,0), (1+a, 0))$ would be in $P$.
        However, this triangle contains $q$ in its interior, a contradiction. 
    \end{proof}

    Recall $s$ is assumed to have the largest $y$-coordinate of all points of $P$.
    Let $L_5$ be the line parallel to the $x$-axis and going through the point $(0,s-2)$.
    Since $P$ has width greater than $2$, $P$ intersects $L_5$ in a line segment $[a, b]$ of positive length which by the previous claims is not in $\strint(C_i)$ for $i=1,2,3$.
    By symmetry, we may assume $[a,b]$ lies between $C_1$ and $C_3$.
    Furthermore, suppose $a$ is to the left of $b$.

    The key idea of the proof is to consider the family of quadrilaterals (see also \Cref{fig:3lines_Qx})
    \[
    	Q_x\coloneqq \conv((-a,0), (1+a,0), (r,s), (x,s-2)) \qquad \text{for $(x, s-2)\in \RR^2$ between $C_1$ and $C_3$.}
    \]
    Note $x$ is the only free parameter of this family.
    We think of the elements of this family to be quadrilaterals where the bottom vertex can move.
    Indeed, $Q_{x'}$ is obtained from $Q_x$ via a piecewise linear transformation $\psi_{x'-x} \colon \RR^2 \to \RR^2$ which is given as follows: below the $x$-axis, apply the shearing which maps $(x, s-2)$ to $(x', s-2)$; above, apply the identity.
    We want to show that every $Q_x$ of this family contains an $\RR$-translation of $\widetilde{\Delta} \coloneqq \conv(\mathbf{0}, -e_1, e_2)$.
    We do this in two stages: 1) show if $Q_{x'}$ is obtained from $Q_x$ by moving the bottom vertex further to the left and $Q_x$ contains an $\RR$-translation of $\Tilde{\Delta}$, then $Q_{x'}$ does so as well; 2) show that the quadrilateral $Q_{\bar{x}}$ for the largest possible $\bar{x}$ contains an $\RR$-translation of $\Tilde{\Delta}$.
    
    Granting this statement for a moment, let us complete the proof of \Cref{theorem:flt_R_D2_is_2}.
    There is an $x \in \RR$ such that $Q_x = \conv((-a,0),(1+a,0),(r,s),b)$.
    Suppose $\Delta$ is an $\RR$-translation of $\Tilde{\Delta}$ that is contained in $Q_x$.
	Our goal is to show that $P_x\text{\textdiv}\Delta$ is $2$-dimensional where $P_x$ is the pentagon $\conv(Q_x, a)$.
	
    Consider the horizontal line segment $B$ which forms the base of $\Delta$.
    Note that $B$ lies below the $x$-axis: otherwise, the vertex $v$ of $\Delta$ which is a translation of $e_2$ would have $y$-coordinate at least $1$, which forces its $x$-coordinate to be less than $1-a$; this in turn forces the vertex of $\Delta$ corresponding to $-e_1$ to have $x$-coordinate less than $-a$, which is a contradiction since no such point lies in $Q_x$.
	Thus, if $\Delta$ intersects the boundary of the pentagon $P_x$, then it does it in a single point that is contained in the interior of the respective facet of $P_x$.
	Since the vertex of $\Delta$ corresponding to $-e_1$ lies in the interior of $P_x$, it follows that at most two vertices of $\Delta$ lie on the boundary of $P_x$.
	With \Cref{lem_translations} it's straightforward to show that $P_x\text{\textdiv}\Delta$ is $2$-dimensional.
	By \Cref{lem_R_X_free_with_Tran}, it follows that $P_x$ isn't $\RR$-$\Delta_2$-free.
	A contradiction to the assumption $\width(P)>2$.
    Therefore, an $\RR$-$\Delta_2$-free polygon $P$ has $\width(P) \leq 2$.
    Thus, $\flatness_2^\RR(\Delta_2) = 2$.
    
    It remains to prove the two claims from above.
    \begin{claim}\label{claim:vertex_on_right}
        Suppose $Q_x$ contains a translation $\Delta$ of $\widetilde{\Delta}$.
        Then $Q_{x'}$ also contains a translation $\Delta'$ of $\widetilde{\Delta}$, for $x' \leq x$ and $(x',s-2) \not\in C_i$ for all $i=1,2,3$.
    \end{claim}
    \begin{proof}[Proof of \Cref{claim:vertex_on_right}]
        Consider the horizontal line segment $B$ which forms the base of $\Delta$.
        With the same argument as above, it follows that $B$ lies below the $x$-axis.
        
        Let $x' \leq x$.
        Clearly, the image of $B$ under the transformation $\psi_{x'-x}$ is the segment $B' \coloneqq B + (x'-x,0)$, which is thus contained in $Q_{x'}$.
        In order to show that $\Delta' \coloneqq \Delta + (x'-x,0)$ is contained in $Q_{x'}$, it remains to prove that $v' \coloneqq v + (x'-x,0)$ is contained in $Q_{x'}$. 
        Note $v'$ and $v$ are both above the $x$-axis.
        Since $Q_x$ and $Q_{x'}$ coincide above the $x$ axis, $v$ lies in both.
        Since $v'$ lies to the left of $v$, it must also lie to the left of $L_1$. 
        Since $L_2$, the line through $(-a,0)$ and $(r,s)$, has slope greater than or equal to $1$, $v'$ must be to the right of it; otherwise, the left vertex of $B'$, which lies on the line through $v'$ with slope $1$, would also be to the left of $L_2$, and thus contained in $C_1$, a contradiction.
        Thus $\Delta'$ is contained in $Q_{x'}$.
    \end{proof}
    
    We now want to show that, letting $(\bar{x}, s-2)$ be the right-most point between $\strint(C_1)$ and $\strint(C_3)$, $Q_{\bar{x}}$ contains an $\RR$-translation of $\Tilde{\Delta}$.
    Explicitly we have $\bar{x}= \frac{r(s-2)+(1+a)(1-2a)}{s-1-2a}$.
    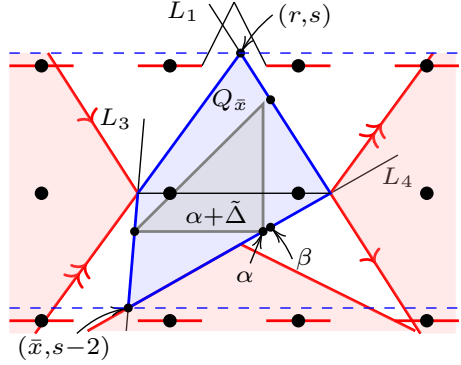
\begin{figure}[ht!]
        \centering
        \scalebox{1.3}{
        \begin{tikzpicture}[scale=1.3]
            \pgfmathsetmacro{\R}{.55}
            \pgfmathsetmacro{\S}{1.1}
            \pgfmathsetmacro{\x}{-.6}
            \pgfmathsetmacro{\A}{.25}
            \pgfmathsetmacro{\xb}{1+\A + (\R-1-\A)*(\S-2)/(\S-1-2*\A)}
            \pgfmathsetmacro{\t}{1/(\S-1-2*\A)}
            \pgfmathsetmacro{\u}{-(\S-2)/\S}
            \pgfmathsetmacro{\bx}{1+\A + (\R-1-\A)/(1+2*\A)}
            \pgfmathsetmacro{\by}{(\S-1-2*\A)/(1+2*\A)}
            \pgfmathsetmacro{\ax}{\xb + (1+\A-\xb)/(1+2*\A)}
            \pgfmathsetmacro{\ay}{((\S-2)*2*\A)/(1+2*\A)}
    
            \draw (2+2*\A -\R,-\S) -- (-0.35-0.35*\A +1.35*\R, 1.35*\S);
            \node[left] at (-0.33-0.35*\A +1.35*\R, 1.3*\S) {$\scriptstyle L_1$};
            \draw (2.7*\R-1.7*\A-1.7, 2.7*\S-2.7-5.4*\A) -- (1+\A+.3*\t*\R-.3*\t*\A-.3*\t, .3*\t*\S-.3*\t-.6*\t*\A);
            \draw (1+\A+.3*\t*\R-.3*\t*\A-.3*\t, .3*\t*\S-.3*\t-.6*\t*\A) node [below] {$\scriptstyle L_4$};
            \draw
            (.2*\A+1.2*\xb, 1.2*\S-2.4)--
            (-.65*\xb-1.65*\A, 1.3-.65*\S);
            \draw  (-.65*\xb-1.65*\A, 1.3-.65*\S) node [left] {$\scriptstyle L_3$};
    
            \draw[red, thick, ->>-]  (-\R-2*\A,-\S) -- (-\A, 0) ;
            \draw[red, thick, ->-]   (-1-2*\A+\R,\S) -- (-\A,0);
            \fill[fill=red!35,opacity=.25] (-\A, 0) -- (-\R-2*\A,-\S) -- (-\R-2.8*\A,-\S) -- (-\R-2.8*\A,\S) -- (-1-2*\A+\R,\S) -- cycle;
    
            \draw[red, thick, ->-] (1+\A, 0) -- (2+2*\A -\R,-\S);
            \draw[red, thick, ->>-] (1+\A,0) -- (1+2*\A+\R,\S);
            \fill[fill=red!35,opacity=.25] (1+\A, 0) -- (2+2*\A -\R,-\S) -- (2+3.1*\A -\R,-\S) -- (2+3.1*\A -\R,\S) -- (1+2*\A+\R,\S) -- cycle;
    
            \draw[red, thick] (\R,\S-1-2*\A) -- (2.7*\R-1.7*\A-1.7, 2.7*\S-2.7-5.4*\A);
            \draw[red, thick] (\R,\S-1-2*\A) -- (2.7*\R+1.7*\A, 2.7*\S-2.7-5.4*\A);
            \fill[fill=red!35,opacity=.25] (\R,\S-1-2*\A) -- (2.7*\R-1.7*\A-1.7, 2.7*\S-2.7-5.4*\A) --  (2.7*\R+1.7*\A, 2.7*\S-2.7-5.4*\A) -- cycle;
    
            \fill[fill=blue!35,opacity=.25] (\R, \S) -- (-\A, 0) -- (\xb, \S-2) -- (\A+1,0);
            \draw[thick, blue] (\R, \S) -- (-\A, 0) -- (\xb, \S-2) -- (\A+1,0) -- cycle;
            \draw (.47, .55) node[ above] {$\scriptstyle Q_{\bar{x}}$};
            \draw (\xb-.4,\S-2.2) edge[-{>[length=5,width=3]},bend left=15] (\xb,\S-2);
            \node[below]  at (\xb-.5,\S-2.12) {$\scriptstyle (\bar{x},s-2)$};
    
            \fill[fill=gray!35,opacity=.45] (\ax, \ay) -- (\ax-1, \ay) -- (\ax, \ay+1) -- cycle;
            \draw[thick, gray] (\ax, \ay) -- (\ax-1, \ay) -- (\ax, \ay+1) -- cycle;
            \draw (.36, -.35) node[ above] {$\scriptstyle \alpha+ \Tilde{\Delta}$};
    
            \draw[thin] (-\A,0) -- (1+\A,0);
            \fill (\R, \S) circle (1pt);
            \draw (\R+.25,\S+.3)  edge[-{>[length=5,width=3]},bend right=10] (\R,\S);
            \node[above] at (\R+.5,\S+.1) {$\scriptstyle (r,s)$};
    
            \draw[thin] (\A, 1) -- (.5, .5+1/\A/4);
            \draw[thin] (1-\A, 1) -- (.5, .5+1/\A/4);
            \fill (\xb,\S-2) circle (1pt);
            \fill (\bx,\by) circle (1pt); 
            \draw (\bx,\by) edge[bend left=10,{<[length=5,width=3]}-] (\bx+.17,\by-.25);
            \node[right] at (\bx+.1,\by-.25)  {$\scriptstyle \beta$};
            \fill (\bx,\by+1) circle (1pt);
            \fill (\ax,\ay) circle (1pt);
           \draw[{<[length=5,width=3]}-] (\ax,\ay) -- (\ax-.1,\ay-.25) node[xshift=-1.5, yshift=-4] {$\scriptstyle \alpha$};
            \fill (\ax-1,\ay) circle (1pt);
    
            \draw[thin, blue, dashed] (-\A-1, \S) -- (\A+2, \S);
            \draw[thin, blue, dashed] (-\A-1, \S-2) -- (\A+2, \S-2);
    
            \foreach \x in {-1, ..., 2}{
            \foreach \y in {-1, 1}{
            \draw[red, thick]  (\x-\A,\y) -- (\x+\A, \y);}}
            \foreach \x in {-1, ..., 2}{
            \foreach \y in {-1, ..., 1}{
            \fill (\x,\y) circle (1.5pt);}}
        \end{tikzpicture}}
        \caption{The quadrilateral $Q_{\bar{x}}$ and the lines $L_1,L_3,L_4$.\label{fig:3lines_Qx}
}
    \end{figure}
    Let $L_1, L_2, L_3, L_4$ be the lines defining the boundary of $Q_{\bar{x}}$: $L_1$ is the line through the points $(r,s)$ and $(1+a, 0)$, and the others are chosen to lie in  counterclockwise order along the boundary of $Q_{\bar{x}}$.
    Explicitly,
   \begin{align*}
        L_1 &= \setcond{(1+a, 0) + t_1 (r-1-a, s)}{t_1 \in \RR}, \\
        L_3 &= \setcond{(\bar{x}, s-2) + t_3 (-a-\bar{x}, 2-s)}{t_3 \in \RR},\\
        L_4 &= \setcond{(1+a, 0) + t_4 (r-1-a, s-1-2a)}{t_4 \in \RR} \\
            &= \setcond{(\bar{x}, s-2) + t_4' (1+a-\bar{x}, 2-s)}{t_4' \in \RR} \text{.}
    \end{align*}
    The lines $L_1 - e_2$ and $L_4$ intersect at $t_1=t_4=\frac{1}{1+2a}$ in the point $\beta=(\beta_1, \beta_2)=(1+a+ \frac{r-1-a}{1+2a}, \frac{s-1-2a}{1+2a})$. Both the points $\beta$ and $\beta + e_2$ are in $Q_{\bar{x}}$.
    The intersection between $L_4$ and $L_3 + e_1$, obtained at $t_3=t_4'=\frac{1}{1+2a}$, is the point $\alpha=(\alpha_1, \alpha_2)=(\bar{x}+\frac{1+a-\bar{x}}{1+2a},\frac{(s-2)2a}{1+2a})$; and both $\alpha$ and $\alpha - e_1$ are in $Q_{\bar{x}}$.
    Since $s>1$ and $0<a<\frac12$, we have $\beta_2 -\alpha_2 = \frac{(s-1)(1-2a)}{1+2a}>0$, or equivalently $\alpha_2 \leq \beta_2$.
    Hence $\alpha + \Tilde{\Delta}$ and $\beta + \Tilde{\Delta}$ are both contained in $Q_{\bar{x}}$ (recall $\Tilde{\Delta}\coloneqq \conv(\mathbf{0}, -e_1, e_2)$).
\end{proof}


\subsection{\texorpdfstring{Inclusion-maximal $\RR$-$\Delta_2$-free convex bodies in dimension $2$}{inclusion-maximal R-Delta-2-free convex bodies in dimension 2}}\label{sec_incl_max_dim_2}
In the previous section, we have established the maximum width of $\RR$-$\Delta_2$-free convex bodies.
We now devote our attention to inclusion-maximal $\RR$-$\Delta_2$-free convex bodies.
We have seen in \Cref{thm_incl_max} that inclusion-maximal $\RR$-$\Delta_2$-free convex bodies are in fact always polytopes.
It would be interesting to have a complete characterisation of these polytopes, in analogy to the classification of maximal hollow $2$-bodies of Hurkens~\cite{Hurkens}.
Note that hollow convex bodies are $\ZZ$-$\{\mathbf{0}\}$-free convex bodies.

Here, we will see that the situation for $\RR$-$\Delta_2$-free bodies is more intricate than for hollow ones. 
We construct infinite families of $\RR$-$\Delta_2$-free bodies but cannot classify all such bodies.
It might thus be of particular interest to investigate the special class of those inclusion-maximal $\RR$-$\Delta_2$-free polygons which achieve the maximum width $2$.
Of these, we only know two examples (up to $\RR$-unimodular equivalence): the cross-polygon $\conv(\pm e_1, \pm e_2)$ and the triangle $\conv(e_1, e_2, -e_1-e_2)$, see \Cref{fig:maximal_width_2}. 

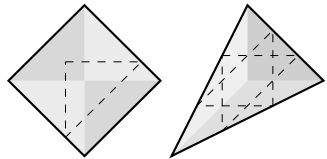
\begin{figure}[!ht]
    \centering
    \begin{tikzpicture}
    \fill[fill=gray!50,opacity=.6,very thin,dashed] (0,0) -- (1,0) -- (0,-1) -- cycle;
    \fill[fill=gray!50,opacity=.6,very thin,dashed] (0,0) -- (-1,0) -- (0,1) -- cycle;
    \fill[fill=gray!50,opacity=.3,very thin,dashed] (0,0) -- (-1,0) -- (0,-1) -- cycle;    
    \fill[fill=gray!50,opacity=.3,very thin,dashed] (0,0) -- (1,0) -- (0,1) -- cycle;
    \draw[very thin,dashed] (-.25,.25) -- (.75,.25) -- (-.25,-.75) -- cycle;
    \draw[thick] (0,1) -- (1,0) -- (0,-1) -- (-1,0)  -- cycle;
    \end{tikzpicture}
    \begin{tikzpicture}
    \fill[fill=gray!50,opacity=.75,very thin,dashed] (0,0) -- (1,0) -- (0,1) -- cycle;
    \fill[fill=gray!50,opacity=.5,very thin,dashed] (0,0) -- (1,0) -- (-1,-1) -- cycle;
    \fill[fill=gray!50,opacity=.25,very thin,dashed] (0,0) -- (0,1) -- (-1,-1) -- cycle;
    \draw[very thin,dashed] (1/3,-1/3) -- (1/3,2/3) -- (-2/3,-1/3) -- cycle;
    \draw[very thin,dashed] (-1/3,1/3) -- (2/3,1/3) -- (-1/3,-2/3) -- cycle;
    \draw[thick] (-1,-1) -- (1,0) -- (0,1)  -- cycle;
  \end{tikzpicture}
    \caption{Maximal $\RR$-$\Delta_2$-free polygons of width $2$, with inscribed $\RR$-unimodular triangles.\label{fig:maximal_width_2}
}
\end{figure}

To discuss examples of maximal $\RR$-$\Delta_2$-free polygons, we need the notion of \emph{locked} facet.
A facet $F$ of a polygon $P$ defines two half-planes, whose boundary is the affine line spanned by the facet $F$.
We say that a point $x$ is \emph{beyond} the facet $F$ of $P$ if it lies in the half-plane defined by $F$ which does not contain the interior $\strint(P)$, while $x$ is \emph{beneath} $F$ if it lies in the same half-plane as $P$.
\begin{definition}
    A facet $F$ of an $\RR$-$\Delta_2$-free polygon $P$ is \emph{locked} if for any $x \in \RR^2$ beyond $F$, an $\RR$-unimodular copy of $\Delta_2$ is contained in the interior of $\conv(P, x)$.
\end{definition}

Clearly all facets of an $\RR$-$\Delta_2$-free polygon $P$ are locked if and only if $P$ is maximal.
Indeed, any point outside of $P$ must be beyond at least one facet, and thus the polygon obtained by adding this point to $P$ would by the definition of locked facet contain an $\RR$-unimodular triangle in its interior.

The following characterisation of locked facets will play an important role in our study of maximal $\RR$-$\Delta_2$-free polygons.

\begin{proposition}\label{prop:char-RR-locked}
	Let $P \subset \RR^2$ be an $\RR$-$\Delta_2$-free polygon.
	Then a facet $F$ of $P$ is locked if and only if $P$ contains an $\RR$-unimodular copy $\Delta$ of $\Delta_2$ such that the face $\mathcal{F} \coloneqq F \cap \Delta$ of $\Delta$ is not empty, lies in the relative interior of $F$ and the face $\ell$ of $\Delta$ that is opposite to $\mathcal{F}$ satisfies $\dim(P\text{\textdiv}\ell)=2$.
\end{proposition}

In the proof of \Cref{prop:char-RR-locked}, we use the following statement which can be straightforwardly checked.
\begin{proposition}\label{prop:intersec-2d-cones}
	Let $\sigma_1, \sigma_2 \subset \RR^2$ be two $2$-dimensional polyhedral cones.
	Then $\dim(\sigma_1 \cap \sigma_2) = 2$ if and only if one of the following conditions are satisfied:
	\begin{enumerate}
	\item
		$\sigma_1 \subset \sigma_2$ or $\sigma_2 \subset \sigma_1$; or
	\item
		there exists a ray $\rho$ of $\sigma_1$ that is contained in the interior of $\sigma_2$ and vice versa.
	\end{enumerate}
\end{proposition}

\begin{proof}[Proof of \Cref{prop:char-RR-locked}]
	Suppose the facet $F = [a, b]$ of $P$ is locked ($a, b \in \RR^2$).
	Let $\eta \in \RR^2$ be an outer facet normal of $F$.
	The points $x_n \coloneqq \frac12(a+b)+\frac1n\eta$ for $n \in \NN$ are beyond $F$.
	Hence for every $n \in \NN$ there exists an $\RR$-unimodular copy $S_n$ of $\Delta_2$ that is contained in the interior of $P_n \coloneqq \conv(P, x_n)$.
	Note $S_n$ is an $\RR$-translation of a $\ZZ$-unimodular copy of $\Delta_2$ that is contained in $P_1 + [0,1]^2$.
	Since $P_1+[0,1]^2$ contains only finitely many $\ZZ$-unimodular copies of $\Delta_2$, by restricting to an appropriate subsequence, we may assume that $S_n$ is an $\RR$-translation of exactly one fixed $\ZZ$-unimodular copy  $\widetilde{\Delta}$ of $\Delta_2$ for all $n$, i.e., $S_n = \widetilde{\Delta} + \delta_n$ for some $\delta_n \in \RR^2$.
	Since there exists a (large enough) natural number $N \in \NN$ such that $\delta_n$ is in the compact set $[-N,N]^2$ for all $n$, the sequence of $\delta_n$'s has a convergent subsequence.
	To simplify notation, we abuse notation and use the same notation $(\delta_n)_{n \in \NN}$, $(S_n)_{n \in \NN}$ for these subsequences.
	Then $\Delta\coloneqq \lim_{n \to \infty} S_n$ is contained in $P$ and intersects the facet $F$ in a non-empty face $\mathcal{F} \coloneqq F \cap \Delta$ of $\Delta$.

	Let $\ell$ be the face of $\Delta$ opposite to $\mathcal{F}$.
	Since an $\RR$-translation of $\Delta$ is in the interior of $P_n$ we have $\dim(P_n \text{\textdiv} \Delta) = 2$ (see \Cref{lem:int-full-dim}).
	We claim $\dim(P\text{\textdiv}\ell)=2$.
	Assume towards a contradiction $\dim(P\text{\textdiv}\ell)\le 1$.
	This is only possible if $\ell$ is an edge, say $[c_1, c_2]$ for $c_1, c_2 \in \RR^2$ (as otherwise $P \text{\textdiv}\ell$ is just a translation of $P$, and thus full-dimensional).
	By \Cref{lem_translations}, $P\text{\textdiv}[c_1,c_2] = (P-c_1) \cap (P-c_2)$.
	For sufficiently small $\varepsilon >0$, we have $(c_i + \varepsilon B^2) \cap P = (c_i + \varepsilon B^2) \cap P_n$ for $i=1,2$ and $n \in \NN$.
	Thus
	\[
		(P_n \text{\textdiv} \Delta) \cap \varepsilon B^2 = \bigcap_{v \in V(\Delta)} (P_n-v) \cap \varepsilon B^2 \subset \bigcap_{i\in\{1,2\}}(P-c_i) \cap \varepsilon B^2 = (P\text{\textdiv}\ell)\cap \varepsilon B^2\text{,}
	\]
	where $V(\Delta)$ denotes the set of vertices of $\Delta$.
	Hence $\dim(P_n \text{\textdiv}\Delta) \le 1$, a contradiction.
	
	If $\mathcal{F}$ is contained in the relative interior of $F$, the implication ``$\Rightarrow$'' follows.
	Suppose otherwise, i.e. $\mathcal{F}$ intersects $F$ in an endpoint.
	Note it is not possible that $\mathcal{F}$ intersects $F$ in both endpoints (since the end points $a, b$ of $F$ will remain vertices of $P_n$ for sufficiently large $n \in \NN$).
	Suppose $a \in \mathcal{F}$.
	In particular, $a$ is a vertex of $\Delta$.
	Let $a_1 \coloneqq a$ and $a_2$ be the two vertices of $\Delta$ that are not contained in the relative interior of $F$.
	Call $\Delta$'s third vertex $c$.
	There exists a \emph{fixed} $\varepsilon>0$ such that $(P_n-a_i) \cap \varepsilon B^2 = \sigma_{a_i}^{(n)} \cap \varepsilon B^2$ and $(P-a_i)\cap \varepsilon B^2 = \sigma_{a_i} \cap \varepsilon B^2$ for some polyhedral cones $\sigma_{a_i}^{(n)}, \sigma_{a_i} \subset \RR^2$ ($i=1,2$).
	However, around $c$ we might need to choose $\varepsilon_n >0$ depending on $n \in \NN$ such that $(P_n - c) \cap \varepsilon_n B^2 = \sigma_c \cap \varepsilon_n B^2$ for some polyhedral cone $\sigma_c \subset \RR^2$.
	In addition, $(P-c)\cap \varepsilon B^2 = \sigma_c'$ for a possibly different cone $\sigma_c' \subset \RR^2$ and we might need to decrease the $\varepsilon$ from above.
	Note if $\varepsilon_n$ needs to be adjusted with $n\in \NN$, then $\varepsilon_n \to 0$ as $n \to \infty$.
	Furthermore, $\sigma_{a_2}^{(n)} = \sigma_{a_2}$ doesn't change for sufficiently small $\varepsilon>0$.
	Indeed, only one ray of $\sigma_a^{(n)}$ changes, namely the ray $\rho_a^{(n)} \coloneqq \RR_{\ge0} (\frac12 (b-a) + \frac1n\eta)$.
	Since $\dim(P_n \text{\textdiv}\Delta) = 2$ for all $n \in \NN$ (where $(P_n \text{\textdiv} \Delta) \cap \varepsilon_n B^2 = \sigma_a^{(n)} \cap \sigma_{a_2} \cap \sigma_c \cap \varepsilon_n B^2$ for all sufficiently large $n \in \NN$) and $\dim(P\text{\textdiv}\Delta)\le 1$, it follows with \Cref{prop:intersec-2d-cones} that $\rho_a^{(n)}$ lies in the relative interior of $\sigma_{a_2} \cap \sigma_c$.
	Since $\dim(P\text{\textdiv}\Delta)\le 1$, one ray of $\sigma_{a_2} \cap \sigma_c'$ is $\rho_a = \RR_{\ge0}\frac12(b-a)$.
	Hence, $\rho_a \cap \varepsilon B^2 \subset P \text{\textdiv}\Delta$ for some sufficiently small $\varepsilon>0$, and thus we can move $\Delta$ within $P$ such that it intersects the facet $F$ in its relative interior.

	For the reverse implication suppose $x$ is strictly beyond the facet $F$ of $P$.
	Then the vertices of $\mathcal{F}$ lie in the interior of $\conv(P, x)$, and thus for sufficiently small $\varepsilon >0$, it follows that $(P\text{\textdiv}\ell) \cap \varepsilon B^2 \subset \conv(P,x) \text{\textdiv}\Delta$.
	That is, $\dim(\conv(P,x)\text{\textdiv}\Delta)=2$, i.e., an $\RR$-translation of $\Delta$ is contained in the interior of $\conv(P, x)$ by \Cref{lem:int-full-dim}.
\end{proof}

Certainly, the previous proof heavily relies on properties of $2$-dimensional geometry.
It would be interesting to know a characterisation of locked facets similar to \Cref{prop:char-RR-locked} in higher dimensions:
\begin{question}
	Find a characterisation of $\RR$-locked facets in $d$ dimensions for $d\ge 3$ similar to the one in \Cref{prop:char-RR-locked}.
\end{question}

The following statements will be useful to prove that facets of our candidate maximal polygons are locked and follow by \Cref{prop:char-RR-locked}.

\begin{corollary}\label{lem:locked_skew}
  Let $P$ be an $\RR$-$\Delta_2$-free polygon containing $\Delta = \conv(v_1, v_2, v_3)$, an $\RR$-unimodular copy of $\Delta_2$.
  Suppose that each vertex $v_i$ lies in the relative interior of a distinct facet $F_i$ of $P$.
  If the lines spanned by $F_2$ and $F_3$ meet at a point, then $F_1$ is locked.
  See \Cref{fig:locked_skew}.
\end{corollary}

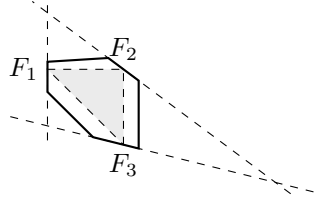
\begin{figure}[!ht]
    \centering
        \begin{tikzpicture}
        \fill[fill=gray!50,opacity=.3,very thin,dashed] (1,0) -- (1,1) -- (0,1) -- cycle;
        \draw[thin, dashed] (1,0) -- (1,1) -- (0,1) -- cycle;
        \draw[thick] (0,0.7)--(0,1.1)--(0.8,1.15)--(1.2,0.85)--(1.2,-0.05)--(0.6,0.1)--cycle;
        \draw (0,1) node[left] {$F_1$};
        \draw (1,1) node[above] {$F_2$};
        \draw (1,0) node[below] {$F_3$};
        \draw[thin, dashed] (0,1.85) -- (0, 0);
        \draw[thin, dashed] (3.2,-.65) -- (-.2, 1.9);
        \draw[thin, dashed] (3.5,-.625) -- (-.5, .375);
    \end{tikzpicture}
    \caption{Facets $F_1$, $F_2$, $F_3$ are locked by \Cref{lem:locked_skew}.\label{fig:locked_skew}
}
\end{figure}

\begin{proof}
	By the assumption, $\Delta$ is an $\RR$-unimodular copy of $\Delta_2$ that is contained in $P$ such that $F_1 \cap \Delta$ is a vertex of $\Delta$ that lies in the relative interior of $F_1$.
	Since the lines spanned by $F_2, F_3$ intersect in a point beneath $F_1$, it follows that $\dim(P\text{\textdiv}\conv(v_2, v_3)) = 2$.
	The statement follows by \Cref{prop:char-RR-locked}.
\end{proof}

\begin{remark}
	Note that in \Cref{lem:locked_skew}, the point where the two lines spanned by the facets $F_2$ and $F_3$ meet is beneath the facet $F_1$.
	It cannot be beyond the facet $F_1$, as otherwise $\dim(P\text{\textdiv}\Delta)=2$, and thus an $\RR$-translation of $\Delta$ is contained in the interior of $P$.
\end{remark}

\begin{corollary}\label{lem:locked_parallel}
    If $P$ is an $\RR$-$\Delta_2$-free polygon containing $\Delta=\conv(v_1, v_2, v_3)$, an $\RR$-unimodular copy of $\Delta_2$, with $v_1, v_2$ lying in the relative interior of two distinct parallel facets $F_1, F_2$ of $P$, while $v_3 \in \strint(P)$, then the two parallel facets are locked.
\end{corollary}
\begin{proof}
	We show that $F_1$ is locked (a similar argument works for $F_2$).
	Note $\Delta$ is an $\RR$-unimodular copy of $\Delta_2$ contained in $P$ such that $F_1 \cap \Delta$ is a vertex of $\Delta$ that is contained in the relative interior of $F_1$.
	Since $v_3 \in \strint(P)$, we have $\dim(P\text{\textdiv}\conv(v_2,v_3))=2$.
	The statement follows by \Cref{prop:char-RR-locked}.
\end{proof}

We now construct a family of maximal $\RR$-$\Delta_2$-free bodies consisting of all parallelograms circumscribed around a unit square $[0,1]^2$.
An $\RR$-unimodular copy of $[0,1]^2$ we call an \novel{$\RR$-unimodular parallelogram} (or simply \novel{unimodular parallelogram} if it is clear from the context that we consider $\RR$-unimodular copies).

\begin{proposition}
  If $P$ is a parallelogram such that the relative interior of each of its facets contains one vertex of a fixed unimodular parallelogram, see \Cref{fig_circumscribed_parallelograms}, then $P$ is a maximal $\RR$-$\Delta_2$-free convex set.
\end{proposition}

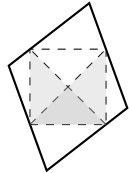
\begin{figure}[!ht]
    \centering
    \begin{tikzpicture}
      \draw[thin, dashed] (0,0)--(1,0)--(1,1)--(0,1)--cycle;
      \draw[thin, dashed] (0,0)--(1,1);
      \draw[thin, dashed] (1,0)--(0,1);
      \draw[thick] (-.28, .78) -- (.22, -.61)-- (1.28,.22)-- (.78,1.61)--cycle;
      \fill[fill=gray!50,opacity=.3,very thin,dashed] (0,0) -- (1,0) -- (1,1) -- cycle;
      \fill[fill=gray!50,opacity=.3,very thin,dashed] (0,0) -- (1,0) -- (0,1) -- cycle;
    \end{tikzpicture}
    \caption{A parallelogram circumscribed around a unit square.
    	Any such parallelogram is maximal $\RR$-$\Delta_2$-free, since each facet is locked.\label{fig_circumscribed_parallelograms}
}
\end{figure}

\begin{proof}
    We first show $P$ is $\RR$-$\Delta_2$-free.
    Up to an $\RR$-unimodular transformation, we can assume that the unimodular parallelogram is the standard unit square $Q=[0,1]^2$.
    Note $P$ is contained in the union of the vertical and horizontal strip containing $Q$, and two vertices of $P$ lie in each strip.
    Thus the parallel lines supporting facets of $P$ through $e_1$ and $e_2$ have positive slope, while those through $\mathbf{0}$ and $e_1 + e_2$ have negative slope.
    We call the parallel lines through $\mathbf{0}$ and $e_1 + e_2$ respectively $L$ and $L + e_1 + e_2$.

    Let $v=(v_1,v_2)\in \RR^2$ be a normal vector to $L$ (which is not necessarily rational).
    Since $L$ has negative slope, we can assume $v \in \RR^2_{>0}$.
    The lattice segment $\conv(\mathbf{0}, e_1 + e_2)$  has endpoints on both lines, and thus the width of $P$ with respect to $v$ is $\langle v, e_1+e_2 \rangle = v_1+v_2$.
    Any segment with width in direction $v$ larger than this can't be between the two lines, and hence doesn't lie in $P$.
 
 	Our goal is to show that every segment that is contained in $P$ and that is an $\RR$-translation of a primitive lattice segment is contained in $Q$.
	Hence, the only $\RR$-unimodular copies of $\Delta_2$ that are contained in $P$ are an $\RR$-translation of one of the unimodular triangles contained in $Q$.  
    For any primitive lattice segment $\ell$ there exists a $\ZZ$-translation such that $\ell \subset \RR^2_{\ge0}$ or $\ell \subset \RR_{\ge0} \times \RR_{\le0}$.
    Note a segment parallel to the primitive lattice segment  $\conv(\mathbf{0}, m e_1+ne_2)$, for integers $m, n$ with $\gcd(m,n)=1$, has width in direction $v$ equal to $mv_1+nv_2$.
    If $m,n \geq 1$ and $m+n>2$, the segment cannot be contained in $P$.
    We have thus excluded all primitive lattice segments that are contained in $\RR_{\ge0}^2$ (up to a $\ZZ$-translation) and that can't be moved in the square $Q$ via an $\RR$-translation.
    An analogous argument using the other pair of parallel facets disqualifies all primitive lattice segments that are contained in $\RR_{\ge0} \times \RR_{\le0}$ (up to a $\ZZ$-translation) and that can't be moved in the square $Q$ via an $\RR$-translation.
    
    The only segments which are contained in $P$ are therefore $\RR$-translations of the lattice segments contained in $Q$.
    It follows that the only $\RR$-unimodular copies of $\Delta_2$ contained in $P$ are those contained in $Q$.
    Since no translation of these is contained in the relative interior of $P$ (see \Cref{lem:int-full-dim}), $P$ is $\RR$-$\Delta_2$-free.
    
    We now wish to show that $P$ is maximal with this property.
    To do so, we apply \Cref{lem:locked_skew} for each facet of $P$: the square $Q$ contains four unimodular triangles, each of which has one vertex in the interior of three of the four facets of $P$.
    Consider $\conv(\mathbf{0}, e_1, e_2)$: this triangle shows that the facet containing $e_1$ is locked, because the facets of $P$ containing $\mathbf{0}$ and $e_2$ are consecutive and thus meet at a point beneath the facet containing $e_1$.
    Thus the conditions of \Cref{lem:locked_skew} are satisfied and the facet is locked.
    The same triangle shows that the facet of $P$ containing $e_2$ is locked, while applying the same arguments to the triangle $\conv(\mathbf{0}, e_1, e_1+e_2)$ shows that the remaining two facets are also locked.
    As we have remarked earlier, when all facets of an $\RR$-$\Delta_2$-free polygon are locked, the polygon is maximal.
\end{proof}

We thus have a family of quadrilaterals with plenty of structure that are maximal $\RR$-$\Delta_2$-free.
However, there are many examples of maximal $\RR$-$\Delta_2$-free quadrilaterals which we do not know how to characterise.
On the left of \Cref{fig:skew_maximal} we see one such example.
That this quadrilateral $Q$ is  $\RR$-$\Delta_2$-free can be seen by checking that any unimodular triangle which fits within a box circumscribed to $Q$ cannot be translated into $Q$ (see the git-repository mentioned above for a Magma script that automates this verification).
To prove that it is inclusion-maximal, we can again apply \Cref{lem:locked_skew} to each facet: in \Cref{fig:skew_maximal} the unimodular triangles corresponding to each locked facet are represented. 
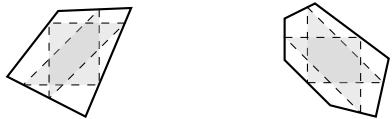
\begin{figure}[!ht]
    \centering
    \begin{tikzpicture}
        \fill[fill=gray!50,opacity=.3,very thin,dashed] (0,0) -- (1,0) -- (1,1) -- cycle;
        \draw[thin, dashed] (0,0) -- (1,0) -- (1,1) -- cycle;
        \fill[fill=gray!50,opacity=.3,very thin,dashed] (.34,-.18) -- (.34,.82) -- (1.34,.82) -- cycle;
        \draw[thin, dashed] (.34,-.18) -- (.34,.82) -- (1.34,.82) -- cycle;
        \draw[thick] (-.21,.11) -- (.46,.98)-- (1.42,1.02) --(.82,-.42) --cycle;
    \end{tikzpicture}
    \hspace{50pt}
    \begin{tikzpicture}
    \fill[fill=gray!50,opacity=.3,very thin,dashed] (.3,.4) -- (1.3,.4) -- (.3,1.4) -- cycle;
    \draw[thin, dashed] (.3,.4) -- (1.3,.4) -- (.3,1.4) -- cycle;
    \fill[fill=gray!50,opacity=.3,very thin,dashed] (1,0) -- (1,1) -- (0,1) -- cycle;
    \draw[thin, dashed] (1,0) -- (1,1) -- (0,1) -- cycle;
    \draw[thick] (0,0.7)--(0,1.25)--(0.4,1.45)--(1.37,0.72)--(1.2,-0.05)--(0.6,0.1)--cycle;
    \end{tikzpicture}
    \caption{The quadrilateral on the left with vertices $ (-0.21,0.11)$, $(0.46,0.98)$, $(1.42,1.02)$, $(0.82,-0.42)$ is maximal $\RR$-$\Delta_2$-free, as certified by the inscribed triangles.
    \\
    On the right is a maximal $\RR$-$\Delta_2$-free hexagon with vertices $(0,0.7)$, $(0,1.25)$, $(0.4,1.45)$, $(1.37,0.72)$, $(1.2,-0.05)$, $(0.6,0.1) $.
    The inscribed triangles certify that each facet is locked, relying on the fact that no two facets are parallel.\label{fig:skew_maximal}
}
\end{figure}
There are also examples of maximal $\RR$-$\Delta_2$-free polygons with more facets.
On the right side of \Cref{fig:skew_maximal} is one such example.
Again the fact that it is $\RR$-$\Delta_2$-free can be checked by testing all unimodular triangles which fit into an appropriate rectangle (see also the git-repository from above), while its maximality follows from applying \Cref{lem:locked_skew} to each facet, with respect to the triangles inscribed to the hexagon in \Cref{fig:skew_maximal}.

In \Cref{sec_incl_max_dim_2} we showed that maximal $\ZZ$-$\Delta_2$-free polygons have at most $4$ facets. It is natural to ask if there is also an upper bound on the facets of maximal $\RR$-$\Delta_2$-free polygons. In fact, Lovasz proved that maximal hollow convex bodies in any dimension $d$ have at most $2^d$ facets (see~\cite{Averkov} for a complete proof). This suggests the following questions.

\begin{question}
    Is there an upper bound on the number of facets of maximal $\RR$-$\Delta_d$-free polytopes in $\RR^d$?
    For maximal $\ZZ$-$\Delta_d$-free polytopes?
\end{question}

\newcommand{\etalchar}[1]{$^{#1}$}
\providecommand{\bysame}{\leavevmode\hbox to3em{\hrulefill}\thinspace}
\providecommand{\href}[2]{#2}

\end{document}